\documentclass[11pt,reqno,hidelinks]{amsart}
\usepackage{amsthm, amsmath,amsfonts,amssymb,euscript,hyperref,graphics,color}
\usepackage{graphicx}
\usepackage[square, comma, sort&compress, numbers]{natbib}
\usepackage{subfigure}
\usepackage{comment}
\usepackage{import}
\usepackage{tikz}
\usepackage{latexsym}

\usepackage{ulem}
\usepackage{todonotes}

\makeatother
\newtheorem{theorem}{Theorem}[section]
\newtheorem{lemma}[theorem]{Lemma}
\newtheorem{corollary}[theorem]{Corollary}
\newtheorem{proposition}[theorem]{Proposition}

\theoremstyle{definition}

\theoremstyle{remark}
\newtheorem{remark}[theorem]{Remark}

\newcommand \del \partial
\newcommand \be {\begin{equation}}
\newcommand \ee {\end{equation}}
\numberwithin{equation}{section}

\def\ub{\underline{u}}
\def\Lb{\underline{L}}
\def\Cb{\underline{C}}
\def\Eb{\underline{E}}

\def\chib{\underline{\chi}}
\def\hb{\underline{h}}
\def\Hb{\underline{H}}
\def\tr{\text{tr}}

\def\gslash{\mbox{$g \mkern -8mu /$ \!}}

\def\Dslash{\mbox{$D \mkern -13mu /$ \!}}
\def\nablaslash{\mbox{$\nabla \mkern -13mu /$ \!}}
\def\laplacianslash{\mbox{$\triangle \mkern -13mu /$ \!}}

\def\p{\partial}

\newcommand{\di}{\mathrm{d}} 
\newcommand{\D}{\mathcal{D}}
\newcommand{\Rp}{\mathcal{R}^\prime}

\newcommand \bei  {\begin{itemize}}
\newcommand \eei {\end{itemize}}
\allowdisplaybreaks[4]
\begin{document}
\title[Global smooth solution to RME with large data]{Global existence of smooth solution to relativistic membrane equation with large data}

\author[J. Wang]{Jinhua Wang} \email{wangjinhua@xmu.edu.cn}
\address{School of Mathematical Sciences, Xiamen University, Xiamen 361005, China.}

\author[C. Wei]{Changhua Wei} \email{changhuawei1986@gmail.com}
\address{Corresponding author: Department of Mathematics, Zhejiang Sci-Tech University, Hangzhou, 310018, China.}

\date{}


\begin{abstract}
This paper is concerned with the Cauchy problem for the relativistic membrane equation (RME) embedded in $\mathbb R^{1+(1+n)}$ with $n=2, \, 3$. We show that the RME with a class of large (in energy norm) initial data admits a global, smooth solution. The initial data are given by the short pulse type, which is introduced by Christodoulou in his work on the formation of black holes \cite{Christodoulou}. Due to the quasilinear feature of RME, we construct two multipliers adapted to the geometry of membrane and present an efficient way for proving the global existence of smooth solution to the geometric wave equation with double null structure. We also derive the asymptotic geometry
of the future null infinity and find out a nonlinear (expanding) effect at infinity.
\end{abstract}

\maketitle
\tableofcontents



\section{Introduction}
\label{Section1}

In this paper, we reconsider the well-known $n$-dimensional relativistic membrane equation (RME), which in  graph description reads
\begin{equation}\label{1.1}
\frac{\partial}{\partial t}\frac{\p_t\phi}{\sqrt{1-(\p_t\phi)^{2}+|\nabla_{x}\phi|^2}}-\sum_{i=1}^n\frac{\partial}{\partial x^{i}}\frac{\p_i \phi}{\sqrt{1-(\p_t\phi)^2+|\nabla_{x}\phi|^2}} = 0,
\end{equation}
where $\p_i\phi=\partial\phi/\partial x^i$, $\p_t\phi=\partial\phi/\partial t$, and $\phi=\phi(t,x^1,\cdots,x^n)$ denotes the graph of the hypersurface embedded in the Minkowski spacetime $\mathbb R^{1+(1+n)}$.

RME is the hyperbolic counterpart of the minimal surface equation in the Euclidean space, which plays an important role in geometry and physics. The dynamics of RME are closely related to the two dimensional fluid dynamics \cite{Hoppe2,Hoppe3}. In particular, the graph description of RME \eqref{1.1} can also describe the ``wave character'' of irrotational relativistic Chaplygin gas, see \cite{Christodoulou1}.

Due to its importance in geometry and physics, there have been a host of recent activities in relativistic membrane in the last decade.
The general formulation of RME is a quasilinear hyperbolic system (see \cite{Christodoulou-2000, WWong-11}), as long as the pullback metric is Lorentzian and the local well-posedness for smooth initial data holds (see \cite{Aurilia-christdoulou, Krieger-Lindblad}). Hoppe \cite{Hoppe4, Hoppe5} derived the RME in graph description \eqref{1.1}, which possesses good structures: \eqref{1.1} satisfies the double null condition \cite{Klainerman-null}, and it could be reformulated as an equation in a divergence form \cite{Lindblad}. Based on these two facts,  Lindblad \cite{Lindblad} proved the global existence of smooth solutions to \eqref{1.1} (for $n\geq 2$) with small, compactly supported data. Similar idea was utilized by Allen, Andersson and Isenberg \cite{Allen} to prove the small data-global existence for timelike extremal submanifold with codimension larger than one. Starting from a geometric perspective, Brendle \cite{Brendle} applied the scheme of \cite{Christodoulou-K-93}, and established a result on stability of a flat hyperplane for the RME with spatial dimension $n\geq3$. Later, Krieger and Lindblad \cite{Krieger-Lindblad} investigated the stability problem of the Catenoid minimal surface for \eqref{1.1} subject to certain generic radial perturbations. The key observation therein is that if one perturbs outside the ``collar region'', the solution exists as long as the resulting deformation stays away from the collar. Donninger, Krieger, Szeftel and Wong \cite{D-K-S-Wong} further studied general nonlinear (in)stability of the catenoid and showed that the linear instability of catenoid is the only obstruction to the globally nonlinear stability. Namely, restricted to the perturbation which is a codimension one Lipschitz manifold transverse to the unstable mode, the solution exists globally in time and converges asymptotically to the catenoid. On the other hand, for the light-cone gauge description of RME, Hoppe \cite{Hoppe4, Hoppe5} exploited its algebraic and geometric properties and found the classical solution. It is remarkable that in \cite{Hoppe2, Hoppe3}, variable transformations are introduced to relate the relativistic membranes with two dimensional fluid dynamics. This further inspired the work of Kong, Liu and Wang \cite{K-L-W}, in which they proved the global stability of the two dimensional isentropic Chaplygin gas without vorticity via the potential theory. Lei and Wei \cite{Lei-Wei} also investigated the relationship between relativistic membrane equation and relativistic Chaplygin gas and proved the global existence of the radial solution to 3D non-isentropic relativistic Chaplygin gas.

Most of the results about RME are concerned with the Cauchy problem with small initial data, while little is known for the large initial data case, until the breakthrough of Christodoulou \cite{Christodoulou}, where he in fact achieved the global existence of the solution to the Einstein vacuum equations with large data. This work has been generalized by Klainerman and Rodnianski \cite{Klainerman-Rodnianski}, by introducing a key idea of relaxed propagation estimates. Inspired by these two works, the first author and Yu \cite{Wang-Yu,Wang-Yu1} considered the characteristic problem for semilinear wave equations satisfying the standard null condition, and showed the global existence of solutions for a class of large initial data, which are of {\it short pulse} type originated from \cite{Christodoulou}. Instead, Miao, Pei and Yu \cite{Miao-Yu} studied the Cauchy problem with short pulse data prescribed on a thin domain of the initial Cauchy hypersurface. In their proof, one could take great advantage of the fully geometric structures of the Minkowski spacetime from the semilinearity of wave equations considered. We here also mention that Yang \cite{Yang-13} obtained a global existence result for semilinear wave equations with large Cauchy data, which is due to slower decay of the initial data at spatial infinity. All these works heavily rely on the {\it null condition} which was first observed by Klainerman \cite{Klainerman-null} and Christodoulou \cite{Christodoulou-null} when considering semilinear wave equations in $\mathbb{R}^{1+3}$ with small data.

Turning to quasilinear wave equations, if there is no null structure, shocks will form in finite time generally. It is well known that Christodoulou \cite{Christodoulou1}, followed by Speck \cite{Speck-shock} and Speck-Holzegel-Luk-Wong \cite{Speck-Hol-Luk-Wong}, proposed a geometric mechanism of shock formation for quasilinear wave equations by studying the corresponding Lorentzian metric. They constructed optical functions associated to the metric and showed that a shock was driven by the intersection of null (characteristic) hypersurfaces. Whereas, we should also mention that Miao-Yu \cite{Miao-Yu1} study a quasilinear wave equation satisfying the null condition in a large data setting, and find out similar mechanism of shock formation. When it comes to the quasilinear wave equation with null structure, for which the global solution is expected, however, it would be possible to find a balance between the curved geometry for the quasilinear wave equations and the flat geometry for the semilinear ones, and present an effective approach for large data problem of the geometrically quasilinear wave equations with null structure. This serves as one of the motivations for this work. On the other aspect, Majda's conjecture \cite{Majda} suggests that the symmetric hyperbolic system with totally linearly degenerate characteristics admits a globally classical solution unless the solution itself blows up in finite time. Our result can also solve this conjecture \cite{Majda} in the cases of irrotational and isentropic perfect fluids, according to the relationship between the relativistic membrane and the Chaplygin gas (see \cite{Matt}).

\subsection{Main results} Let $r=\sqrt{(x^{1})^2+\cdots+(x^n)^2}$ and $\theta$ denote the usual radial and angular coordinates on $\mathbb R^n$. We use $\nabla$ to denote the spacial derivative $\p_i, \, i=1, \cdots, n$. In Minkowski spacetime, we define the following two optical functions $u=\frac{1}{2}(t-r), \, \underline{u}=\frac{1}{2}(t+r)$. Let $\delta$ be a small positive parameter.

Concerning the initial value problem of \eqref{1.1},  our set up for data is similar to that in \cite{Miao-Yu}. We take $\{t=1\}$ as the initial hyperplane and divide it into three parts (see Figure \ref{fig:foliation})
$$
\{t=1\}=B_{1-2\delta}\cup(B_{1}-B_{1-2\delta})\cup(\mathbb R^n-B_{1}),
$$
where
$B_{r}$ denotes the ball centred at the origin with radius $r$.

At first, we let $\phi_{0}=\phi_{1}=0$ on $B_{1-2\delta}\cup (\mathbb R^n-B_{1})$.

Next, we prescribe our initial data $(\phi_{0},\phi_{1})$ on $B_{1}-B_{1-2\delta}$ as follows: $\phi_{0}(r,\theta)$ and $\phi_{1}(r,\theta)$ are smooth functions supported in $r\in(1-2\delta,1)$ and they satisfy the constraints
\begin{subequations}
\begin{align}
\|\nablaslash^{l}(\delta\partial_r)^{k}(\phi_{1}+\partial_{r}\phi_{0})\|_{L^{\infty}(B_{1}-B_{1-2\delta})} & \leq C_{k,l} \delta, \label{1.2}\\
\|\nablaslash^{l}(\delta\partial_r)^{k}\phi_{0}\|_{L^{\infty}(B_{1}-B_{1-2\delta})}+\delta\|\nablaslash^{l}(\delta\partial_r)^{k-1}\phi_{1}\|_{L^{\infty}(B_{1}-B_{1-2\delta})} & \leq \hat{C}_{k,l} \delta, \label{1.3}\\
\|\nablaslash^{l}(\partial_{t}+\partial_{r})^{k}(\delta\partial_r)^{m}\phi\|_{L^{\infty}(B_{1}-B_{1-2\delta})} & \leq C_{k,l,m} \delta, \label{1.3-associated}
\end{align}
\end{subequations}
where $\nablaslash$ denotes the covariant derivative (associated to the induced metric from the flat  Lorentzian metric) on the sphere with constant $t$ and $r$, and $k,\,l,\,m$ are non-negative integers. For simplicity, we use $\texttt{C}_{k,l,m}$ denotes a sequence of positive constants depending on $m,k,l$ appeared in \eqref{1.2}-\eqref{1.3-associated}.
\begin{remark}
We remark that \eqref{1.2} and \eqref{1.3} are conditions on the freely disposable data $(\phi_0, \phi_1)$, while \eqref{1.3-associated} is a condition on the derived data as it involves $\partial_t^k \phi$ for $k \geq 2$ which is obtained from $(\phi_0, \phi_1)$ by solving the equation and its $k-2$th differential consequence, obtained by applying $\partial_t^{k-2}$, along the initial hypersurface. Furthermore, the condition \eqref{1.3} can be replaced by $$\|\nablaslash^{l}(\delta\partial_r)^{k}\phi_{0}\|_{L^{\infty}(B_{1}-B_{1-2\delta})} \lesssim \delta,$$ because the bound on the 1st term in the left hand side of \eqref{1.3} together with \eqref{1.2} implies the bound on the 2nd term in the left hand side of \eqref{1.3}.
\end{remark}
 In the following, we always use $f\lesssim g$ to denote $f\leq C g$ for some universal and positive constant $C$. And $f \sim g$ means $f \lesssim g$ and $g \lesssim f$. We will adopt the following notation for any function $f$: We write $f=O(r^b)$ for some $b\in \mathbb{R}$, if for any integer $k\geq0$, we can bound $|r^k\p^k f|\lesssim r^b$; and $f=o(r^b)$ if  $\lim_{r \rightarrow +\infty} \frac{|r^k\p^k f|}{r^b}= 0$.

\begin{remark}
We have several remarks on the initial data and energies:
\begin{itemize}
\item (\emph{Short pulse data})
The above initial data sets are called {\it short pulse} data and the existence of this kind of Cauchy data is shown in Appendix \ref{existence of data}. During the propagation of our data, the hyperbolicity of the RME \eqref{1.1} which is equivalent to $Q:=\eta^{\alpha\beta}\partial_{\alpha}\phi\partial_{\beta}\phi>-1$ holds, since $Q=-L\phi \Lb\phi+|\nablaslash \phi|^2$ and we always have $|\Lb\phi| \lesssim 1$, $|L\phi|\lesssim \delta$, $|\nablaslash\phi|\lesssim \delta^{\frac{3}{4}}$ (refer to Lemma \ref{lemma-apriori-estimate} for these estimates and the relevant notations here are given later in Section \ref{sec-comm-proof}) as \eqref{1.1} evolves.
\item (\emph{Geometric/ Physical meaning of the constraint \eqref{1.3-associated}}). Similar to Miao-Pei-Yu \cite{Miao-Yu}, the constraint \eqref{1.3-associated} is imposed to ensure that the incoming energy is as small as possible and the total energy will be propagated along the outgoing direction.
\item  (\emph{Largeness of the initial higher order energy}). The initial first order energy should be small to guarantee the hyperbolicity of \eqref{1.1}. In contrast, the higher order initial energies ($k>1$) could be arbitrarily large if we choose $\delta$ small enough:
\begin{align*}
E^{2}_{0} & =\int_{\mathbb R^n}(|\nabla_{x}\phi_0|^{2}+|\phi_{1}|^2)dx\sim \delta, \\
E^{2}_{k} & =\int_{\mathbb R^n}(|\nabla_{x}^{k+1}\phi_{0}|^2+|\nabla_{x}^{k}\phi_{1}|^2)dx\sim \delta^{-2k+1}.
\end{align*}
\end{itemize}
\end{remark}

Our main result can be summarized as follows:
\begin{theorem}
There exists a positive constant $\delta_{0}$ depending on a finite subset of the $\texttt{C}_{m,k,l}$ arising from the constraints \eqref{1.2}-\eqref{1.3-associated}, such that
the Cauchy problem \eqref{1.1} with the short pulse initial data admits a unique and globally smooth solution in $[1,+\infty)\times \mathbb R^n$ with $n=2, \, 3$ providing that $0
<\delta<\delta_{0}$. Moreover, we show an expanding effect along the nearby bundles of flow lines of an incoming  null (and almost geodesic) vector field near the null infinity (which holds true in both small and large data settings), and this nonlinear effect is due to the curved geometry of membrane.
\end{theorem}
\begin{remark}\label{rk-diverge}
In general, a geometric mechanism of shock formation for quasilinear wave equations violating the null condition is driven by the focusing of null hypersurfaces (or characteristic lines), see for instance \cite{Christodoulou1, Hol-Speck-K-Wong, Speck-shock, Speck-Hol-Luk-Wong}. Analogous mechanism is found in \cite{Miao-Yu1} for a quasilinear wave equation with null condition and large data. Another geometric shock formation is also demonstrated in \cite{K-weak-null-18} by showing that quasilinear equations satisfying the weak null condition (introduced by Lindblad and Rodnianski \cite{Lind-Rod}) and admitting global existence of solutions (with small data) can exhibit ``shock formation at infinity''. As a complement, for the RME which satisfies the double null condition, our result reveals the asymptotic geometry of the future null infinity (of the membrane) and capture the nonlinear effect at infinity.
\end{remark}

\begin{figure}
\centering
\includegraphics[width=6.0in]{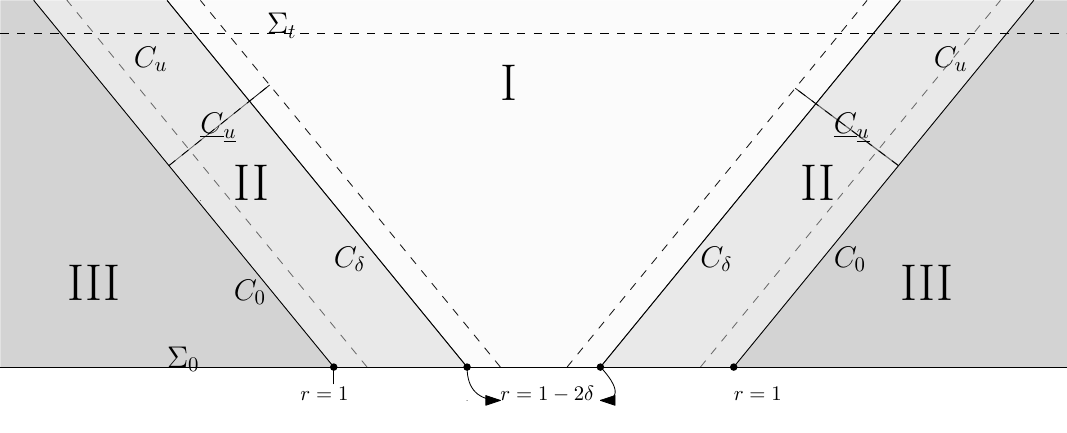}
\caption{The foliation}
 \label{fig:foliation}
\end{figure}

\subsection{Comments on the proof}\label{sec-comm-proof}

We recall the standard null foliation (See Figure \ref{fig:foliation}) of  the Minkowski spacetime ($\mathbb R^{1+n}, \eta$) where $\eta$ is the flat Lorentzian metric. Reminding the two optical functions $u, \, \ub$, we use $C_{u}$ to denote the level surface of $u$, and $\underline{C}_{\underline{u}}$ the level surface of $\underline{u}$. $S_{u,\underline{u}}$ denotes the intersection of $C_{u}$ and $\underline{C}_{\underline{u}}$, which is an $(n-1)$-sphere. The double null foliation of the Minkowski spacetime $(\mathbb R^{1+n}, \eta)$ is referred to as  the outgoing null foliation $\{C_{u}|u\in\mathbb R\}$ and incoming null foliation $\{\underline{C}_{\underline{u}}|\underline{u}\in\mathbb R\}$ (see the obliquely dotted lines in Figure \ref{fig:foliation}). The two corresponding null vectors are given by
\begin{equation*}
L=\partial_{\underline{u}}=\partial_{t}+\partial_{r},\quad \underline{L}=\partial_{u}=\partial_{t}-\partial_{r}.
\end{equation*}
The future Cauchy development of the initial hypersurface $J^+(\Sigma_1)$ is divided into three regions: {\rm I, II, III}, by the two null cones $C_0, C_\delta$, which are labelled as the obliquely solid lines in Figure  \ref{fig:foliation}.

Due to the finite speed of propagation of wave equations, the RME has a trivial solution in the region {\rm III}.
Our proof in the regions {\rm I} and {\rm II} composes of four parts:
\bei
\item {\bf Step 1.} The global existence of the solution $\phi$ to the RME in the short plus region {\rm II}. In this region, the high order energy of $\phi$ would be large. Its proof depends on the geometry of membrane and the null structure of RME. This is the main body of this paper.
\item {\bf Step 2.} The smallness for $\phi$ and its higher order derivatives on $C_\delta$. This is done by integrating along the integral curves of $L$ on $C_\delta$ and making use of the fact that $\phi$ vanishes on the initial sphere $C_{\delta}\cap\Sigma_{1}$.
\item {\bf Step 3.} The global existence in the small data region {\rm I}.
\item {\bf Step 4.} The expanding effect near the null infinity of membrane.
\eei

\subsubsection{The structures of RME}\label{sec-stru-RME}
The main difference between quasilinear and semilinear wave equations is that they are related to different geometry.
This reflects on the fact that the quasilinear wave equation \eqref{1.1} can be written as a geometric wave equation
\begin{equation}\label{RME-curved-box}
\Box_{g(\partial\phi)} \phi=0,
\end{equation}
where $$g_{\alpha\beta} (\partial\phi) =\eta_{\alpha\beta}+\partial_{\alpha}\phi\partial_{\beta}\phi$$ is the Lorentzian metric depending on the first order derivative of the unknown $\phi$ and $\Box_{g(\partial\phi)}$ is the associated wave operator. That is to say,  ($\mathbb{R}^{1+n}, g(\partial\phi)$) is the geometry adapted to the RME \eqref{RME-curved-box}, while the corresponding geometry for semilinear wave equation is the Minkowski spacetime ($\mathbb{R}^{1+n}, \eta$).  This kind of geometric viewpoint for quasilinear wave equation is initiated in \cite{Christodoulou1} to describe the geometric feature of shocks for the relativistic Euler equations and later in \cite{Miao-Yu1, Speck-Hol-Luk-Wong} for general classes of quasilinear wave equations.

Taking derivatives on \eqref{RME-curved-box} yields the higher order RME, which again is a geometric wave equation with inhomogeneous terms satisfying the double null structure, see Section \ref{subsec-The high order equation and double null condition}.
\begin{remark}
In regard to the hypersurfaces $C_u$, the level sets of $u$ which is defined relative to the Minkowski metric $\eta$, it would be crucial that these hypersurfaces be non-timelike with respect to the metric $g_{\alpha \beta}$.
Because in {\bf Step 3} when we treat the problem in the region {\rm I}, the data on the common boundary $C_\delta$ are induced by the solution in the region {\rm II} $\cup$ {\rm III}. This depends crucially on the property of the region {\rm II} $\cup$ {\rm III} being a domain of dependence, which reduces to the property that $C_\delta$ is non-timelike relative to $g_{\alpha\beta}$. This property will follow by the argument below.

For any vector $X$, there is $g(X,X) = \eta(X,X)+ (X\phi)^2$. Hence if $X$ is non-timelike relative to $\eta$, it is non-timelike relative to $g$ as well. Now a vector tangent to $C_u$ is spacelike or null relative to $\eta$ and hence it is non-timelike relative to $g$. Thus $C_u$ is non-timelike with respect to $g$. Similarly with $\ub$ in the role of $u$ and $\Cb_{\ub}$ in the role of $C_u$.

Another aspect corresponding to the non-timelike nature of $C_u$ and $\Cb_{\ub}$ is that both of the gradients (with respect to $g$) of $u$ and $\ub$ are causal relative to $g$, see \eqref{g-grad-u-ub-causal}.
\end{remark}

\subsubsection{Multipliers adapted to the geometry of membrane}
The strategy of our proof is in spirit of the recent studies of large data problem for semilinear wave equations \cite{Miao-Yu} (see also \cite{Wang-Yu,Wang-Yu1}), but in a quasilinear setting. The RME is a geometric wave equation on the curved background $(\mathbb{R}^{1+n}, g(\partial\phi))$. Therefore the null frame $(L,\underline{L})$ in Minkowski spacetime is no longer causal (non-spacelike) with regard to $g(\p \phi)$. If we still apply the vector fields $L, \, \underline{L}$ as multipliers, and treat those quasilinear terms in the RME via integration by parts, then the resulted energies associated to $L$ and $\Lb$ respectively are not positive.  To be more specific, in the energy argument for $\phi_N$ (the highest order derivative of $\phi$) which satisfies the highest order RME \eqref{3.18},  the energy associated  to $L$ (with $L$ being the multiplier) is roughly
\begin{equation}\label{modified-energy-L-simplify}
 \int_{C_u} |L\phi_N|^2 + \int_{\Cb_{\ub}}  \left( 1 - \frac{1}{2} |\nablaslash\phi|^2 \right)  |\nablaslash\phi_N|^2  + \frac{3}{8} |\Lb\phi|^2 |L\phi_N|^2  -  \frac{1}{8} |L\phi|^2 |\Lb\phi_N|^2 + \cdots
\end{equation}
where we use $\cdots$ to denote terms with indeterminate signs.
Notice that, $\phi$ is denoted as the unknown solution of the RME, and there is $|L\phi| \lesssim \delta, \, |\nablaslash\phi | \lesssim \delta^{\frac{3}{4}}, \, |\Lb \phi|\lesssim 1$. Due to the presence of the non-positive terms such as $-\int_{\Cb_{\ub}} \frac{1}{8} |L\phi|^2 |\Lb\phi_N|^2$, \eqref{modified-energy-L-simplify} will never be positive.
 A similar issue occurs for the energy associated to $\underline{L}$, which takes the form of
\begin{equation}\label{modified-energy-Lb-simplify}
\int_{\Cb_{\ub}} |\Lb\phi_N|^2  + \int_{C_u} \left( 1 - \frac{1}{2} |\nablaslash\phi|^2 \right) |\nablaslash\phi_N|^2 + \frac{3}{8 } |L\phi|^2 |\Lb\phi_N|^2  - \frac{1}{8 } |\Lb\phi|^2 |L\phi_N|^2 + \cdots
\end{equation}
Note that, for the large data problem, there is a hierarchy in using $L$ and $\Lb$ as the multipliers in energy estimates \cite{Miao-Yu,Wang-Yu,Wang-Yu1}.
But now, the non-positivity of \eqref{modified-energy-L-simplify} and \eqref{modified-energy-Lb-simplify} suggests that this hierarchy of energy estimates is not allowed (although \eqref{modified-energy-L-simplify} $+$ \eqref{modified-energy-Lb-simplify} could be non-negative, but not \eqref{modified-energy-L-simplify} or \eqref{modified-energy-Lb-simplify} respectively). These would be explained in more detail in Remark \ref{rk-quailinear-bdry}. In a word, to establish the hierarchy of energy estimates for the RME (in the large data setting),  $L$ and $ \underline{L}$ are not eligible for multipliers, since they are not causal with respect to $g(\partial\phi)$, and hence the associated energies are not positive.

In regard of the above difficulties, we should pay attention to the geometry of membrane $(\mathbb{R}^{1+n}, g(\partial\phi))$. We employ the modified vector fields $\tilde L, \, \underline{\tilde L}$:
\begin{equation}\label{def-L-Lb-tilde}
\tilde{L}=L+(L\phi)^2\underline{L} \quad\text{and}\quad \tilde{\underline{L}}=\underline{L}+(\underline{L}\phi)^2L,
\end{equation}
which are causal with respect to $g(\partial\phi)$. Moreover, $\tilde{L}$ and $\tilde{\Lb}$ inherit the main features of $L, \, \Lb$ in the following sense. They have signature $+1$ and $-1$ (see the definition for signature in Section \ref{section null form}) respectively, and quantitatively, there is the equivalence $(\tilde L,\underline{\tilde L}) \sim (L, \Lb)$. In the meantime, taking $\tilde L, \, \underline{\tilde L}$ as multipliers ensures the following coercive energies
\begin{align*}
& \int_{C_u} |L \phi_N|^2+|L\phi|^2|\nablaslash \phi_N|^2+|L\phi|^4|\underline{L} \phi_N|^2 + \int_{\Cb_{\ub}} |\nablaslash \phi_N|^2+|\underline{L}\phi|^2 |L \phi_N|^2+|L\phi|^2|\underline{L} \phi_N|^2,\\
& \int_{C_u} |\nablaslash \phi_N|^2+|\underline{L}\phi|^2|L \phi_N|^2+|L\phi|^2|\underline{L} \phi_N|^2 + \int_{\Cb_{\ub}}  |\underline{L} \phi_N|^2+|\underline{L}\phi|^2|\nablaslash \phi_N|^2+|\underline{L}\phi|^4|L \phi_N|^2.
\end{align*}

In fact, $\tilde L$ and $\underline{\tilde L}$ are the leading parts of $-2D u,$ $-2D \ub$, where $D u, \, D \ub$ mean taking the gradient in terms of $g(\p\phi)$ and $D u, \, D \ub$ are both causal with respect to $g(\partial\phi)$ as well. Meanwhile, in contrast to the intrinsic method in the work \cite{Christodoulou1, Miao-Yu1, Speck-Hol-Luk-Wong}, we do not introduce optical functions associated to $g(\p\phi)$ to foliate the membrane $(\mathbb{R}^{1+n}, g(\partial\phi))$ with the ``genuine'' double null foliation. Instead, we stick to the Minkowski null foliation $\{C_{u}|u\in\mathbb R\}$ and $\{\underline{C}_{\underline{u}}|\underline{u}\in\mathbb R\}$, so that we can make full use of the symmetries of Minkowski spacetime. So far, it is natural to take $D u$, $D \ub$ or the related ones $\tilde L, \, \underline{\tilde L}$ as the candidates for multipliers, since these four vector fields are not necessarily null but causal with respect to $g(\partial\phi)$. In application, we choose the intermediates $\tilde L, \, \underline{\tilde L}$ which are not only effective in handling the quasilinear feature of RME, but also avoid the long story of the intrinsic method. Besides, the usage of $\tilde L, \, \underline{\tilde L}$ as multipliers rather than $-2D u$, $-2D \ub$ also saves us quite a lot of tedious calculations.

\subsubsection{The expanding effect near the null infinity of membrane}\label{sub:1.3.2}
Referring to \cite{Christodoulou-K-93}, we also explore the asymptotical behavior of the geometry of membrane (we take $n=3$ for instance). The following results hold true even in the small data setting.
We introduce a null frame $\{e_3, e_4, e_A, \, A=1, 2\}$ adapted to the sphere $S_{u, \ub}$ (see Section \ref{sec-null-frame}) with respect to the geometry of membrane $(\mathbb{R}^{1+3}, g(\p \phi))$, such that the incoming null vector field $e_4$ is almost affine, namely $D_4 e_4 = O(r^{-3}) e_4 + O(r^{-3}) e_A$ (see \eqref{D-4-4-null-frame}). We always use $D$ to denote the covariant derivative associated to $g(\p \phi)$.

By the previously obtained decay results, we can define the limitation $\Xi_1\doteq \lim\limits_{r \rightarrow \infty} \left( r  e_4  \phi \right)$. Let $\chib_{AB}$ be the second fundamental form of $S_{u, \ub}$ along $e_4$.
Then (see Theorem \ref{lem-2nd-form}), $$ \tr \chib + \frac{2}{r} = \frac{\Hb}{r^{3}} + O( r^{-3-\frac{1}{2}}) \quad \text{with} \quad \Hb = \frac{1}{2} (\Xi_1)^2.$$ In other words, $\lim\limits_{r \rightarrow \infty} r^2 (r \tr \chib + 2) = \Hb >0$ indicates that the volume of the projected (onto $S_{u, \ub}$) cross-sectional area of the bundle of flow lines of $e_4$ has an expanding effect near the null infinity, see Figure \ref{fig:expanding}: The black lines denote the Minkowski picture, while the red lines represent the membrane picture.
Actually, this expanding effect is due to the curvature of $g (\p \phi)$, as manifested by the Raychaudhuri equation along $e_4$ (refer to Theorem \ref{thm-memory}):
\begin{equation*}
\frac{\p}{\partial u} \Hb = -  \lim\limits_{r \rightarrow \infty}  r^{3} R_{4  4} = \frac{1}{2} \partial_u (\Xi_1)^2,
\end{equation*}
where $R_{\mu \nu}$ is the Ricci tensor of $g(\p \phi)$.
\begin{figure}\label{fig:expanding}
\centering
\includegraphics[width=2.5in]{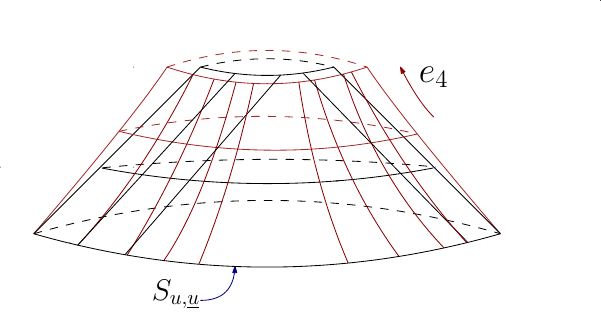}
 \caption{Expansion projected onto $S_{u, \ub}$}
\end{figure}
 In the mean time, we manage to show that (see Section \ref{sec-expanding-e3}), $$D_4  e_3 = \frac{\Hb}{r^3} e_3 + O( r^{-3-\frac{1}{2}}) e_3 + O( r^{-3-\frac{1}{2}}) e_B,$$ with $e_B$ tangent to $S_{u, \ub}$

Another interpretation for this nonlinear effect is: The ``Hawking mass'' $m$ of $S_{u, \ub}$ has the asymptotical behavior $$m = \frac{M}{r} + O(r^{-2}) \quad \text{with} \quad M = \frac{1}{32 \pi} \int_{S^2} (\Xi_1)^2 \di \mu_{\gamma} >0.$$ Here $\gamma$ is the standard metric on $S^2$.
Although $\lim\limits_{r\rightarrow \infty} m =0$,
we have $\lim\limits_{r\rightarrow \infty} rm =M >0$ and
\begin{equation*}
\frac{\p M}{\p u} = - \frac{1}{16 \pi} \int_{S^2} \left( \lim\limits_{r \rightarrow \infty}  r^{3} R_{4  4} \right) \di \mu_{\gamma} = \frac{1}{32 \pi} \int_{S^2} \p_u (\Xi_1)^2 \di \mu_{\gamma}.
\end{equation*}
This is an analogue of gravitational wave memory \cite{Christodoulou-memory, Christodoulou-K-93}.

\subsection{Arrangement of the paper}

We arrange our paper as follows: In Section \ref{Section2}, we give some preliminaries including notations, energy equalities, Gr\"{o}nwall inequalities and so on. Section \ref{Section-large-data} is devoted to proving the global existence of the solution in the short pulse region {\rm II}. The smallness of the solution on the last slice $C_{\delta}$ is justified in Section \ref{Section4}. In Section \ref{Section-small-data}, we prove the global existence of the smooth solution in the small data region {\rm I}. In Section \ref{sec-blow up-scri}, we show the expanding effect near the null infinity. At last, we collect some calculations and prove the existence of initial data 
in the appendix.


\section{Preliminaries}
\label{Section2}
In this section, we mainly introduce some necessary preliminaries used repeatedly throughout the paper.

\subsection{Notations}\label{Subsection2.1}
With regard to the indices, the Greek indices $\alpha,\beta, \mu,\nu, \cdots$ range from 0 to n. Specifically, we use the Latin indices $a,b,\cdots$ to denote the null coordinates $u$ and $\underline{u}$, and $\theta, \omega, \varphi, \cdots$ to denote the angular coordinates on $S_{u,\underline{u}}$. The capital Latin letters $A, B, \cdots$ will also be used to denote the indices on $S_{u,\underline{u}}$. In the following, Einstein's summation convention is also used, which means that repeated upper and lower indices are summed over their ranges.

Let $\eta_{\mu \nu}$ be the standard flat Lorentzian metric, $d^{n+1}x$ or $d^nxdt$ denotes the volume form in the Minkowski spacetime.
We use $g_{\mu\nu}$ to denote the Lorentzian metric associated to the geometry of membrane, and $D$ the associated Levi-Civita connection. $\gamma_{AB}$ denotes the induced metric (from the Minkowski metric) on $S_{u,\underline{u}}$  and $\nablaslash$ the connection associated to $\gamma_{AB}$.

We also use the Klainerman's vector fields $Z=\{\partial,\Gamma\} =\{\partial_{t},\partial_{i},\Omega_{0i},\Omega_{ij}, S\}$, $i, j =1, \cdots, n$ as  commutators, where $\Omega_{0i}=t\partial_i+x_{i}\partial_{t}$ is the Lorentz boost, $\Omega_{ij}=x_{i}\partial_{j}-x_{j}\partial_{i}$ denotes the rotation and $S=t\partial_{t}+r\partial_{r}$ denotes the scaling. We let $\bar{\p}=\{L, \nablaslash\}$ be the vector fields tangential to the outgoing null cone $C_u$.
For any given function $f$, we use the shorthand notation $\Omega^{m}f$ to denote $ \prod_{k=1}^{m}\Omega_{i_{k}j_{k}}f$, similar conventions also apply to $Z^{m}f$. It is clear that in the short pulse region $\mathcal{D}=\{(t,x)|1\leq t<+\infty, 0\leq u \leq\delta\}$, we have $r\sim \underline{u}$.
Therefore,
$$
\Omega^{k}f \sim r^{k}|\nablaslash^{k} f|\sim \underline{u}^{k}|\nablaslash^{k}f|. 
$$

Suppose $u, \,u^\prime \in(0,u^{*})$ and $\underline{u}, \,\underline{u}' \in(1-u^{*},u^{*})$. Let $C_{u}^{\underline{u}'}$ denote the part of $C_{u}$ with $1-u^{*}\leq \underline{u}\leq \underline{u}'$ and $\underline{C}_{\underline{u}}^{u'}$ denote the part of $\underline{C}_{\underline{u}}$ with $0\leq u\leq u'$. We shall use $\Sigma_{1}$ to denote initial hypersurface $\{t=1\}$. For the energy norms, given any function $f$, we define
\begin{align*}
\|f\|_{L^{p}(C_{u}^{\underline{u}})} & \doteq \left(\int_{1-u^{*}}^{\underline{u}}\int_{S_{u,\underline{u}'}}|f|^{p}r^{n-1}d\underline{u}'d\sigma_{S^{n-1}}\right)^{1/p}, \\
\|f\|_{L^{p}(\underline{C}_{\underline{u}}^{u})} &\doteq \left(\int_{0}^{u}\int_{S_{u',\underline{u}}}|f|^{p}r^{n-1}du' d\sigma_{S^{n-1}}\right)^{1/p},
\end{align*}
and
$$\|f\|_{L^{\infty}(D)} \doteq ess\sup_{(x^\mu)\in D}|f(x^\mu)|.$$
For simplicity, we always use $\|f\|_{L^{p}(C_{u})}$ to denote $\|f\|_{L^{p}(C_{u}^{\underline{u}})}$ and $\|f\|_{L^{p}(\underline{C}_{\underline{u}})}$ to denote $\|f\|_{L^{p}(\underline{C}_{\underline{u}}^{u})}$.

\subsection{Null Forms}\label{section null form}
For a real valued quadratic form $\mathcal{Q}$ defined on $\mathbb{R}^{1+n}$, it is called a \emph{null form} if for any null vector $\xi \in \mathbb{R}^{n+1}$, we have $\mathcal{Q}(\xi, \xi) = 0.$

There are examples of null forms, for $0\leq \alpha, \beta \leq n$ and $\alpha \neq \beta$,
\begin{equation}\label{basic null forms}
\begin{split}
 Q_0 (\xi, \chi) &= \eta^{\mu \nu} \xi_{\mu}\chi_{\nu}, \\
 Q_{\alpha \beta}(\xi, \chi) &=\p_\alpha  \xi  \p_\beta\chi - \p_\alpha \chi \p_\beta \xi.
 \end{split}
\end{equation}
Given scalar functions $\varphi,\psi$ and a null form $Q(\xi,\chi)=Q^{\alpha\beta}\xi_{\alpha}\chi_{\beta}$,
we use $\mathcal{Q}(\varphi,\psi)$ as a short hand for $\mathcal{Q}(\varphi,\psi)=Q^{\alpha\beta}\partial_{\alpha}\varphi\partial_{\beta}\psi$, then it is easy to see
\begin{equation*}
    Z \mathcal{Q} (\varphi, \psi) = \mathcal{Q}(Z \varphi,  \psi) + \mathcal{Q}(\varphi, Z \psi) + \sum a_{\alpha \beta} \mathcal{Q}_{\alpha \beta} (\varphi, \psi),
\end{equation*}
for some constants $a_{\alpha \beta}.$
In particular,
\begin{equation}\label{commu-Q-Gamma}
[\partial_\mu, Q_0] = 0, \quad [\Omega_{0i}, Q_0]=0, \quad [\Omega_{ij}, Q_0]=0, \quad [S, Q_0]=-2Q_0,
\end{equation}
where for any commutator $\Gamma$ and functions $\varphi, \, \psi$, we define the following Lie bracket
\begin{equation*}
[\Gamma, \mathcal{Q}](\varphi, \psi) = \Gamma \mathcal{Q}(\varphi, \psi) - \mathcal{Q}(\Gamma\varphi, \psi) -\mathcal{Q}(\varphi, \Gamma\psi).
\end{equation*}
We remark here that in the membrane equation, only $Q_{0}$ plays a key role in this paper. The null forms can also be interpreted by a signature consideration. Indeed, if we assign the following signatures to $L, \, \Lb$ and $\nablaslash$
\begin{equation}\label{def-Sign}
\text{Sign}(L) =+1, \quad \text{Sign}(\Lb) =-1,\quad \text{Sign}(\nablaslash)=0,
\end{equation}
it is easy to check that the null forms \eqref{basic null forms} have the signature $p$ with $-2< p <2$ and $p \in \mathbb{N}$. Specifically, the signature for the quadratic null forms could not reach $-2$, which entails that terms involving two $\Lb$ derivatives such as $\underline{L}\varphi\underline{L}\psi$ are forbidden in the null forms \eqref{basic null forms}.

 This could be generalized to {\it double null forms}, which are cubic and take the form of $\mathcal{Q}(\varphi,\mathcal{Q}(\chi,\psi))$. The double null form also has a signature $p$ with $-2< p< 2$ and $p \in \mathbb{N}$ .

\subsection{Formulation of the relativistic membrane equation}
In this section, we recall the derivation of the RME \cite{Hoppe4, Hoppe5}, and then treat it as a geometric wave equation evolving in a curved spacetime. This idea is inspired by the breakthrough of Christodoulou \cite{Christodoulou1}, in which he showed the shock formation for the isentropic and irrotational relativistic Euler equations (the Chaplygin gas is not included) by deeply investigating the acoustic geometry.

\subsubsection{The Lagrangian description}\label{Langrangian-M-eq}
Let
\begin{equation}\label{2.3}
\eta=\eta_{\mu\nu}dx^{\mu}\otimes dx^{\nu}=-dt^{2}+\sum_{1 \leq i \leq n+1}(dx^{i})^2,
\end{equation}
be the standard flat metric in the Minkowski spacetime $\mathbb R^{(1+n)+1}$. Assume $x^{\mu}=x^{\mu}(\psi^{0},\cdots,\psi^{n})$ be an $1+n$ dimensional timelike submanifold $\mathcal{M}$ embedded in $\mathbb R^{(1+n)+1}$. Then we have the following induced volume for $\mathcal{M}$
\begin{equation}\label{2.4}
\text{Vol} (\mathcal{M})=\int_{\mathcal{M}}\sqrt{g}d^{1+n}\psi=\int_{\mathcal{M}}\sqrt{\Big| \det \left(g_{\alpha\beta}=\frac{\partial x^{\mu}}{\partial \psi^{\alpha}}\frac{\partial x^{\nu}}{\partial \psi^{\beta}}\eta_{\mu\nu}\right)\Big|}d^{1+n}\psi.
\end{equation}
We call $\mathcal{M}$ the relativistic membrane if it satisfies the Euler-Lagrange equation of \eqref{2.4}. It is straightforward to calculate and deduce the membrane equation
\begin{equation}\label{2.5}
\Box_g x^\mu = \frac{1}{\sqrt{g}}\partial_{\alpha}\left(\sqrt{g}g^{\alpha\beta}\partial_{\beta}x^{\mu}\right)=0,\quad \mu=0,1,\cdots,n+1.
\end{equation}

Now we are interested in the graphic description. Namely, if we let $x^{\mu}=\psi^{\mu}$ for $\mu=0,1,\cdots,n$, we have
\begin{equation}\label{wave-coordinate}
\partial_{\alpha}\left(\sqrt{g}g^{\alpha\beta}\right)=0,\quad \beta=0,1,\cdots,n.
\end{equation}
Then the only one single equation for $x^{n+1}=\phi(x^{0},\cdots,x^{n})$ is given by
\begin{equation}\label{2.7}
\Box_g \phi = \frac{1}{\sqrt{g}}\partial_{\alpha}\left(\sqrt{g}g^{\alpha\beta}\partial_{\beta}\phi\right)=0.
\end{equation}
Combining \eqref{wave-coordinate} with \eqref{2.7}, we obtain
\begin{equation}\label{2.8}
g^{\alpha\beta}\partial_{\alpha}\partial_{\beta}\phi=0,
\end{equation}
where the associated metric takes the following form
\begin{equation}\label{2.9}
g_{\alpha\beta}=\eta_{\alpha\beta}+\partial_{\alpha}\phi\partial_{\beta}\phi,\quad g^{\alpha\beta}=\eta^{\alpha\beta}-\frac{\partial^{\alpha}\phi\partial^{\beta}\phi}{g},
\end{equation}
and the determinant is given by
\begin{equation}\label{2.10}
g \doteq1+Q(\phi,\phi)=1+\eta^{\alpha\beta}\partial_{\alpha}\phi\partial_{\beta}\phi.
\end{equation}
We should notice that in \eqref{2.9} the indices are raised by the flat Lorentzian metric $\eta$. This convention will be used throughout the paper unless otherwise stated.

The resulted equation \eqref{2.8} is exactly the RME in graphic description \eqref{1.1}. With the formula of the metric \eqref{2.9} in the coordinate system $(x^\mu), \, \mu = 0, \cdots, n$,  \eqref{wave-coordinate} and \eqref{2.8} are in fact equivalent. Hence, the RME in graphic description reduces to a single equation \eqref{2.8} or equivalently \eqref{wave-coordinate}.

\begin{remark}
\eqref{wave-coordinate} is called the wave coordinates gauge, which plays an important role in the proof of global smooth solution for the Einstein vacuum equation by Lindblad and Rodnianski \cite{Lind-Rod}.

The RME in graphic description \eqref{2.8} could also be understood as the following geometric wave equation
\begin{equation}\label{2.11}
\Box_{g(\p\phi)}\phi=0,
\end{equation}
with the wave gauge \eqref{wave-coordinate}, where the metric $g_{\mu \nu} (\p \phi)$ is defined by \eqref{2.9}.

\end{remark}

\subsubsection{The geometric description}\label{sec-geometry-mem-eq}
Set $\vec{x} \doteq \left(x^0, x^1, \cdots , x^n \right)$ and consider the embedding of an $1+n$ dimensional (time-like) manifold $\mathcal{M}$ in the Minkowski space-time $\mathbb{R}^{(1+n)+1}$,
\[ \begin{array}{cccc}
X: & \mathcal{M}  & \longrightarrow & \mathbb{R}^{(1+n)+1}\\
\,& \vec{x} & \longrightarrow & \left( \vec{x}, \, \phi(\vec{x}) \right).
\end{array}
\]
Equivalently, $\mathcal{M}$ is the level set of the function $f=x^{n+1} - \phi (\vec{x})$.
Let $$\hat e_\mu \doteq \p_\mu X= \p_\mu + \p_\mu \phi \p_{n+1}, \quad \mu = 0, 1, \cdots, n,$$ then
$\{ \hat e_\mu \}, \, \mu= 0, 1,\cdots, n$ is a basis tangent to $\mathcal{M}$. And the gradient $( \p^{\mu'} f), \, \, \mu'=0,1, \cdots, n+1$ is spaclike in the $(1+n+1)$-dim Minkowski target manifold.  Thus $$N= \frac{\p^{\mu'} f}{|\p^{\mu'} f| } = \frac{-\eta^{\mu\nu} \p_\mu \phi \p_\nu + \p_{n+1}}{1+ \eta^{\mu\nu} \p_\mu \phi \p_\nu \phi}$$ is the unit  vector field normal to $\mathcal{M}$. Now,
\begin{equation}\label{def-induced-metric-mem}
g_{\mu\nu} = \eta(\hat e_\mu, \hat e_\nu), \quad \mu, \nu = 0, 1, \cdots, n,
\end{equation}
is the induced metric on $\mathcal{M}$. And the second fundamental form takes
\begin{equation}\label{def-2nd-mem}
k_{\mu\nu} = - \eta( \p_{\hat e_\mu} N, \hat e_\nu) = \eta(N, \p_{\hat e_\mu} \hat e_{\nu}), \quad \mu, \nu = 0, 1, \cdots, n.
\end{equation}
In fact, letting $D$ be the covariant derivative associated to $g_{\mu \nu}$, and $\Gamma_{\mu\nu}^{\alpha}$ be the Christoffel symbol, then we derive
\begin{equation}\label{2nd-calculate}
\begin{split}
k_{\mu \nu} N &=\p_{\hat e_\mu} \hat e_{\nu} - D_{\hat e_\mu} \hat e_{\nu} \\
&=  \p_{\mu} \p_{\nu} X - \Gamma_{\mu\nu}^{\alpha} \p_{\alpha} X \\
&= D_\mu D_\nu X.
\end{split}
\end{equation}
The membrane equation is given by
\begin{equation}\label{eq-mem-trk}
\tr k N=0 \quad  \Longleftrightarrow \quad \Box_g X =0.
\end{equation}
Comparing the last component, we have
\begin{equation*}
k_{\mu \nu}  =\left(1+ \eta^{\mu\nu} \p_\mu \phi \p_\nu \phi \right) D_\mu D_\nu \phi.
\end{equation*}

Let $R_{\mu \alpha \nu \beta}, \, R_{\mu \nu}$ be the Riemannian and Ricci tensors of $g_{\mu \nu}$.
The Gauss equations are given by (an identity)
\begin{equation}\label{eq-Gauss}
R_{\mu \alpha \nu \beta} - k_{\mu\nu} k_{\alpha \beta} + k_{\mu \alpha} k_{\nu \beta} =0.
\end{equation}
Contracting \eqref{eq-Gauss} with $g^{\alpha \beta}$, and making use of the membrane equation $\tr k =0,$  we obtain
\begin{equation}\label{eq-Gauss-Ricci}
R_{\mu \nu}+ k_{\mu \alpha} k_{\,\nu}^{\alpha} =0.
\end{equation}
The Codazzi equations read
\begin{equation}\label{eq-Cdazzi}
D_{\alpha}k_{\beta \gamma} - D_{\beta} k_{\alpha \gamma} =0.
\end{equation}
Thess are identities as well.

\subsection{Energy identity}
For the geometric wave equation \eqref{2.11}, we can define the corresponding energy momentum tensor
\begin{equation}\label{2.13}
T^{\alpha}_{\beta}(\phi) \doteq g^{\alpha \mu}\partial_{\mu}\phi\partial_{\beta}\phi-\frac{1}{2} g^{\mu\nu}\partial_{\mu}\phi\partial_{\nu}\phi\delta^{\alpha}_{\beta},
\end{equation}
where $\delta^{\alpha}_{\beta}$ denotes the Dirac function.

For any vector field $\xi$, the corresponding current $P^{\alpha}$ is defined as
\begin{equation}\label{2.14}
P^{\alpha}= P^{\alpha}[\phi,\xi] \doteq T^{\alpha}_{\beta} (\phi) \cdot \xi^{\beta}.
\end{equation}
The divergence of $P^{\alpha}$ could be calculated straightforwardly.
\begin{lemma}\label{lem:2.4}
The energy current $P^{\alpha}$ satisfies
\begin{equation}\label{2.15}
\frac{1}{\sqrt{g}}\partial_{\alpha}(\sqrt{g}P^{\alpha})= \Box_{g(\p\phi)}\phi \cdot \xi \phi +T^{\alpha}_{\beta} (\phi) \cdot \partial_{\alpha}\xi^{\beta}
-\frac{1}{2\sqrt{g}}\xi(\sqrt{g}g^{\gamma\rho})
\partial_{\gamma}\phi\partial_{\rho}\phi.
\end{equation}
\end{lemma}
\begin{proof}
 By a straightforward calculation, we have $$D_\alpha P^\alpha = \Box_{g(\p\phi)}\phi \cdot \xi \phi+T^{\alpha}_{\beta} (\phi) D_{\alpha}\xi^{\beta}.$$ The covariant derivative is further expressed by $D_{\alpha}\xi^{\beta} = \partial_{\alpha}\xi^{\beta} - \Gamma_{\alpha \gamma}^\beta \xi^\gamma$, where $\Gamma_{\alpha \gamma}^\beta$ are the connection coefficients of $g$ in the rectangular coordinates of the underlying Minkowski spacetime. Expanding the $\Gamma_{\alpha \gamma}^\beta$, we obtain $T^{\alpha}_{\beta} (\phi)\Gamma_{\alpha \gamma}^\beta \xi^\gamma = \frac{1}{2\sqrt{g}}\xi(\sqrt{g}g^{\gamma\rho})
\partial_{\gamma}\phi\partial_{\rho}\phi.$
\end{proof}
\begin{remark}
Introducing the symmetric $(2,0)$ tensor field $T^{\alpha \beta} = T_{\gamma}^\beta g^{\alpha \gamma}$, $T^{\alpha}_{\beta} (\phi) D_{\alpha}\xi^{\beta}$ becomes $$\frac{1}{2}T^{\alpha \beta} \mathcal{L}_\xi g_{\alpha \beta},$$ where $\mathcal{L}_\xi$ denotes the lie derivative with respect to $\xi$. Equivalently, we have the following geometric formulation of \eqref{2.15}: $$D_\alpha P^\alpha = \Box_{g(\p\phi)}\phi \cdot \xi \phi+\frac{1}{2}T^{\alpha \beta} \mathcal{L}_\xi g_{\alpha \beta}.$$
\end{remark}

The $(0,2)$-energy momentum tensor $T_{\alpha \beta}$ is defined by $T_{\alpha \beta} \doteq g_{\alpha \sigma} T^{\sigma}_{\beta}$, where $T^{\sigma}_{\beta}$ is the $(1,1)$-energy momentum tensor defined in \eqref{2.13}. We have
\begin{equation}\label{def-energy-xi-u-ub}
-P^{u}[\phi,\xi] = T(-Du, \xi), \quad -P^{\ub}[\phi,\xi] = T(-D \ub, \xi),
\end{equation}
and
\begin{equation}\label{def-energy-xi-t}
-P^{t}[\phi,\xi] = T(-Dt, \xi),
\end{equation}
where for any function $f$, $Df$ means taking gradient on $f$ with respect to $g_{\alpha\beta}$. Hence, the corresponding energies on the slices $C_{u}$, $\underline{C}_{\underline{u}}$ and $\Sigma_{\tau}$ are defined as follows
\begin{subequations}
\begin{align}
E^2[\phi,\xi](u) & =-\int_{C_{u}}\sqrt{g}P^{u}[\phi,\xi]r^{n-1}d\underline{u}d\sigma_{S^{n-1}}, \label{2.17} \\
\underline{E}^2[\phi,\xi](\underline{u}) & =-\int_{\underline{C}_{\underline{u}}}\sqrt{g}P^{\underline{u}}[\phi,\xi]r^{n-1}dud\sigma_{S^{n-1}},  \label{2.17} \\
\hat{E}^2[\phi,\xi](\tau) & =-\int_{\Sigma_{\tau}}\sqrt{g}P^{t}[\phi,\xi]r^{n-1}drd\sigma_{S^{n-1}}.\label{2.18}
\end{align}
\end{subequations}
By the divergence theorem, there is the following energy identity
\begin{equation}\label{2.19}
E^2[\phi,\xi](u)+\underline{E}^2[\phi,\xi](\underline{u})=\hat{E}^2[\phi,\xi](t_{0})+E^2[\phi,\xi](u_{0})-\iint_{\mathcal{D}_{u, \ub}}\partial_{\alpha}(\sqrt{g}P^{\alpha})dx^{n+1},
\end{equation}
where $\mathcal{D}_{u, \ub}$ is the domain bounded by $C_{u}$, $\underline{C}_{\underline{u}}$, $C_{u_{0}}$ and $\Sigma_{t_0}$.

At last, we recall the standard Gr\"{o}nwall inequality.
\subsection{Gr\"{o}nwall inequality}
\begin{lemma}
[\emph{Differential Form}]
Let $f(t)$ be a non-negative function defined on an interval $I$ with initial point $t_{0}$. If $f$ satisfies
$$
\frac{df}{dt}\leq a\cdot f+b,
$$
for  non-negative functions $a,\, b\in L^{1}(I)$, then for all $t\in I$, we have
$$
f\leq e^{A(t)}\left(f(t_{0})+\int_{t_{0}}^t e^{-A(\tau)}b(\tau)d\tau\right),
$$
where $A(t)=\int_{t_{0}}^{t}a(\tau)d\tau$.
\end{lemma}
\begin{lemma}[Integral Form] If $f(t), a(t)$ and $b(t)$ are non-negative, $b(t)$ is increasing and
\begin{equation*}
f(t)\leq b(t)+\int_{t_{0}}^{t}a(\tau)f(\tau)d\tau,
\end{equation*}
then,
\begin{equation*}
f(t)\leq b(t) \exp \left(\int_{t_{0}}^{t} a(\tau) d\tau \right).
\end{equation*}
\end{lemma}

\section{Globally smooth solution in the short pulse region {\rm II}}\label{Section-large-data}

In this section, we prove the global existence of smooth solution in the short pulse region by the method of energy estimate.

Letting $\phi_{k}=\sum_{l\leq k}\delta^{l}\partial^{l}\Gamma^{k-l}\phi$, we define the $k$-th order energy as follows
\begin{equation}\label{3.1}
\begin{split}
\delta E_{k}(u,\underline{u})&=\|\nablaslash\phi_{k}\|_{L^{2}(\underline{C}_{\underline{u}}^{u})}+\|\underline{L}\phi L\phi_{k}\|_{L^{2}(\underline{C}_{\underline{u}}^{u})}+\|L\phi\underline{L}\phi_{k}\|_{L^{2}(\underline{C}_{\underline{u}}^{u})}\\
&+\|L\phi_{k}\|_{L^{2}(C_{u}^{\underline{u}})}+\|L\phi\nablaslash \phi_{k}\|_{L^{2}(C_{u}^{\underline{u}})}+\|(L\phi)^2\underline{L}\phi_{k}\|_{L^{2}(C_{u}^{\underline{u}})},
\end{split}
\end{equation}
and
\begin{equation}\label{3.2}
\begin{split}
\delta^{\frac{1}{2}}\underline{E}_{k}(u,\underline{u})&=\|\nablaslash\phi_{k}\|_{L^{2}(C_{u}^{\underline{u}})}
+\|\underline{L}\phi L\phi_{k}\|_{L^{2}(C_{u}^{\underline{u}})}+\|L\phi\underline{L}\phi_{k}\|_{L^{2}(C_{u}^{\underline{u}})}\\
&+\|\underline{L}\phi_{k}\|_{L^{2}(\underline{C}_{\underline{u}}^{u})}+\|\underline{L}\phi\nablaslash\phi_{k}\|_{L^{2}(\underline{C}_{\underline{u}}^{u})}
+\|(\underline{L}\phi)^2\underline{L}\phi_{k}\|_{L^{2}(\underline{C}_{\underline{u}}^{u})}.
\end{split}
\end{equation}
The energy $E_k$ is associated to the $\tilde L$ multiplier and the energy $\Eb_k$ is associated to the $\tilde{\Lb}$ multiplier. This will be proved in Lemma \ref{lemma-energy-formula}.
The inhomogeneous energy is defined as
\begin{equation}\label{3.3}
E_{\leq k}(u,\underline{u})=\sum_{0\leq j\leq k}E_{j}(u,\underline{u}),\quad \underline{E}_{\leq k}(u,\underline{u})=\sum_{0\leq j\leq k}\underline{E}_{j}(u,\underline{u}).
\end{equation}

Given any subregion in the short pulse region {\rm II}, we suppose $u, \ub$ have the upper bound: $u\leq u^{*}, \,\, \underline{u} \leq \underline{u}^{*}$.
The main aim of this section is to prove the following theorem.
\begin{theorem}
There exists a positive constant $\delta_{0}$ such that for $\delta\in(0,\delta_{0})$, $u\in[0,u^{*}]$ and $\underline{u}\in [1-\frac{\delta}{2},\underline{u}^{*}]$, we have for $N\in \mathbb{N}$  and $N\geq4$
\begin{equation}\label{3.4}
E_{\leq N}(u,\underline{u})+\underline{E}_{\leq N}(u,\underline{u})\lesssim I_N(\psi_0, \psi_1),
\end{equation}
where $I_N(\psi_0, \psi_1)$ is a positive constant depending only on the initial data up to $N+1$ order of derivatives.
\end{theorem}
The proof is mainly based on the standard bootstrap argument. We assume that there exists a large constant $M$ (may depend on $\phi$) to be determined, such that
\begin{equation}\label{Bootstrap-assumption}
E_{\leq N}(u',\underline{u}')+\underline{E}_{\leq N}(u',\underline{u}')\leq M,
\end{equation}
for all $u'\in [0, u]$ and $\ub' \in [1-\frac{\delta}{2}, \ub]$. At the end of the current section, we will show that we can choose $M$ such that it depends only on the initial data. We shall introduce $\mathcal{D}_{u, \ub}$ to denote the domain bounded by $C_{u}$, $\underline{C}_{\underline{u}}$, $C_{u_{0}}$ and the initial hypersurface $\Sigma_{1}$.

\subsection{Sobolev inequality}
We start with some preliminary estimates based on the Sobolev inequalities, see Lemma \ref{lemma-apriori-estimate}.
These\emph{ a priori} estimates would then help us to see the positivity of the energy, refer to  Lemma \ref{lemma-energy-formula}.
\begin{lemma}\label{lemma-apriori-estimate}
With the bootstrap assumption \eqref{Bootstrap-assumption}, we have for $n=2, \, 3$
\begin{eqnarray*}
\|\nablaslash\partial^{l}\Gamma^{k-l}\phi\|_{L^{\infty}(S_{u,\underline{u}})}&\lesssim&\delta^{\frac{3}{4}-l}|\underline{u}|^{d_{n}-\frac{1}{2}}M,\quad 0\leq l\leq k\leq N-2,\\
\|L\partial^{l}\Gamma^{k-l}\phi\|_{L^{\infty}(S_{u,\underline{u}})}&\lesssim&\delta^{1-l}|\underline{u}|^{d_{n}-\frac{1}{2}}M,\quad 0\leq l\leq k\leq N-2,\\
\|\underline{L}\partial^{l}\Gamma^{k-l}\phi\|_{L^{\infty}(S_{u,\underline{u}})}&\lesssim&\delta^{-l}|\underline{u}|^{d_{n}}M,\quad\quad\;\; 0\leq l\leq k\leq N-2,
\end{eqnarray*}
where $d_{n}=-\frac{n-1}{2}$ denotes the decay rate in $|\underline{u}|$.
\end{lemma}
\begin{proof}When $n=3$,
we shall recall the Sobolev inequality on $S_{u,\underline{u}}$, $C_{u}$ and $\underline{C}_{\underline{u}}$, see \cite{Christodoulou}. In the short pulse region, we have $|\underline{u}|\sim r$ provided $\delta$ is sufficiently small. For any smooth function $\varphi$, we have:

On $S_{u,\underline{u}}$,
\begin{equation}\label{3.5}
\|\varphi\|_{L^{\infty}(S_{u,\underline{u}})} \lesssim |\underline{u}|^{-\frac{1}{2}}(\|\varphi\|_{L^{4}(S_{u,\underline{u}})}+\|\Omega\varphi\|_{L^{4}(S_{u,\underline{u}})}).
\end{equation}

On $\underline{C}_{\underline{u}}$, if in addition $\varphi=0$ on $C_{0}$,
\begin{equation}\label{3.7}
\|\varphi\|_{L^{4}(S_{u,\underline{u}})} \lesssim |\underline{u}|^{-\frac{1}{2}}\|\underline{L}\varphi\|^{\frac{1}{2}}_{L^{2}(\underline{C}_{\underline{u}})}
(\|\varphi\|^{\frac{1}{2}}_{L^{2}(\underline{C}_{\underline{u}})}+\|\Omega\varphi\|^{\frac{1}{2}}_{L^{2}(\underline{C}_{u})}).
\end{equation}

On $C_{u}$, for $\ub \geq \ub_0$,
\begin{equation}\label{3.9}
\|\varphi\|_{L^{4}(S_{u,\underline{u}})}\leq |\underline{u}|^{-1}\left(|\underline{u}_{0}|\|L\phi\|_{L^{4}(S_{u,\underline{u}_{0}})
}+\|L\phi\|^{\frac{1}{2}}_{L^{2}(C_{u})}(\|\varphi\|^{\frac{1}{2}}_{L^{2}(C_{u})}+\|\Omega\varphi\|^{\frac{1}{2}}_{L^{2}(C_{u})})\right).
\end{equation}

Besides, we notice that
\begin{equation}\label{estimate-good-derivatives}
|\partial L\varphi|\leq  |\ub|^{-1}  |\partial \Gamma \varphi|, \quad |\partial \nablaslash \varphi|\leq    |\ub|^{-1}  |\partial \Gamma \varphi|.
\end{equation}

Based on the above inequalities, letting $\phi_{l,k}=\partial^{l}\Gamma^{k-l}\phi$, we obtain
\begin{align*}
&\quad \|\nablaslash\phi_{l,k}\|_{L^{4}(S_{u,\underline{u}})}\lesssim|\underline{u}|^{-\frac{1}{2}}
\|\underline{L}\nablaslash\phi_{l,k}\|^{\frac{1}{2}}_{L^{2}(\underline{C}_{\underline{u}})}
(\|\nablaslash\phi_{l,k}\|^{\frac{1}{2}}_{L^{2}(\underline{C}_{\underline{u}})} + \| \nablaslash \Omega \phi_{l,k}\|^{\frac{1}{2}}_{L^{2}(\underline{C}_{\underline{u}})})\\
& \lesssim|\underline{u}|^{-\frac{1}{2}} (\delta^{\frac{1}{2}-l}|\underline{u}|^{-1}M )^{\frac{1}{2}} (\delta^{1-l}M )^{\frac{1}{2}}  \lesssim\delta^{\frac{3}{4}-l}|\underline{u}|^{-1}M,\quad 0\leq l\leq k\leq N-1,
\end{align*}
and
\begin{align*}
\|\nablaslash\phi_{l,k}\|_{L^{\infty}(S_{u,\underline{u}})}&\lesssim|\underline{u}|^{-\frac{1}{2}}
(\|\nablaslash \phi_{l,k}\|_{L^{4}(S_{u,\underline{u}})}+\|\nablaslash\Omega\phi_{l,k}\|_{L^{4}(S_{u,\underline{u}})})\\
&\lesssim\delta^{\frac{3}{4}-l}|\underline{u}|^{-\frac{3}{2}}M,\quad 0\leq l\leq k\leq N-2.
\end{align*}

Viewing \eqref{3.9} and the data, we can also prove in a similar way that
\begin{align*}
\|L\phi_{l,k}\|_{L^{4}(S_{u,\underline{u}})} & \leq \delta^{1-l}|\underline{u}|^{-1}M, \quad 0\leq l \leq k\leq N-1,\\ 
\|L\phi_{l,k}\|_{L^{\infty}(S_{u,\underline{u}})} & \leq \delta^{1-l}|\underline{u}|^{-\frac{3}{2}}M,\quad 0\leq l \leq k\leq N-2. 
\end{align*}

For $\underline{L}\phi_{l,k}$, we have
\begin{align*}
& \quad \|\underline{L}\phi_{l,k}\|_{L^{4}(S_{u,\underline{u}})} \lesssim |\underline{u}|^{-\frac{1}{2}}\|\underline{L}^2\phi_{l,k}\|^{\frac{1}{2}}_{L^{2}(\underline{C}_{\underline{u}})}
(\|\underline{L}\phi_{l,k}\|^{\frac{1}{2}}_{L^{2}(\underline{C}_{\underline{u}})}
+ \| \underline{L} \Omega \phi_{l,k}\|^{\frac{1}{2}}_{L^{2}(\underline{C}_{\underline{u}})})\\
&\lesssim|\underline{u}|^{-\frac{1}{2}}(\delta^{\frac{1}{2}-1 -l}M)^{\frac{1}{2}}(\delta^{\frac{1}{2}-l}M)^{\frac{1}{2}} \lesssim|\underline{u}|^{-\frac{1}{2}}\delta^{-l} M ,\quad 0\leq l\leq k\leq N-2,
\end{align*}
and
\begin{equation*}
\|\underline{L}\phi_{l,k}\|_{L^{\infty}(S_{u,\underline{u}})}\leq \delta^{-l} |\underline{u}|^{-1}M,\quad 0\leq l\leq k\leq N-2.
\end{equation*}
Putting these estimates together, we prove the lemma with $n=3$.

For the case $n=2$, there are the following Sobolev inequalities (see \cite{Wang-Yu}) as well:

On $S_{u,\underline{u}}$, we have
\begin{equation*}
\|\varphi\|_{L^{\infty}(S_{u,\underline{u}})} \lesssim |\underline{u}|^{-\frac{1}{4}}(\|\varphi\|_{L^{4}(S_{u,\underline{u}})}+\|\Omega\varphi\|_{L^{4}(S_{u,\underline{u}})}).
\end{equation*}

On $\underline{C}_{\underline{u}}$, if in addition $\phi=0$ on $C_{0}$, we have
\begin{equation*}
\|\varphi\|_{L^{4}(S_{u,\underline{u}})} \lesssim |\underline{u}|^{-\frac{1}{4}}\|\underline{L}\varphi\|^{\frac{1}{2}}_{L^{2}(\underline{C}_{\underline{u}})}
\left(\|\varphi\|^{\frac{1}{2}}_{L^{2}(\underline{C}_{\underline{u}})}+ \|\Omega\varphi\|^{\frac{1}{2}}_{L^{2}(\underline{C}_{\underline{u}})}\right).
\end{equation*}

On $C_{u}$,  for $\ub \geq \ub_0$, we have
\begin{equation*}
\|\varphi\|_{L^{4}(S_{u,\underline{u}})}\leq |\underline{u}|^{-\frac{1}{4}}\left(|\underline{u}_{0}|^{\frac{1}{4}}\|L\varphi\|_{L^{4}(S_{u,\underline{u}_{0}})
}+\|L\varphi\|^{\frac{1}{2}}_{L^{2}(C_{u})}(\|\varphi\|^{\frac{1}{2}}_{L^{2}(C_{u})}+\|\Omega\varphi\|^{\frac{1}{2}}_{L^{2}(C_{u})})\right).
\end{equation*}
By an analogous argument as before, we can prove the lemma with $n=2$.
\end{proof}

\begin{remark}
Since $|\partial \bar{\p}\phi| \leq  |\underline{u}|^{-1} |\partial \Gamma \phi|$, we indeed have better decay estimates for  $|\partial \bar{\p}\phi|$:
\begin{equation}\label{3.16}
\begin{split}
|L\bar{\p}\phi|_{L^{\infty}(S_{u,\underline{u}})}&\leq \delta |\underline{u}|^{d_{n}-\frac{3}{2}}M,\\
|\underline{L}\bar{\p}\phi|_{L^{\infty}(S_{u,\underline{u}})}&\leq |\underline{u}|^{d_{n}-1}M,
\end{split}
\end{equation}
where $\bar{\p} = \{L, \nablaslash\}$ and $d_{n}$ is defined in Lemma \ref{lemma-apriori-estimate}.
\end{remark}

\subsection{Energy of the membrane equation}

We recall the geometric formulation for the RME \eqref{2.11}
\begin{equation*}
\Box_{g(\partial\phi) } \phi=0,
\end{equation*}
where $g(\partial\phi)$ is the Lorentzian metric depending on $\p\phi$ and $\Box_{g(\partial\phi)}$ is the associated wave operator. That is to say, our geometric background for the RME \eqref{2.11} has changed to ($\mathbb{R}^{1+n}, g(\partial\phi)$). We should note that the Minkowski null vector fields $L$ and $\underline{L}$ are not causal with respect to the membrane background ($\mathbb{R}^{1+n}, g(\partial\phi)$), since it is easy to check that
\begin{align*}
g(L,L) & =\eta(L,L)+(L\phi)^{2} = (L\phi)^{2}\geq0, \\
g(\underline{L},\underline{L}) & =\eta(\underline{L},\underline{L})+(\underline{L}\phi)^2 = (\underline{L}\phi)^2\geq0.
\end{align*}

In order to associate with non-negative energies, we shall modify $(L,\underline{L})$ and introduce
\begin{equation}\label{2.12}
\tilde{L}=L+(L\phi)^2\underline{L} \quad\text{and} \quad \tilde{\underline{L}}=\underline{L}+(\underline{L}\phi)^2L.
\end{equation}
\begin{lemma}\label{lemma-causal-tdL-tdLb}
If Lemma \ref{lemma-apriori-estimate} holds true,
$\tilde{L}$ and $\tilde{\underline{L}}$ are causal vector fields adapted to $g(\partial\phi)$.
\end{lemma}
\begin{proof}
We calculate straightforwardly to obtain
\begin{align*}
g(\tilde{L},\tilde{L}) & =-3(L\phi)^2+\left( 2L\phi\underline{L}\phi+ (L\phi)^2(\underline{L}\phi)^2\right)(L\phi)^2, \\
g(\tilde{\underline{L}},\tilde{\underline{L}}) & =-3(\underline{L}\phi)^2+\left( 2L\phi\underline{L}\phi+(L\phi)^2(\Lb\phi)^2\right)(\underline{L}\phi)^2.
\end{align*}
It is easy to see that
$g(\tilde{L},\tilde{L})\sim -3(L\phi)^2\leq0$ and $g(\tilde{\underline{L}},\tilde{\underline{L}})\sim -3(\underline{L}\phi)^2\leq0$, where we have used Lemma \ref{lemma-apriori-estimate}.
\end{proof}

Comparing with the geometric method proposed in \cite{Christodoulou1, Miao-Yu1, Speck-Hol-Luk-Wong}, we avoid constructing the optical functions adapted to the background $(\mathbb{R}^{1+n}, g(\partial\phi))$, and stick to the Minkowski double null foliation: $\{C_{u}|u\in\mathbb R\}$ and $\{\underline{C}_{\underline{u}}|\underline{u}\in\mathbb R\}$. In this way, it would be nature to consider the two vector fields $D u$ and $D \ub$, which are given explicitly by
\begin{equation}\label{grad-u-ub-expansion-1}
\begin{split}
D u &\doteq g^{\mu\nu}\p_{\mu}u \p_{\nu}= -\frac{1}{2} L - \frac{L\phi \Lb\phi}{4g} L  - \frac{(L\phi)^2}{4g} \Lb + \frac{L\phi \p^\omega\phi}{2g} \p_\omega, \\
D \ub &\doteq g^{\mu\nu}\p_{\mu}\ub \p_{\nu}=  -\frac{1}{2} \Lb - \frac{L\phi \Lb\phi}{4g} \Lb  - \frac{(\Lb\phi)^2}{4g} L + \frac{\Lb\phi \p^\omega\phi}{2g} \p_\omega.
\end{split}
\end{equation}
 Notice that, $D u$ and $D \ub$ are not necessarily null but they are causal in the background $(\mathbb{R}^{1+n}, g(\partial\phi))$:
 \begin{equation}\label{g-grad-u-ub-causal}
\begin{split}
g(D u, D u) &=g^{uu}= - \frac{(L\phi)^2}{4g} \leq0,\\
g(D \ub, D \ub) &=g^{\ub\ub}= - \frac{(\Lb\phi)^2}{4g} \le 0.\\
\end{split}
\end{equation}

Viewing the formulae \eqref{2.12}, \eqref{grad-u-ub-expansion-1} and the $L^\infty$ estimates for $\p\phi$ in Lemma \ref{lemma-apriori-estimate}, we know that $\tilde L$, $\underline{\tilde L}$ are in fact obtained by extracting the first two leading terms in $-2 D u$ and $-2D \ub$. Instead of using the genuine null vector fields in $(\mathbb{R}^{1+n}, g(\partial\phi))$ or $L$ and $\Lb$ which come from the Minkowski background, the intermediates $\tilde L, \, \underline{\tilde L}$
suffice as multipliers. This can be seen from the following lemma.

\begin{lemma}\label{lemma-energy-formula}
With the bootstrap assumption \eqref{Bootstrap-assumption}, we have $g\sim1$ and
\begin{equation}\label{3.17}
\begin{split}
-P^{u}[\phi_{k},\tilde{L}]&\sim |L\phi_{k}|^2+|L\phi|^2|\nablaslash\phi_{k}|^2+|L\phi|^4|\underline{L}\phi_{k}|^2,\\
-P^{\underline{u}}[\phi_{k},\tilde{L}]&\sim |\nablaslash\phi_{k}|^2+|\underline{L}\phi|^2|L\phi_{k}|^2+|L\phi|^2|\underline{L}\phi_{k}|^2,\\
-P^{u}[\phi_{k},\tilde{\underline{L}}]&\sim |\nablaslash\phi_{k}|^2+|\underline{L}\phi|^2|L\phi_{k}|^2+|L\phi|^2|\underline{L}\phi_{k}|^2,\\
-P^{\underline{u}}[\phi_{k},\tilde{\underline{L}}]&\sim |\underline{L}\phi_{k}|^2+|\underline{L}\phi|^2|\nablaslash\phi_{k}|^2+|\underline{L}\phi|^4|L\phi_{k}|^2,
\end{split}
\end{equation}
providing that $\delta$ is sufficiently small and here $\phi_{k}=\delta^{l}\partial^{l}\Gamma^{k-l}\phi$.
\end{lemma}

\begin{proof}
From Lemma \ref{lemma-apriori-estimate}, we know that
\begin{equation*}
|\bar\partial\phi| \lesssim \delta^{\frac{3}{4}} M \quad \text{and} \quad |\Lb\phi| \lesssim M.
\end{equation*}
Thus if $\delta$ is sufficiently small, we have
\begin{equation*}
g=1-\p_u\phi \p_{\ub}\phi + |\nablaslash \phi|^2  \sim 1 \pm \delta M^2 \sim1.
\end{equation*}

We have shown in Lemma \ref{lemma-causal-tdL-tdLb} and in \eqref{g-grad-u-ub-causal} that $\tilde{L}, \, \tilde{\Lb}$ and $-D u, \, -D \ub$ are all causal with respect to $g(\partial\phi)$. In view of the relations between the energy momentum tensor $T_{\alpha \beta}(\phi_k)$ and the currents
\begin{align*}
-P^{u}[\phi_k,\tilde{L}] = T(-D u, \tilde{L}), & \quad -P^{\ub}[\phi_{k},\tilde{L}] = T(-D \ub, \tilde{L}), \\
-P^{u}[\phi_k,\tilde{\Lb}] = T(-D u, \tilde{\Lb}), & \quad -P^{\ub}[\phi_{k},\tilde{\Lb}] = T(-D \ub, \tilde{\Lb}),
\end{align*}
the four currents above should be all non-negative by the energy condition. This is verified as below.

For the first line of \eqref{3.17}, we calculate that
\begin{align*}
P^{u}[\phi_{k},\tilde{L}]&= g^{u\gamma}\partial_{\gamma}\phi_{k}\tilde{L}\phi_{k}-\frac{1}{2}g^{\gamma\delta}\partial_{\gamma}\phi_{k}\partial_{\delta}\phi_{k}\delta^{u}_{\alpha} \tilde{L}^\alpha \\
&=g^{u\gamma}\partial_{\gamma}\phi_{k}\partial_{\underline{u}}\phi_{k}+|L\phi|^2g^{u\gamma}\partial_{\gamma}\phi_{k}\partial_{u}\phi_k
-\frac{1}{2}|L\phi|^2g^{\gamma\delta}\partial_{\gamma}\phi_{k}\partial_{\delta}\phi_{k}\\
&=g^{u\gamma}\partial_{\gamma}\phi_{k}\partial_{\ub}\phi_{k}+\frac{1}{2}|L\phi|^2g^{uu}\partial_{u}\phi_{k}\partial_{u}\phi_{k}\\
&\quad -\frac{1}{2}|L\phi|^2(g^{\underline{u}\underline{u}}(L\phi_{k})^2+2g^{\underline{u}\omega}L\phi_{k}\p_\omega\phi_{k}+g^{\omega\theta}\partial_{\omega}\phi_{k}\partial_{\theta}\phi_{k}).
\end{align*}
Substitute the metric \eqref{2.9} into the above formula,
\begin{align*}
- P^{u}[\phi_{k},\tilde{L}]& = \left( \frac{1}{2}+\frac{L\phi\underline{L}\phi}{4g}\right)  |L\phi_{k}|^2
- \frac{L\phi\nablaslash\phi}{2g}\nablaslash\phi_{k}L\phi_{k} +\frac{|L\phi|^2}{4g} \underline{L}\phi_{k} L\phi_{k} +\frac{|L\phi|^4}{8g}|\Lb\phi_{k}|^2 \\
&\quad + \frac{1}{2}|L\phi|^2\left(-\frac{|\underline{L}\phi|^2}{4g}|L\phi_{k}|^2+\frac{\underline{L}\phi\nablaslash\phi}{g}L\phi_{k}\nablaslash\phi_{k}
+|\nablaslash\phi_{k}|^2 - \frac{|\nablaslash\phi \nablaslash\phi_k|^2}{g}\right)\\
&=\frac{1}{2}|L\phi_{k}|^2 + \frac{1}{8g} |L\phi|^4 |\underline{L}\phi_{k}|^2+\frac{1}{2}|L\phi|^2|\nablaslash\phi_{k}|^2+ Er_{1} + Er_2,
\end{align*}
where in the last equality, the first three terms are the leading terms with good signs, while the error terms $Er_1, \, Er_2$ are defined by
\begin{align*}
Er_1 & =  \frac{|L\phi|^2}{4g}L\phi_{k}\underline{L}\phi_{k}, \\
Er_2 & = \frac{L\phi\underline{L}\phi}{4g}|L\phi_{k}|^2 -
\frac{|L\phi|^2|\underline{L}\phi|^2}{8g}|L\phi_{k}|^2 - \frac{L\phi\nablaslash\phi}{2g}\nablaslash\phi_{k}L\phi_{k}  \\
& \quad +  \frac{|L\phi|^2 \underline{L}\phi\nablaslash\phi}{2g} \nablaslash\phi_{k} L\phi_{k} - \frac{|L\phi|^2}{2g}  |\nablaslash\phi \nablaslash\phi_{k}|^2.
\end{align*}
We will treat them one by one. For $Er_1$, we apply the Cauchy's inequality
$$
|Er_{1}| \leq \frac{|L\phi|^4}{16g}|\underline{L}\phi_{k}|^2+\frac{1}{4g}|L\phi_{k}|^2.
$$
Therefore, $Er_{1}$ can be absorbed by the first two leading terms $\frac{1}{2}|L\phi_{k}|^2+\frac{|L\phi|^4}{8g}|\underline{L}\phi_{k}|^2$, so that these main terms are roughly $\frac{1}{4}|L\phi_{k}|^2+ \frac{|L\phi|^4}{16}|\underline{L}\phi_{k}|^2$, which still exhibit good signs.

For $Er_2$, it is easy to check that
$|Er_2| \lesssim ( \delta^{\frac{3}{4}} M + \delta M^2) |L\phi_{k}|^2+\delta^{\frac{3}{4}} M |L\phi|^2|\nablaslash\phi_{k}|^2$. If $\delta$ is small enough, they are only small perturbations of the main terms, thus we arrive at
$$
-P^{u}[\phi_{k},\tilde{L}]\sim |L\phi_{k}|^2+|L\phi|^4|\underline{L}\phi_{k}|^2+|L\phi|^2|\nablaslash\phi_{k}|^2>0.
$$

As for the second one in \eqref{3.17},
\begin{align*}
P^{\ub}[\phi_{k},\tilde{L}] &= g^{\ub \gamma}\partial_{\gamma}\phi_{k}\tilde{L}\phi_{k}-\frac{1}{2}g^{\gamma\delta}\partial_{\gamma}\phi_{k}\partial_{\delta}\phi_{k}\delta^{\ub}_{\alpha} \tilde{L}^\alpha \\
&=g^{\ub \gamma}\partial_{\gamma}\phi_{k}\partial_{\underline{u}}\phi_{k}+|L\phi|^2g^{\ub \gamma}\partial_{\gamma}\phi_{k}\partial_{u}\phi_k
-\frac{1}{2} g^{\gamma\delta}\partial_{\gamma}\phi_{k}\partial_{\delta}\phi_{k}\\
&=-\frac{1}{2}|\nablaslash \phi_k |^2  -\frac{1}{2}|L \phi |^2 |\Lb \phi_k|^2 -\frac{Q(\phi, \phi_k)}{g} \left( \p^{\ub} \phi \p_{\ub} \phi_k + |L\phi|^2 \p^{\ub} \phi \p_u \phi_k \right) + \frac{Q^2(\phi, \phi_k)}{2g} \\
&=-\frac{1}{2} |\nablaslash \phi_k |^2  -\frac{1}{2} |L \phi|^2 |\Lb \phi_k|^2 -\frac{Q(\phi, \phi_k)}{g} \left( \frac{1}{4} \p_{\ub} \phi \p_{u} \phi_k - \frac{1}{4} \p_{u} \phi \p_{\ub} \phi_k   - \frac{1}{2} \nablaslash \phi \nablaslash \phi_k \right)\\
&\quad + \frac{Q(\phi, \phi_k)}{2g} |L\phi|^2 \Lb \phi \Lb \phi_k.
\end{align*}
It follows by straightforward calculations that
\begin{align*}
- P^{\ub}[\phi_{k},\tilde{L}] &=\frac{1}{2} |\nablaslash \phi_k |^2  + \left( \frac{1}{2} - \frac{1}{8g}\right) |L \phi |^2 |\Lb \phi_k|^2 + \frac{1}{8g} |\Lb \phi |^2 |L \phi_k|^2  + Er_3, \\
Er_3 &= \frac{\nablaslash \phi \nablaslash \phi_k}{2g}  L \phi \Lb \phi_k - \frac{1}{2g} |\nablaslash \phi \nablaslash \phi_k|^2  - \frac{Q(\phi, \phi_k)}{2g} L\phi \Lb \phi L\phi \Lb \phi_k.
\end{align*}
We can check that
$|Er_3 | \lesssim \delta^{\frac{3}{4}} M |\nablaslash \phi_{k}|^2 + \left( \delta^{\frac{3}{4}} M + \delta M^2 \right) |L\phi|^2 |\Lb\phi_{k}|^2+\delta M^2 |\Lb\phi|^2|L\phi_{k}|^2$. Thus,
$$
-P^{\ub}[\phi_{k},\tilde{L}]\sim |\nablaslash \phi_{k}|^2+ |L \phi |^2 |\Lb \phi_k|^2 +  |\Lb \phi |^2 |L \phi_k|^2.
$$

The third and fourth lines in \eqref{3.17} can be proved in an analogous way.
\end{proof}

Before proceeding to the energy estimates, we need to investigate the double null structure of the higher order RME in the null coordinate system $(u,\underline{u},\theta)$.

\subsection{Higher order equations and double null condition}\label{subsec-The high order equation and double null condition}
\subsubsection{The higher order RME}
In this section, we will deduce the higher order equations of RME which arises from commuting with the Klainerman vector fields $Z$ for $N$ times.
\begin{lemma}\label{lemma-high-order-eq}
Commuting the RME \eqref{1.1} with the vector field $Z$ for $N$ times, we have
\begin{equation}\label{3.18}
\begin{split}
& \Box_{g(\p\phi)}Z^{N}\phi   = (1+Q)^{-1} g^{\mu \nu} \partial_{\nu}Q \partial_{\mu}Z^N \phi \\
+ & \sum_{ \substack{i_1+\cdots +i_k \leq N, \\ i_1,\cdots, i_k <N, \\ 3 \leq k \leq 2N+3} } F_{i_1, \cdots, i_k}(Q)Q(Z^{i_1} \phi, Q( Z^{i_2}\phi, Z^{i_3} \phi) )\\
&\quad \quad \quad \quad \quad \quad  \times Q(Z^{i_4} \phi,Z^{i_5} \phi) \cdots Q( Z^{i_{k-1}}\phi, Z^{i_k} \phi),
\end{split}
\end{equation}
where $F_{i_1, \cdots i_k}(Q)$ is a fractional function on $Q = Q(\phi,\phi)=\eta^{\alpha\beta}\p_{\alpha} \phi \p_{\beta} \phi$ and $k$ is odd, and if $k = 3$, then this is to be interpreted as the factor $Q(Z^{i_4} \phi,Z^{i_5} \phi)$ being absent.

\end{lemma}
\begin{proof}
We rewrite the RME \eqref{1.1} as below
\begin{equation}\label{3.19}
\Box\phi-\frac{1}{2} \frac{Q(\phi,Q(\phi,\phi))}{1+Q(\phi,\phi)}=0.
\end{equation}
Applying $Z$ to \eqref{3.19} for $N$ times and recalling that $[\Box,Z]=0$ (except for $[\Box,S]=-2\Box$), we obtain
\begin{equation}\label{3.20}
\begin{split}
&\Box Z^{N}\phi-\frac{Q(\phi,Q(Z^{N}\phi,\phi))}{1+Q}-\frac{Q(Z^{N}\phi,Q(\phi,\phi))}{2(1+Q)} + \frac{Q(\phi,Q(\phi, \phi)) Q(Z^{N}\phi, \phi)}{(1+Q)^2}\\
&= \sum_{ \substack{i_1+\cdots +i_k \leq N, \\ i_1,\cdots, i_k <N, \\ 3 \leq k \leq 2N+3} } F_{i_1, \cdots, i_k}(Q)Q(Z^{i_1} \phi, Q( Z^{i_2}\phi, Z^{i_3} \phi) )\\
&\quad \quad \quad \quad \quad \quad  \times Q(Z^{i_4} \phi,Z^{i_5} \phi) \cdots Q( Z^{i_{k-1}}\phi, Z^{i_k} \phi),
\end{split}
\end{equation}
Here we have made use of the fact that $Q$ is the basic null form $Q_0$, and the commuting formulae for $Z$ and $Q_0$ \eqref{commu-Q-Gamma}.
Equivalently if $Z$ is a Poincar\'{e} generator (the Poincar\'{e} group being the isometry group of $\eta$), then $[Z, Q_0] = 0$; if $Z=S$ is the dilatation, then $[S, Q_0] = -2 Q_0$.
Expanding $Q$, and noticing that
\begin{equation}\label{3.21}
\Box Z^{N}\phi-\frac{\partial^{\mu}\phi\partial^{\nu}\phi\partial^{2}_{\mu\nu}Z^N\phi}{1+Q}=g^{\mu\nu}\partial_{\mu}\partial_{\nu}Z^{N}\phi = \Box_{g(\p\phi)}Z^{N}\phi,
\end{equation}
where in the second equality,  the wave coordinate condition \eqref{wave-coordinate} is used, we obtain that the first line of \eqref{3.20} is identical to
\begin{equation*}
\Box_{g(\p\phi)}Z^{N}\phi - \frac{\partial^{\mu}Q \partial_{\mu}Z^N\phi}{1+Q} + \frac{\p^\nu\phi \p_\nu Q  \p_\mu\phi  \p^\mu Z^N\phi}{(1+Q)^2} =\Box_{g(\p\phi)}Z^{N}\phi - \frac{g^{\nu\mu} \p_\nu Q \p_\mu Z^N\phi}{1+Q}.
\end{equation*}
Thus we achieve \eqref{3.18}.
\end{proof}
\begin{remark}
The higher order RME \eqref{3.18} is a geometric wave equation with inhomogeneous terms that are at least cubic. What is more, these cubic non-linearities taking the double null form $Q(\varphi,Q(\chi,\psi))$ contain at least four derivatives. Let us stress here that in the high order RME, the null form $Q$ takes only the first type of the basic null forms $Q_0$ \eqref{basic null forms}. This would be crucial in Section \ref{Double null condition in null coordinates}.
\end{remark}

\subsubsection{The double null condition}\label{Double null condition in null coordinates}
We are ready to study the structure of the nonlinear term $Q(\phi,Q(\psi,\chi))$ in null coordinates.
The projection of $\partial_{i},$ $i=1, \cdots, n$ onto the tangent space of the sphere $S^{n-1}$ would be denoted by
\begin{equation*}
\bar{\partial_{i}}=\partial_{i}-\omega_{i}\omega^{j}\partial_{j}=\partial_{i}-\omega_{i}\partial_{r}, \quad \omega_{i}=\frac{x_{i}}{r}.
\end{equation*}
Note that, $\{\bar{\partial_{i}}, \, i=1,\cdots, n\}$ composes of a global frame on $S^{n-1}$ and
$\bar{\partial_{i}}=r^{-1}\omega^{j}\Omega_{ij}.$

Let us introduce a lemma, which is originated from \cite{Lindblad-Rodnianski1}.
\begin{lemma}\label{lemma-L-R}
If $k^{\alpha\beta}$ is a symmetric tensor, then
\begin{equation*}
\begin{split}
k^{\alpha\beta}\partial_{\alpha}\partial_{\beta}&=k_{LL}\underline{L}^2+2k_{L\underline{L}}L\underline{L}+k_{\underline{L}\underline{L}}L^2
+k^{ij}\bar{\partial_{i}}\bar{\partial_{j}}\\
&+2k_{L}^{j}\bar{\partial_{j}}\underline{L}+2k_{\underline{L}}^{j}\bar{\partial_{j}}L+r^{-1}\bar{tr}k\underline{L}+r^{-1}\bar{tr}k L
-r^{-1}k^{ij}\omega_{i}\bar{\partial_{j}},
\end{split}
\end{equation*}
where $\bar{tr}k=\bar{\delta}^{ij} k_{ij}$, with $\bar{\delta}^{ij}= \delta^{ij} - \omega^i \omega^j$ being the projection operator, and the sum is over $i,j =1, \cdots, n$.
\end{lemma}
We recall the definition $Q(\phi,Q(\psi,\chi))=\partial^{\alpha}\phi\partial^{\beta}\psi \partial_{\alpha}\partial_{\beta}\chi+
\partial^{\alpha}\phi\partial^{\beta}\chi \partial_{\alpha}\partial_{\beta}\psi$.
For notational convenience, we define the symmetrization
\begin{equation}\label{def-symmetri-double-null}
S^{\alpha \beta}(\phi, \psi)=\frac{1}{2}( \partial^{\alpha}\phi\partial^{\beta}\psi + \partial^{\alpha}\psi\partial^{\beta}\phi).
\end{equation}
Then
\begin{equation}\label{def-Q-double-null}
Q(\phi,Q(\psi,\chi))=
S^{\alpha \beta}(\phi, \psi) \partial_{\alpha}\partial_{\beta}\chi +S^{\alpha \beta}(\phi, \chi)\partial_{\alpha}\partial_{\beta}\psi.
\end{equation}
In what follows, we will further utilize the decomposition $\p_i = \bar{\p}_i +\omega^i \p_r$, and reformulate $Q(\phi,Q(\psi,\chi))$ so that it exhibits the null structure apparently in null coordinates.
\begin{lemma}\label{lemma-double}
In terms of the null frame $\{ L,\underline{L},\bar{\partial_{i}}, i=1,\cdots, n \}$, we have
\begin{equation}\label{double-null}
\begin{split}
S^{\alpha \beta}(\phi, \psi) \partial_{\alpha}\partial_{\beta}\chi
&= S_{LL} \underline{L}^2\chi+2S_{L \Lb} L\underline{L}\chi+S_{\Lb \Lb} L^2\chi
+\sum_{i,j=1}^{n} \frac{1}{2}(\bar{\partial_{i}}\phi\bar{\partial_{j}}\psi + \bar{\partial_{i}}\psi\bar{\partial_{j}}\phi ) \bar{\partial_{i}}\bar{\partial_{j}}\chi\\
&\quad +\sum_{i=1}^{n} (L\phi\bar{\partial_{i}}\psi+ L\psi\bar{\partial_{i}}\phi) \bar{\partial_{i}}\underline{L}\chi
+\sum_{i=1}^{n} (\underline{L}\phi\bar{\partial_{i}}\psi + \underline{L}\psi\bar{\partial_{i}}\phi) \bar{\partial_{i}}L\chi \\
&\quad + \frac{1}{2r}  \sum_{i=1}^{n} (\bar{\partial_{i}}\phi\bar{\partial_{i}}\psi + \bar{\partial_{i}}\psi\bar{\partial_{i}}\phi)  \partial_{r}\chi -\frac{1}{r} \sum_{i=1}^{n} (\bar{\partial_{i}}\phi\partial_{r}\psi + \bar{\partial_{i}}\psi\partial_{r}\phi) \bar{\partial_{i}}\chi.
\end{split}
\end{equation}
\end{lemma}

\begin{proof}
Replacing $k^{\alpha\beta}$ by $S^{\alpha \beta}(\phi, \psi)$ in Lemma \ref{lemma-L-R}, we obtain
\begin{align*}
S^{\alpha \beta}(\phi, \psi) \partial_{\alpha}\partial_{\beta}\chi &= S_{LL} \underline{L}^2\chi+2S_{L \Lb} L\underline{L}\chi+S_{\Lb \Lb} L^2\chi+\frac{1}{2}(\partial^{i}\phi\partial^{j}\psi +\partial^{i}\psi\partial^{j}\phi ) \bar{\partial_{i}}\bar{\partial_{j}}\chi\\
&
\quad +(L\phi\partial^{j}\psi + L\psi\partial^{j}\phi)\bar{\partial_{j}}\underline{L}\chi+
(\underline{L}\phi\partial^{j}\psi + \underline{L}\psi\partial^{j}\phi) \bar{\partial_{j}}L\chi\\
&\quad +
\frac{1}{2r} \sum_{i=1}^{n}(\bar{\partial_{i}}\phi\bar{\partial_{i}}\psi + \bar{\partial_{i}}\psi\bar{\partial_{i}}\phi) \partial_{r}\chi
-\frac{1}{2r}( \partial^{i}\phi\partial^{j}\psi+ \partial^{i}\psi\partial^{j}\phi ) \omega_{i}\bar{\partial_{j}}\chi.
\end{align*}
A key observation underlying the proof is the following identity
$$\omega^{j}\bar{\partial_{j}}=\omega^{j}(\partial_{j}-\omega_{j}\partial_{r})=\partial_{r}-\partial_{r}=0.$$
Based on this cancellation, we can refine the last term in $S^{\alpha \beta}(\phi, \psi) \partial_{\alpha}\partial_{\beta}\chi $,
\begin{align*}
r^{-1} \partial^{i}\phi\partial^{j}\psi\omega_{i}\bar{\partial_{j}}\chi&=
r^{-1}\sum_{i,j}(\bar{\partial_{i}}+\omega_{i}\partial_{r})\phi(\bar{\partial_{j}}+\omega_{j}\partial_{r})
\psi\omega_{i}\bar{\partial_{j}}\chi\\
&=\sum_{j}r^{-1}\partial_{r}\phi\bar{\partial_{j}}\psi\bar{\partial_{j}}\chi.
\end{align*}
In the same way, there are
\begin{equation*}
L\phi\partial^{j}\psi\bar{\partial_{j}}\underline{L}\chi=\sum_{i}L\phi\bar{\partial_{i}}\psi\bar{\partial_{i}}\underline{L}\chi, \quad \underline{L}\phi\partial^{j}\psi\bar{\partial_{j}}L\chi=\sum_{i}\underline{L}\phi\bar{\partial_{i}}\psi\bar{\partial_{i}}L\chi,
\end{equation*}
and
\begin{align*}
 \partial^{i}\phi \partial^{j}\psi\bar{\partial_{i}}\bar{\partial_{j}}\chi 
&=\sum_{i,j}\bar{\partial_{i}}\phi\bar{\partial_{j}}\psi\bar{\partial_{i}}\bar{\partial_{j}}\chi
-\sum_{i,j}\bar{\partial_{i}}\phi\partial_{r}\psi(\bar{\partial_{i}}\omega_{j})\bar{\partial_{j}}\chi\\
&=\sum_{i,j}\bar{\partial_{i}}\phi\bar{\partial_{j}}\psi\bar{\partial_{i}}\bar{\partial_{j}}\chi
-\sum_{i}r^{-1}\bar{\partial_{i}}\phi\partial_{r}\psi\bar{\partial_{i}}\chi,
\end{align*}
where we have used $\bar{\partial_{i}}\omega_{j}=\partial_{i}\omega_{j}=r^{-1}(\delta_{ij}-\omega_{i}\omega_{j}).$
Thus the lemma follows by considering all the above calculations.
\end{proof}
\begin{remark}
From Lemma \ref{lemma-double}, we can see that the first two lines of \eqref{double-null} both have signature $0$ and they contain four derivatives, while in the last line, only three derivatives survive, and their signatures are $-1$, which is seemingly bad. However, these terms involve at least two angular derivatives, which can afford extra powers of $\delta$, and there is also an additional factor $r^{-1}\sim |\underline{u}|^{-1}$ which will provide more decay. Generally speaking, the double null form in the null coordinates still exhibits good structures. Now we can reinterpret the double null form as: Among the four derivatives or three derivatives in $Q(\phi,Q(\psi,\chi))$ (in null coordinates), there exist at least two good ($\{L,\nablaslash\}$) derivatives.
\end{remark}

Now we are ready to prove the \emph{a priori} estimate.
\subsection{Energy estimates}\label{subsec-E-E}
We remark that there is a hierarchy in the energy estimates. Namely, the estimates for $\underline{E}_{\leq N}$ should be done first, and then this information will help to estimate $E_{\leq N}$ later.
\subsubsection{Estimates for $\underline{E}_{\leq N}$}\label{subsec-E-N}
We apply the energy identity \eqref{2.19} to the higher order RME \eqref{3.18}, with $\phi$ being replaced by $\phi_{k}=\delta^{l}\partial^{l}\Gamma^{k-l}\phi$, and take $\xi=\tilde{\underline{L}}$ to obtain
\begin{equation}\label{3.22}
\begin{split}
& \int_{C_{u}}-\sqrt{g}P^{u}[\phi_{k},\tilde{\underline{L}}]r^{n-1}d\underline{u}d\sigma_{S^{n-1}}+
\int_{\underline{C}_{\underline{u}}}-\sqrt{g}P^{\underline{u}}[\phi_{k},\tilde{\underline{L}}]r^{n-1}dud\sigma_{S^{n-1}}\\
&=  \int_{\Sigma_1}-\sqrt{g}P^{t}[\phi_{k},\tilde{\underline{L}}]r^{n-1}drd\sigma_{S^{n-1}} -\iint_{\mathcal{D}_{u, \ub}}\partial_{\alpha}(\sqrt{g}P^{\alpha})d^{n+1}x.
\end{split}
\end{equation}

In view of Lemma \ref{lemma-energy-formula}, we know that the energies in the first line of \eqref{3.22} bound $\delta \underline{E}^{2}_{\leq N}(u,\underline{u})$.
Thus, \eqref{3.22} turns into the energy inequality
\begin{equation}\label{3.23}
\delta \underline{E}^{2}_{\leq N}(u,\underline{u})\lesssim \delta I_N^{2}(\psi_{0},\psi_{1})+\iint_{\D_{u, \ub}}|\partial_{\alpha}(\sqrt{g} P^{\alpha})|d^{n+1}x.
\end{equation}

In the following, we shall estimate the spacetime integral $\iint_{\D_{u, \ub}}|\partial_{\alpha}(\sqrt{g} P^{\alpha})|d^{n+1}x$ in detail.
Taking $\xi=\tilde{\underline{L}}$ in Lemma \ref{lem:2.4}, we derive the divergence of the associated current $P^{\alpha}$
\begin{equation}\label{3.24}
\partial_{\alpha}(\sqrt{g}P^{\alpha})=\sqrt{g}\left(\Box_{g(\p\phi)}\phi_{k}\tilde{\underline{L}}\phi_{k}+T^{\alpha}_{\beta} \cdot \partial_{\alpha}\tilde{\underline{L}}^{\beta}\right)-\frac{1}{2}
\tilde{\underline{L}}(\sqrt{g}g^{\gamma\rho})\partial_{\gamma}\phi_{k}\partial_{\rho}\phi_{k}.
\end{equation}

Next, we calculate the deformation tensor for $\tilde{\underline{L}}$ (similar calculation holds for $\tilde L$)
\begin{equation}\label{3.25}
\begin{split}
\partial_{\alpha}\tilde{\underline{L}}^{\beta}&=\p_{\alpha}\underline{L}^{\beta}+
\partial_{\alpha}(\underline{L}\phi)^2L^{\beta}+|\underline{L}\phi|^2\p_{\alpha}L^{\beta}\\
&=-\frac{1}{r}\gamma^{\beta}_{\alpha}+\partial_{\alpha}(\underline{L}\phi)^2L^{\beta}+\frac{|\underline{L}\phi|^2}{r}\gamma^{\beta}_{\alpha},
\end{split}
\end{equation}
where $\gamma_{\alpha \beta}$ denotes the induced metric from $\eta_{\mu\nu}$ on $S_{u,\underline{u}}$. Thus,
\begin{equation}\label{3.26}
T^{\alpha}_{\beta}\cdot\partial_{\alpha}\tilde{\underline{L}}^{\beta}=\left(g^{\alpha\gamma}\partial_{\gamma}\phi_{k}\partial_{\beta}\phi_{k}
-\frac{1}{2}g^{\gamma\delta}\partial_{\gamma}\phi_{k}\partial_{\delta}\phi_{k}\delta^{\alpha}_{\beta}\right)\partial_{\alpha}\tilde{\underline{L}}^{\beta} :=A+B,
\end{equation}
where $A=g^{\alpha\gamma}\partial_{\gamma}\phi_{k}\partial_{\beta}\phi_{k}\partial_{\alpha}\tilde{\underline{L}}^{\beta}$ and $B=-\frac{1}{2}g^{\gamma\delta}\partial_{\gamma}\phi_{k}\partial_{\delta}\phi_{k}\delta_{\beta}^{\alpha}\partial_{\alpha}\tilde{\underline{L}}^{\beta}$.  Further calculations show that
\begin{align}
A &=g^{\omega\gamma}\partial_{\gamma}\phi_{k}\partial_{\theta}\phi_{k}
\left(-\frac{1}{r}\gamma^{\theta}_{\omega}+\frac{1}{r}|\underline{L}\phi|^2\gamma^{\theta}_{\omega}\right) +g^{\alpha \gamma}\partial_{\gamma}\phi_{k}L\phi_{k} \p_\alpha (\underline{L}\phi)^2, \label{3.27} \\
B &=-\left(-\frac{n-1}{2r}+\frac{(n-1)|\underline{L}\phi|^2}{2r}+L\underline{L}\phi\underline{L}\phi\right)
g^{\gamma\delta}\partial_{\gamma}\phi_{k}\partial_{\delta}\phi_{k}. \label{3.28}
\end{align}
Then \eqref{3.26} could be written explicitly as
\begin{equation}\label{3.29}
\begin{split}
T^{\alpha}_{\beta}\cdot\partial_{\alpha}\tilde{\underline{L}}^{\beta}&=\left(\frac{1}{r}-\frac{|\underline{L}\phi|^2}{r}\right) \left(\frac{n-1}{2}
g^{\gamma\delta}\partial_{\gamma}\phi_{k}\partial_{\delta}\phi_{k}-g^{\omega\gamma}\partial_{\gamma}\phi_{k}\partial_{\omega}\phi_{k}\right)\\
& \quad +g^{\alpha \gamma}\partial_{\gamma}\phi_{k}L\phi_{k} \p_\alpha (\underline{L}\phi)^2 -
L\underline{L}\phi\underline{L}\phi g^{\gamma\delta}\partial_{\gamma}\phi_{k}\partial_{\delta}\phi_{k}\\
& := H_{1}+H_{2}+H_3.
\end{split}
\end{equation}

The error terms associated to the connection coefficients can also be calculated as below
\begin{eqnarray*}
\tilde{\underline{L}}(\sqrt{g}g^{\gamma\rho})\partial_{\gamma}\phi_{k}\partial_{\rho}\phi_{k}&=&
\frac{\tilde{\underline{L}}Q}{2\sqrt{g}}g^{\gamma\rho}\partial_{\gamma}\phi_{k}\partial_{\rho}\phi_{k}
-\frac{\tilde{\underline{L}}(\partial^{\gamma}\phi\partial^{\rho}\phi)}{\sqrt{g}}\partial_{\gamma}\phi_{k}\partial_{\rho}\phi_{k}\nonumber\\
&=&\frac{Q(\underline{L}\phi,\phi)+|\underline{L}\phi|^2Q(L\phi,\phi)}{\sqrt{g}}g^{\gamma\rho}\partial_{\gamma}\phi_{k}\partial_{\rho}\phi_{k}\nonumber\\
&-&\frac{2}{\sqrt{g}}(\partial^{\gamma}\underline{L}\phi+|\underline{L}\phi|^2\partial^{\gamma}L\phi)\partial^{\rho}\phi\partial_{\gamma}\phi_{k}\partial_{\rho}\phi_{k}\nonumber\\
&-&\frac{1}{\sqrt{g}}(\partial_{\alpha}\underline{L}^{\beta}+|\underline{L}\phi|^2\partial_{\alpha}L^{\beta})\partial^{\alpha}\phi\partial_{\beta}\phi
g^{\gamma\rho}\partial_{\gamma}\phi_{k}\partial_{\rho}\phi_{k}\nonumber\\
&+&\frac{2}{\sqrt{g}}(\partial^{\gamma}\underline{L}^{\beta}+|\underline{L}\phi|^2\partial^{\gamma}L^{\beta})\partial_{\beta}\phi\partial^{\rho}\phi
\partial_{\gamma}\phi_{k}\partial_{\rho}\phi_{k}.\nonumber
\end{eqnarray*}
Substituting $\p_\alpha \Lb^\beta = - \frac{1}{r} \gamma_\alpha^\beta$ and $\p_\alpha L^\beta = \frac{1}{r} \gamma_\alpha^\beta$ into the above formula, we have
\begin{eqnarray}\label{3.30}
\tilde{\underline{L}}(\sqrt{g}g^{\gamma\rho})\partial_{\gamma}\phi_{k}\partial_{\rho}\phi_{k}&=&
\frac{Q(\underline{L}\phi,\phi)+|\underline{L}\phi|^2Q(L\phi,\phi)}{\sqrt{g}}g^{\gamma\rho}\partial_{\gamma}\phi_{k}\partial_{\rho}\phi_{k}\nonumber\\
&-&\frac{2}{\sqrt{g}} (\partial^{\gamma}\underline{L}\phi+|\underline{L}\phi|^2\partial^{\gamma}L\phi)\partial^{\rho}\phi\partial_{\gamma}\phi_{k}\partial_{\rho}\phi_{k}\nonumber\\
&+&\frac{1-|\underline{L}\phi|^2}{r\sqrt{g}} \gamma_{\omega\theta}\partial^{\omega}\phi\partial^{\theta}\phi g^{\gamma\rho}\partial_{\gamma}\phi_{k}\partial_{\rho}\phi_{k}\nonumber\\
&-&\frac{2-2|\underline{L}\phi|^2}{r\sqrt{g}}\gamma_{\omega\theta}\partial^{\omega}\phi\partial^{\theta}\phi_{k}\partial^{\rho}\phi\partial_{\rho}\phi_{k}\nonumber\\
&:=&I_1+I_2+I_3+I_4.
\end{eqnarray}
We remark that these calculations are all done in $(x^\mu, \mu = 0, \cdots, n)$ coordinates, not the null coordinates, since we have used the wave coordinate gauge in the higher order RME \eqref{3.18}.
As a consequence, the spacetime integral in \eqref{3.23} could be classified as
\begin{equation}\label{3.31}
\iint_{\mathcal{D}_{u, \ub}}\sqrt{g}\Big[|\Box_{g(\p\phi)}\phi_{k}\tilde{\underline{L}}\phi_{k}|+\sum_{1 \leq i \leq 3}|H_{i}|+\sum_{1 \leq j \leq 4}|I_{j}|\Big]d^{n+1}x.
\end{equation}

For $H_{1}$, it could be further expanded as
\begin{equation}\label{H1-n-Lb}
\begin{split}
 H_1 &= \sqrt{g}\left(\frac{1}{r}- \frac{|\Lb\phi|^2}{r}\right) \Big[ \frac{n-1}{2} g^{ab}\partial_{a}\phi_{k}\partial_{b}\phi_{k} \\
  &\quad +(n-2) g^{\omega a}\partial_{a}\phi_{k}\partial_{\omega}\phi_{k} + \frac{n-3}{2} g^{\omega \theta}\partial_{\theta}\phi_{k}\partial_{\omega}\phi_{k} \Big],
\end{split}
\end{equation}
where $a, b$ denote the null coordinates $u, \,\ub$.
When $n=3$, the worst term in \eqref{H1-n-Lb} is the quadratic one $\frac{1}{r}L\phi_{k}\underline{L}\phi_{k}$. We can bound it by
\begin{equation}\label{3.32}
\begin{split}
& \quad \iint_{\mathcal{D}_{u, \ub}}\sqrt{g}\Big|\frac{1}{r}L\phi_{k}\underline{L}\phi_{k}\Big|d^{n+1}x \\
& \lesssim  \int_0^u \|L\phi_{k}\|^2_{L^{2}(C_{u^\prime})} d u^\prime \int_{\ub_0}^u \frac{1}{|\ub^\prime|^2} \|\underline{L}\phi_{k}\|^2_{L^{2}(\Cb_{\ub^\prime})} d \ub^\prime  \\
& \lesssim  u^{\frac{1}{2}} \delta M \ub_0^{-\frac{1}{2}}  \delta^{\frac{1}{2}} M \lesssim \delta^{2} M^2.
\end{split}
\end{equation}
When $n=2$, there is an additional quadratic term $\frac{1}{r}\nablaslash\phi_{k}\nablaslash\phi_{k}$ in \eqref{H1-n-Lb}, which could be estimated in the same way as \eqref{3.32}. In fact, for energy associated to $\tilde{\underline{L}}$, we only expect weak estimate, namely, $\underline{E}^2_{\leq N}\lesssim \delta$. Hence, although $\frac{1}{r}L\phi_{k}\underline{L}\phi_{k}$ and $\frac{1}{r}\nablaslash\phi_{k}\nablaslash\phi_{k}$ are quadratic, the inherently null structures help to provide adequate smallness to close the energy argument.

As for $H_i, \, i=2, 3$, we note that they all take the double null forms, having signatures $-1$ (see the definition in Section \ref{section null form}). In $H_{2}$, the leading term is $(L\phi_{k})^2\underline{L}\phi\underline{L}^2\phi$, which could be estimated as
\begin{equation}\label{3.33}
\begin{split}
& \quad \iint_{\mathcal{D}_{u, \ub}}|(L\phi_{k})^2\underline{L}\phi\underline{L}^2\phi|d^{n+1}x\\
& \lesssim \int_0^u \|L\phi_{k}\|^2_{L^{2}(C_{u^\prime})}  \|\underline{L}\phi \underline{L}^2\phi\|_{L^{\infty}(C_{u^\prime})} d u^\prime  \\
& \lesssim \int_0^u (\delta M)^2  \delta^{-1} M^2 d u^\prime \lesssim \delta^{2} M^4.
\end{split}
\end{equation}
By a similar procedure, we perform $L^2$ on the terms involving the highest order derivative $\p \phi_k$, and $L^\infty$ on the terms of first or second order derivatives, then
\begin{equation}\label{3.34}
\iint_{\mathcal{D}_{u, \ub}}\sqrt{g}(|H_{2}| + |H_{3}| )d^{n+1}x \lesssim \delta^{2} M^4.
\end{equation}

The $I_{1}$ and $I_{2}$ having signatures $-1$, also take the null forms. Therefore, analogous to $H_{2}$ and $H_{3}$, we have
\begin{equation}\label{3.36}
\iint_{\mathcal{D}_{u, \ub}}\sqrt{g}(|I_{1}|+|I_{2}|)d^{n+1}x \lesssim \delta^{2}M^4.
\end{equation}

For $I_{3}$ and $I_{4}$, with the coefficient $\frac{1-|\underline{L}\phi|^2}{r\sqrt{g}} \sim |\ub|^{-1}$, they have signatures $0$, which are good enough to afford sufficient smallness. Typically, they could be controlled as
\begin{equation}\label{3.37}
\begin{split}
& \quad \iint_{\mathcal{D}_{u, \ub}}\sqrt{g} \left( |I_3|+|I_4| \right) d^{n+1}x \\
& \lesssim \|\nablaslash\phi L\phi\|_{ L^{\infty} (\mathcal{D}_{u, \ub} )} \left( \int_{0}^{u} \|\nablaslash\phi_{k}\|^2_{L^{2}(C_{u^\prime})} d u^\prime \right)^{\frac{1}{2}} \left( \int_{\ub_0}^{\ub} \frac{1}{\underline{u}^{\prime 2} }  \|\underline{L}\phi_{k}\|^2_{L^{2}(\Cb_{\ub^\prime})}  d \ub^\prime \right)^{\frac{1}{2}} \\
& \lesssim \delta^{\frac{3}{4}} M^2 u^{\frac{1}{2}} \delta^{\frac{1}{2}} M \delta^{\frac{1}{2}} M \lesssim  \delta^{2+\frac{1}{4}} M^4.
\end{split}
\end{equation}

At last, we consider the source term $\Box_{g(\p\phi)}\phi_{k} \cdot \tilde{\underline{L}}\phi_{k}$. By Lemma \ref{lemma-high-order-eq}, we have
\begin{equation}\label{3.39}
\begin{split}
&\iint_{\mathcal{D}_{u, \ub}}\sqrt{g}|\Box_{g(\p\phi)}\phi_{k}\tilde{\underline{L}}\phi_{k}|d^{n+1}x=\iint_{\mathcal{D}_{u, \ub}}\sqrt{g}|(1+Q)^{-1} g^{\mu \nu} \partial_{\nu}Q \partial_{\mu} \phi_k \\
 & \quad +\sum_{ \substack{i_1+\cdots +i_j \leq k, \\ i_1,\cdots, i_j <k, \\ 3\leq j\leq 2k+3}} F_{i_1, \cdots, i_j}(Q)Q( \phi_{i_{1}}, Q( \phi_{i_{2}},  \phi_{i_{3}}) ) \cdots Q( \phi_{i_{j-1}},  \phi_{i_{j}})\\
 &\quad \quad \quad \quad \quad \quad  \times Q(Z^{i_4} \phi,Z^{i_5} \phi) \cdots Q( Z^{i_{k-1}}\phi, Z^{i_k} \phi)
  \tilde{\underline{L}}\phi_{k}|d^{n+1}x\\
&:=J_{1}+J_{2}.
\end{split}
\end{equation}

Obviously, the leading term in $J_{1}$ is
\begin{equation}\label{3.40}
 \iint_{\mathcal{D}_{u, \ub}}|\eta^{\mu \nu} \partial_{\nu}Q \partial_{\mu} \phi_k \underline{L}\phi_{k}|d^{n+1}x,
\end{equation}
where $\eta^{\mu\nu}\partial_{\nu}Q\partial_{\mu}\phi_{k}$ takes the double null form. In view of Section \ref{Double null condition in null coordinates}, $\eta^{\mu\nu}\partial_{\nu}Q\partial_{\mu}\phi_{k}$ or equivalently $Q(\phi_{k},Q(\phi,\phi))$, contains two types:
\begin{align}\label{J-11}
J_{11}=
\left\{\begin{array}{ccc}
\underline{L}\phi L^2\phi \underline{L}\phi_{k},&  L\underline{L}\phi L\phi \underline{L}\phi_{k},& L\nablaslash\phi\nablaslash\phi\underline{L}\phi_{k}\\
\underline{L}L\phi \underline{L}\phi L\phi_{k},& \underline{L}^2\phi L\phi L\phi_{k},&
\underline{L}\nablaslash\phi\nablaslash\phi L\phi_{k}\\
\nablaslash\underline{L}\phi L\phi \nablaslash\phi_{k},&\nablaslash L\phi \underline{L}\phi \nablaslash\phi_{k},& \nablaslash^2\phi\nablaslash\phi\nablaslash\phi_{k}
\end{array}\right\},
\end{align}
 which we shall refer to as the $J_{11}$ type, and the $J_{12}$ type is
\begin{equation}\label{J-12}
J_{12}= \left\{\begin{array}{cc}
r^{-1}(\nablaslash\phi)^2(L\phi_k-\underline{L}\phi_{k}) ,&r^{-1}(L\phi-\underline{L}\phi) \nablaslash\phi
\nablaslash\phi_{k}
\end{array}\right\}.
\end{equation}

Note that, $J_{11}$ has signature $0$. We have the estimate for the leading term in $J_{11} \cdot \Lb\phi_k$,
\begin{equation}\label{3.41}
\begin{split}
&\quad \iint_{\mathcal{D}_{u, \ub}}|\Lb^2\phi L\phi L\phi_k \underline{L}\phi_{k} |d^{n+1}x\\
& \lesssim  \left( \int_{0}^{u} \|L\phi_{k}\|^2_{L^{2}(C_{u^\prime})} d u^\prime \right)^{\frac{1}{2}} \left( \int_{\ub_0}^{\ub}  \|\underline{L}^2 \phi L\phi\|^2_{ L^{\infty} (\Cb_{\ub^\prime})} \|\underline{L}\phi_{k}\|^2_{L^{2}(\Cb_{\ub^\prime})}  d \ub^\prime \right)^{\frac{1}{2}} \\
& \lesssim  u^{\frac{1}{2}} \delta M \left( \int_{\ub_0}^{\ub}  \left(  \delta^{-1} M \delta M |\ub^\prime|^{2d_{n}- \frac{1}{2}} \right)^2  \delta M^2  d \ub^\prime \right)^{\frac{1}{2}} \lesssim  \delta^{2} M^4.
\end{split}
\end{equation}
The other terms in $J_{11} \cdot \Lb\phi_k$ can be estimated similarly.

For $J_{12}$, it has an additional factor $r^{-1} \sim |\ub|^{-1}$  and it contains two angular derivatives, which will bring extra $\delta$ power. This can be verified sketchily from the following comparison: Both of $L\phi\underline{L}\phi$ and $|\nablaslash\phi|^2$ have signatures 0. However, $|\nablaslash\phi|^2$ is smaller in the sense that it possess more power of $\delta$,
\begin{equation}\label{compare-LLb-angular2}
|L\phi \Lb\phi| \lesssim \delta M^2 |\ub|^{2d_n-\frac{1}{2}}, \quad |\nablaslash \phi|^2 \lesssim \delta^{\frac{3}{2}} M^2|\ub|^{2d_n-1}.
\end{equation}
The worst term in $J_{12} \cdot \Lb\phi_k$ can be estimated as
\begin{equation}\label{r-singular}
\begin{split}
&\quad \iint_{\mathcal{D}_{u, \ub}}|r^{-1}\Lb\phi \nablaslash\phi \nablaslash\phi_k \underline{L}\phi_{k} | d^{n+1}x\\
& \lesssim \| \Lb\phi \nablaslash\phi\|_{ L^{\infty} (\mathcal{D}_{u, \ub} )} \left( \int_{0}^{u} \|\nablaslash\phi_{k}\|^2_{L^{2}(C_{u^\prime})} d u^\prime \right)^{\frac{1}{2}} \left( \int_{\ub_0}^{\ub} \frac{1}{\underline{u}^{\prime 2} }  \|\underline{L}\phi_{k}\|^2_{L^{2}(\Cb_{\ub^\prime})}  d \ub^\prime \right)^{\frac{1}{2}}  \\
& \lesssim \delta^{\frac{3}{4}} M^2 u^{\frac{1}{2}} \delta^{\frac{1}{2}} M \delta^{\frac{1}{2}} M \lesssim  \delta^{2+\frac{1}{4}} M^4.
\end{split}
\end{equation}
After estimating term by term, we conclude
\begin{equation}\label{3.42}
\iint_{\mathcal{D}_{u, \ub}}\sqrt{g}|J_{1}|d^{n+1}x\lesssim \delta^{2} M^4.
\end{equation}

For the lower order term $J_{2}$, there is also double null structure. Without loss of generality, we assume $i_{1}\leq i_{2}\leq i_{3}$, then
\begin{equation}\label{3.43}
\iint_{\mathcal{D}_{u, \ub}}\sqrt{g}|J_{2}|d^{n+1}x\lesssim\iint_{\mathcal{D}_{u, \ub}}|Q(\phi_{i_{1}},Q(\phi_{i_{2}},\phi_{i_{3}}))\underline{L}\phi_{k}|d^{n+1}x,
\end{equation}
where $i_{1}+i_{2}+i_{3} \leq k, \, i_1, i_2, i_3 <k$. Thus, there must be $i_{1}\leq i_{2}<\frac{k}{2}$. We know that $\frac{k}{2}\leq N-2$, since $N\geq 4$ and $k\leq N$. Hence, $i_1 \leq i_2 \leq N-3$, and we can apply the $L^{\infty}$ norm to $\p\phi_{i_{j}}$ or $\p^2 \phi_{i_j}$, $j=1,2$, according to Lemma \ref{lemma-apriori-estimate}.
As before, $Q(\phi_{i_{1}},Q(\phi_{i_{2}},\phi_{i_{3}}))$ can be also split into the $J_{11}$ type \eqref{J-11} and the $J_{12}$ type  \eqref{J-12} as well. The worst term in $J_{2}$ could be estimated by
\begin{equation}\label{3.44}
\begin{split}
& \quad \iint_{\mathcal{D}_{u, \ub}}|L\phi_{i_{1}}\underline{L}^2 \phi_{i_{2}} L\phi_{i_{3}}\underline{L}\phi_{k}|d^{n+1}x\\
&  \lesssim  \left( \int_{0}^{u} \|L\phi_{i_3}\|^2_{L^{2}(C_{u^\prime})} d u^\prime \right)^{\frac{1}{2}} \left( \int_{\ub_0}^{\ub} \|L \phi_{i_{1}} \underline{L}^2 \phi_{i_{2}}\|^2_{ L^{\infty} (\Cb_{\ub^\prime})} \|\underline{L}\phi_{k}\|^2_{L^{2}(\Cb_{\ub^\prime})}  d \ub^\prime \right)^{\frac{1}{2}}  \\
& \lesssim  u^{\frac{1}{2}} \delta M \left( \int_{\ub_0}^{\ub}  \left(  \delta M \delta^{-1} M |\ub^\prime|^{2d_{n}- \frac{1}{2}} \right)^2  \delta M^2  d \ub^\prime \right)^{\frac{1}{2}} \lesssim  \delta^{2} M^4.
\end{split}
\end{equation}
The estimate for the other terms in $J_2$ can be carried out in an analogous way.
Finally, we conclude
\begin{equation}\label{3.45}
\iint_{\mathcal{D}_{u, \ub}}\sqrt{g}|J_{2}|d^{n+1}x\lesssim \delta^{2} M^4.
\end{equation}

\begin{remark}\label{rk-J-11-J-12}
In estimating the source terms $J_1, \, J_{2}$, the double null forms containing the $J_{11}$ and $J_{12}$ types are cubic, which infers more decay rates in $\ub$. This is crucial in proving the uniform boundness of energy in the $n=2$ case.
\end{remark}

In summary, we have accomplished
\begin{equation}\label{3.46}
\delta \underline{E}^2_{\leq N}(u,\underline{u})\lesssim \delta I_N^{2}(\psi_{0},\psi_{1})+\delta^{2}M^4,
\end{equation}
which is equivalent to
\begin{equation}\label{3.47}
\underline{E}^2_{\leq N}\lesssim I_N^{2}(\psi_{0},\psi_{1})+\delta M^4.
\end{equation}

\subsubsection{Estimates for $E_{\leq N}$}\label{subsec-Eb-N}
In this section, we apply the energy identity \eqref{2.19} for the higher order RME \eqref{3.18}, with $\phi$ replaced by $\phi_{k}=\delta^{l}\partial^{l}\Gamma^{k-l}\phi$, and take $\xi=\tilde{L}$ to obtain the identity
\begin{equation}\label{3.48}
\begin{split}
& \int_{C_{u}}-\sqrt{g}P^{u}[\phi_{k},\tilde{L}]r^{n-1}d\underline{u}d\sigma_{S^{n-1}}+
\int_{\underline{C}_{\underline{u}}}-\sqrt{g}P^{\underline{u}}[\phi_{k},\tilde{L}]r^{n-1}dud\sigma_{S^{n-1}}\\
&=  \int_{\Sigma_1}-\sqrt{g}P^{t}[\phi_{k},\tilde{L}]r^{n-1}dr d\sigma_{S^{n-1}} -\iint_{\mathcal{D}_{u, \ub}}\partial_{\alpha}(\sqrt{g}P^{\alpha})d^{n+1}x,
\end{split}
\end{equation}
where the energy in the first line is equivalent to $\delta^{2}E^{2}_{\leq N}(u,\underline{u})$.

By an analogous argument as that for $\xi=\tilde{\Lb}$, the divergence of the current associated to $\tilde{L}$ could be calculated as below
\begin{equation}\label{eq-div-tdL}
\begin{split}
\partial_{\alpha}(\sqrt{g}P^{\alpha})&=\sqrt{g}\left(\Box_{g(\p\phi)}\phi_{k}\tilde{L}\phi_{k}+T^{\alpha}_{\beta} \cdot \partial_{\alpha}\tilde{L}^{\beta}\right)
-\frac{1}{2}
\tilde{L}(\sqrt{g}g^{\gamma\rho})\partial_{\gamma}\phi_{k}\partial_{\rho}\phi_{k}\\
&:=S+\sum_{1 \leq i\leq 3}M_{i}+\sum_{1 \leq j\leq 4}N_{j},
\end{split}
\end{equation}
where $S$ is the source term
\begin{equation}\label{Def-tdL-source}
S \doteq \sqrt{g}\Box_{g(\p\phi)}\phi_{k} \cdot \tilde{L}\phi_{k},
\end{equation}
and $\sum_{i=1}^3 M_{i}$ coming from the deformation tensor of $\tilde L$ is defined by
\begin{equation}\label{Def-tdL-M}
\begin{split}
\sum_{i=1}^3 M_{i}
&\doteq \sqrt{g}\Big[\left(-\frac{1}{r}+\frac{|L\phi|^2}{r}\right)\left(\frac{n-1}{2}
g^{\gamma\delta}\partial_{\gamma}\phi_{k}\partial_{\delta}\phi_{k}-g^{\omega\gamma}\partial_{\gamma}\phi_{k}\partial_{\omega}\phi_{k}\right)\\
& \quad + g^{\alpha \gamma}\partial_{\gamma}\phi_{k}\underline{L}\phi_{k} \p_\alpha (L\phi)^2  -L\underline{L}\phi L\phi g^{\gamma\delta}\partial_{\gamma}\phi_{k}\partial_{\delta}\phi_{k}\Big],
\end{split}
\end{equation}
while the error term $\sum_{j=1}^{4}N_{j}$ associated to the connection coefficients is given by
\begin{equation}\label{Def-tdL-N}
\begin{split}
\sum_{j=1}^4N_{j} &\doteq
-\frac{1}{2}\Big[\frac{Q(L\phi,\phi)+|L\phi|^2Q(\underline{L}\phi,\phi)}{\sqrt{g}}g^{\gamma\rho}\partial_{\gamma}\phi_{k}\partial_{\rho}\phi_{k}\\
&\,\,\,\,\,\,\,\,\,\,\,\,\,\,\,\,\,\,\,\,\,\,\,\,\,\,\,\,\,\,\,\,\,\,\,\,\,\,\,\,\,\,\,\,\,\,\,\,\,\,\, -\frac{2}{\sqrt{g}}(\partial^{\gamma}L\phi+|L\phi|^2\partial^{\gamma}\underline{L}\phi)\partial^{\rho}\phi\partial_{\gamma}\phi_{k}\partial_{\rho}\phi_{k}\\
& \quad - \frac{1-|L\phi|^2}{r\sqrt{g}} \gamma_{\omega\theta}\partial^{\omega}\phi\partial^{\theta}\phi g^{\gamma\rho}\partial_{\gamma}\phi_{k}\partial_{\rho}\phi_{k} +\frac{2-2|L\phi|^2}{r\sqrt{g}}\gamma_{\omega\theta}\partial^{\omega}\phi\partial^{\theta}\phi_k\partial^{\rho}\phi\partial_{\rho}\phi_{k}\Big].
\end{split}
\end{equation}

As before, we start with estimates for $M_{i}, \, i=1,\cdots,3$.

First of all, $M_{1}$ takes the following form
\begin{equation}\label{M1-n-Lb}
\begin{split}
 M_1 &= \sqrt{g}\left(-\frac{1}{r}+\frac{|L\phi|^2}{r}\right) \Big[ \frac{n-1}{2} g^{ab}\partial_{a}\phi_{k}\partial_{b}\phi_{k} \\
  & \quad +(n-2) g^{\omega a}\partial_{a}\phi_{k}\partial_{\omega}\phi_{k} + \frac{n-3}{2} g^{\omega \theta}\partial_{\theta}\phi_{k}\partial_{\omega}\phi_{k} \Big],
\end{split}
\end{equation}
where $a, \, b$ denote the null coordinates $(u, \ub)$.

When $n=3$, $M_1$ has the leading quadratic term $\frac{1}{r}L\phi_{k}\underline{L}\phi_{k}$, for which we shall make use of the previous result for $\Eb^2_{\leq k}(u, \ub)$ \eqref{3.46} to deduce
\begin{equation}\label{3.50}
\begin{split}
&\iint_{\mathcal{D}_{u, \ub}} |\frac{1}{r}L\phi_{k}\underline{L}\phi_{k}|d^{n+1}x\\
\lesssim &\iint_{\mathcal{D}_{u, \ub}}\frac{1}{\delta}|L\phi_{k}|^2d^{n+1}x+\iint_{\mathcal{D}_{u, \ub}}\frac{\delta}{|\underline{u}|^2}|\underline{L}\phi_{k}|^2d^{n+1}x\\
\lesssim& \int_{u_{0}}^{u}\frac{1}{\delta}\|L\phi_{k}\|^2_{L^{2}(C_{u^\prime})}d u^\prime + \int_{\underline{u}_{0}}^{\underline{u}}\frac{\delta}{|\underline{u}^\prime |^2}
\|\underline{L}\phi_{k}\|^2_{L^{2}(\underline{C}_{\underline{u^\prime }})}d \underline{u}^\prime\\
\lesssim &\int_{u_{0}}^{u}\delta E^2_{\leq k}(u^\prime) d u^\prime +  \delta (\delta I_N^2 (\psi_0, \psi_1)  +  \delta^{2} M^4).
\end{split}
\end{equation}
Here
\begin{equation}\label{def-sup-E}
E^2_{\leq k}(u^\prime)= \sup_{\ub} E^2_{\leq k}(u^\prime, \ub).
\end{equation}
The other terms in $M_1$ taking the null forms can be estimated straightforwardly.
\begin{remark}\label{re-hierarchy}
In \eqref{3.50}, if we estimate the leading term $\frac{1}{r}L\phi_{k}\underline{L}\phi_{k}$ in $M_{1}$ directly as in \eqref{3.32}, we would have $$\iint_{\mathcal{D}_{u, \ub}} |\frac{1}{r}L\phi_{k}\underline{L}\phi_{k}|d^{n+1}x \lesssim \delta^{2} M^2,$$ which is not good enough to close the bootstrap argument. Indeed, this is the borderline case for which the hierarchy of energy estimates (see \eqref{3.50}) should be employed.
\end{remark}

When $n=2$,
besides the leading quadratic term $\frac{1}{r}L\phi_{k}\underline{L}\phi_{k}$ which can be estimated in an analogous way as \eqref{3.50}, a quadratic term $\frac{\sqrt{g}}{2r} g^{\omega \theta}\p_{\omega}\phi_k \p_{\theta} \phi_k$ with a favourable sign in the energy identity \eqref{3.48} also occurs in $M_1$. Hence, if we move this term to the left hand side, it will contribute a positive term  $\iint_{\mathcal{D}_{u, \ub}} \frac{\sqrt{g}}{2r} g^{\omega \theta}\p_{\omega}\phi_k \p_{\theta} \phi_k d^{2+1}x$ to the spacetime integral. In other words, $L$ also serves as a Morawetz multiplier and provides an additional integrated decay estimate. This is one of the key points in proving the uniform boundness of the energy for the $n=2$ case.

$M_2$ and $M_3$ have signatures $+1$, and there are the double null structures.
Hence the estimates are all straightforward. We only pick out the leading terms in $M_2$,
\begin{equation}\label{3.51}
\begin{split}
&\quad \iint_{\mathcal{D}_{u, \ub}} | L\phi_{k}\underline{L}\phi_{k} \Lb L \phi L \phi|d^{n+1}x \\
& \lesssim \left( \int_{0}^{u} \|L\phi_{k}\|^2_{L^{2}(C_{u^\prime})} d u^\prime \right)^{\frac{1}{2}} \left( \int_{\ub_0}^{\ub}  \|\underline{L} L \phi L\phi\|^2_{ L^{\infty} (\Cb_{\ub^\prime})}  \|\underline{L}\phi_{k}\|^2_{L^{2}(\Cb_{\ub^\prime})}  d \ub^\prime \right)^{\frac{1}{2}}  \\
& \lesssim  u^{\frac{1}{2}} \delta M \left( \int_{\ub_0}^{\ub}  \left( M \delta M |\ub^\prime|^{2d_{n}-1} \right)^2  \delta M^2  d \ub^\prime \right)^{\frac{1}{2}} \lesssim  \delta^{3} M^4.
\end{split}
\end{equation}
In general, we arrive at
\begin{equation}\label{3.53}
\iint_{\mathcal{D}_{u, \ub}}\left( |M_{2}| +|M_3| \right)d^{n+1}x\lesssim \delta^{3}M^4.
\end{equation}

For the connection coefficients $\sum_{j=1}^4N_{j}$, the first two terms $N_{1}$ and $N_{2}$ with signatures $+1$, satisfy the null conditions. Thus, similar to $M_{2}, \, M_3$, we have
\begin{equation}
\iint_{\mathcal{D}_{u, \ub}} \left( |N_{1}|+|N_{2}| \right) d^{n+1}x\lesssim \delta^{3}M^4.
\end{equation}
As for $N_{3}$ and $N_{4}$, the coefficient admits $\frac{1-|L\phi|^2}{r\sqrt{g}} \sim |\ub|^{-1}$. Although $N_{3}$ and $N_{4}$ have signatures $0$, which are seemingly not as good as $N_{1}$ and $N_{2}$, they involve two angular derivative terms, which afford extra smallness (see \eqref{compare-LLb-angular2}). They can be treated in the same way as the $J_{12}$ type \eqref{J-12} of the double null forms. In a word, they can be  typically estimated as
\begin{equation}\label{3.55}
\iint_{\mathcal{D}_{u, \ub}} \left( |N_{3}|+|N_{4}| \right)d^{n+1}x 
\lesssim \delta^{\frac{7}{2}}M^4.
\end{equation}

At last, we turn to the source term $S$, which takes the null form.  As \eqref{3.39} in Section \ref{subsec-E-N}, we decompose $S=S_{1}+S_{2}$. Each of the $S_1$ and $S_2$ taking the null form can be again split into the $J_{11}$ \eqref{J-11} and $J_{12}$ \eqref{J-12} types.  For the $J_{11}$ type in $S_{1}$, the leading term is
\begin{equation}\label{3.56}
 \iint_{D} \left( |\underline{L}^2\phi L\phi L \phi_{k}| + |\underline{L}L\phi L\phi \Lb\phi_k| \right) | L\phi_{k}|d^{n+1}x\lesssim \delta^{3} M^4.
\end{equation}
The $J_{12}$ type in $S_{1}$ could be estimated in the same way as \eqref{3.55}, so that they are bounded by $\delta^{\frac{7}{2}} M^4$. The same argument also applies to $S_{2}$, thus we have
\begin{equation}\label{3.57}
\iint_{\mathcal{D}_{u, \ub}} \left( |S_1| + |S_{2}| \right) d^{n+1}x\lesssim \delta^{3} M^{4}.
\end{equation}

As in Remark \ref{rk-J-11-J-12}, we also emphasize that the cubic null form $Q(\phi_{i_{1}},  Q(\phi_{i_{2}},\phi_{i_{3}}))$ provides sufficient decay rates in $\ub$, which is needed in proving the uniform energy for the $n=2$ case.

In summary, we conclude
\begin{equation}\label{3.58}
\delta^{2}E_{\leq N}^2 (u, \ub)\lesssim \delta^{2}I_N^{2}(\psi_{0},\psi_{1})+\int_{0}^{u}\delta  E^2_{\leq N}( u^\prime )d u^\prime +\delta^{3} M^4.
\end{equation}
Particularly, in the $n=2$ case, we obtain
\begin{equation}\label{3.58-n-2}
\begin{split}
& \quad \delta^{2}E_{\leq N}^2 (u, \ub) + \iint_{\mathcal{D}_{u, \ub}} \frac{\sqrt{g}}{2r} g^{\omega \theta}\p_{\omega}\phi_k \p_{\theta} \phi_k d^{2+1}x  \\
& \lesssim \delta^{2}I_N^{2}(\psi_{0},\psi_{1})+\int_{0}^{u}\delta E^2_{\leq N}(u^\prime)d u^\prime +\delta^{3} M^4.
\end{split}
\end{equation}

Taking the $\sup$ on $\ub$ (see \eqref{def-sup-E}), we further  derive
\begin{equation}\label{3.59}
E^2_{\leq N}(u) \leq I_N^2(\psi_{0},\psi_{1})+\int_{0}^{u} \delta^{-1}E^2_{\leq N}( u^\prime )d u^\prime +\delta M^4.
\end{equation}
An application of the Gr\"{o}nwall inequality yields
\begin{equation}\label{3.60}
E^2_{\leq N}(u)\lesssim I_N^2(\psi_{0},\psi_{1})+\delta M^4.
\end{equation}

\begin{remark}\label{rk-quailinear-bdry}
Let us explain  what will happen if we ignore the geometry of membrane and why $L, \Lb$ do not qualify for the multipliers.
We rewrite the RME \eqref{3.19} as
\begin{equation}\label{RME-conformal-formular}
\p_\mu\left(\frac{\p^\mu\phi}{\sqrt{1+Q}} \right)=0,
\end{equation}
and the high order version takes similar form
\begin{equation}\label{RME-conformal-formular-high-order}
\begin{split}
\p_\mu\left(\frac{\p^\mu Z^N\phi}{\sqrt{1+Q}} \right) &= \frac{Q(\phi, Q(Z^{N}\phi, \phi) )}{\sqrt{1+Q}^3} +   \frac{Q(\phi, Q(\phi, \phi)) Q(Z^{N}\phi, \phi) }{\sqrt{1+Q}^5} \\
&+\sum_{ \substack{i_1+\cdots +i_k \leq N, \\ i_1,\cdots, i_k <N, \\ 3\leq k\leq 2N+3}} F(Q)Q(Z^{i_1} \phi, Q( Z^{i_2}\phi, Z^{i_3} \phi) ) \cdots Q( Z^{i_k-1}\phi, Z^{i_k} \phi)\\
&\quad \quad \quad \quad \quad \quad  \times Q(Z^{i_4} \phi,Z^{i_5} \phi) \cdots Q( Z^{i_{k-1}}\phi, Z^{i_k} \phi).
\end{split}
\end{equation}
In \eqref{RME-conformal-formular-high-order}, $\frac{Q(\phi, Q(Z^{N}\phi, \phi) )}{\sqrt{1+Q}^3}$ is a quasilinear term, which could be treated through integration by parts.
The corresponding energy momentum tensor is (associated with the flat metric $\eta_{\mu \nu}$)
\begin{equation}\label{energy momentum tensor-conformal}
T_{\mu\nu}(\varphi) =  \frac{\p_\mu \varphi}{\sqrt{1+Q}} \cdot \frac{\p_\nu \varphi}{\sqrt{1+Q}} - \frac{1}{2}\eta_{\mu\nu}  \frac{\eta^{\sigma \rho}\p_\rho \varphi \p_\sigma \varphi}{\sqrt{1+Q}}.
\end{equation}
Letting $\phi_N=Z^N\phi$, we calculate that the divergence
\begin{equation}\label{RME-conformal-formular-high-order1}
\begin{split}
\p^\mu T_{\mu \nu}(\phi_N) & = \p_\mu \left(\frac{\p^\mu \phi Q(\phi_N, \phi) \p_\nu\phi_N}{(1+Q)^2}  \right) - \p_\nu \left(\frac{ Q^2(\phi_N, \phi)}{2(1+Q)^2}  \right)\\
& \quad + \text{nonlinear error terms},
\end{split}
\end{equation}
where the first line on the right hand side of the equality is due to the quasilinear term.
Taking $L$ and $\Lb$ as the multipliers respectively, the divergence theorem leads to  the energy identities
\begin{equation}\label{energy-id-L-conformal}
\begin{split}
& \int_{\Cb_{\ub}} \frac{|\nablaslash\phi_N|^2}{1+Q} - \frac{Q^2(\phi_N, \phi)}{2(1+Q)^2} - \frac{\Lb\phi Q(\phi_N, \phi)L\phi_N}{(1+Q)^2}  \\
&\quad + \int_{C_u} \frac{|L\phi_N|^2}{1+Q} - \frac{L\phi Q(\phi_N, \phi)L\phi_N}{(1+Q)^2} = I_N(\psi_0, \psi_1) + \mathcal{S}_L,
\end{split}
\end{equation}
\begin{equation}\label{energy-id-Lb-conformal}
\begin{split}
&\int_{C_u} \frac{|\nablaslash\phi_N|^2}{1+Q} - \frac{Q^2(\phi_N, \phi)}{2(1+Q)^2} -  \frac{L\phi Q(\phi_N, \phi) \Lb\phi_N}{(1+Q)^2} \\
& \quad +   \int_{\Cb_{\ub}}  \frac{|\Lb\phi_N|^2}{1+Q}  - \frac{\Lb\phi Q(\phi_N, \phi) \Lb\phi_N}{(1+Q)^2}   = I_N(\psi_0, \psi_1) + \mathcal{S}_{\Lb},
\end{split}
\end{equation}
where $I_N(\psi_0, \psi_1)$ is the corresponding datum and $\mathcal{S}_L, \, \mathcal{S}_{\Lb}$ denote integrals on the domain $\mathcal{D}_{u, \ub}$.
Now, we see that, the left hand side of the equality \eqref{energy-id-L-conformal} which is supposed to be the energy takes the form of
\begin{equation}\label{boundary-L-conformal}
\begin{split}
 & \int_{C_u} \frac{|L\phi_N|^2}{1+Q}  + \int_{\Cb_{\ub}}  \frac{|\nablaslash\phi_N|^2}{1+Q}  -  \frac{|\nablaslash\phi \nablaslash \phi_N|^2}{2(1+Q)^2}   + \frac{3|\Lb\phi|^2 |L\phi_N|^2 }{8(1+Q)^2}  -  \frac{|L\phi|^2 |\Lb\phi_N|^2}{8(1+Q)^2} + \cdots,
\end{split}
\end{equation}
while the one in \eqref{energy-id-Lb-conformal} could be explicitly written as
\begin{equation}\label{boundary-Lb-conformal}
\begin{split}
&\int_{\Cb_{\ub}}  \frac{|\Lb\phi_N|^2}{1+Q} + \int_{C_u}  \frac{|\nablaslash\phi_N|^2}{1+Q} -  \frac{|\nablaslash\phi \nablaslash \phi_N|^2}{2(1+Q)^2}  + \frac{3|L\phi|^2 |\Lb\phi_N|^2 }{8(1+Q)^2}  - \frac{|\Lb\phi|^2 |L\phi_N|^2}{8(1+Q)^2} + \cdots.
\end{split}
\end{equation}
Here we use $\cdots$ to denote terms with indeterminate signs.
Due to the presence of the non-positive terms such as $-\int_{\underline{C}_{\underline{u}}}\frac{|L\phi|^2 |\Lb\phi_{N}|^2}{8(1+Q)^2}$ and $-\int_{C_{u}} \frac{|\Lb\phi|^2 |L\phi_N|^2 }{8(1+Q)^2}$ in \eqref{boundary-L-conformal}, \eqref{boundary-Lb-conformal},  we know that both of \eqref{boundary-L-conformal} and \eqref{boundary-Lb-conformal} can be negative themselves, although it is possible to prove that \eqref{boundary-L-conformal} + \eqref{boundary-Lb-conformal} will provide coercive energies. This indicates that for the large data problem we are considering, hierarchies of energy estimates (namely, one should first use \eqref{energy-id-L-conformal} to prove the energy bound for \eqref{boundary-L-conformal}, and then utilize this result to prove the energy bound for \eqref{boundary-Lb-conformal} via \eqref{energy-id-Lb-conformal}) are not allowed. On the other hand, as illustrated in Remark \ref{re-hierarchy}, it is impossible to do the energy estimates for the RME (with large data) without proposing the hierarchies of energy estimates. Consequently, we can not take $L, \, \Lb$ as multipliers for the RME.
\end{remark}
\subsubsection{Global solution in Region {\rm II}} Now we can show the global existence of smooth solution of relativistic membrane equation in Region {\rm II}. By definition of $\lesssim$, there exists a positive constant $C$, such that, by \eqref{3.47} and \eqref{3.60}
\begin{equation}
\underline{E}_{N}+E_{N}\leq C(I_{N}(\psi_{0},\psi_{1})+\delta^{\frac{1}{2}} M^2),
\end{equation}
under the bootstrap assumption that $\underline{E}_{N}^2+E_{N}^{2}\leq M$ with a large constant $M$. By the method of continuity, it suffices to show that there holds
\begin{equation}
\underline{E}_{N}+E_{N}\leq \frac{1}{2}M,
\end{equation}
for all $(u,\underline{u})$ in Region {\rm II}. The above inequality can be easily obtained if we choose $M=4CI_{N}(\psi_{0},\psi_{1})$ and $\delta_{0}=\left(\frac{1}{16C^2I_{N}(\psi_{0},\psi_{1})}\right)^2$. We remark here that $I_{N}(\psi_{0},\psi_{1})$ depends on the constraint equations \eqref{1.2}-\eqref{1.3-associated} on the initial data. Thus, the global existence of the smooth solution is obtained in Region {\rm II}.

\section{Smallness on the last slice $C_{\delta}$}\label{Section4}
 In this section, we shall prove that the solution restricted on the boundary $C_{\delta}$ is small. This can be done  through integrating along $L$ on $C_{\delta}$, viewing that the data on the initial sphere $C_{\delta}\cap\Sigma_{1}$ vanish (see  \cite{Miao-Yu,Wang-Yu,Wang-Yu1}).
For this purpose, we will make use of an ODE form of the membrane equation (refer to \eqref{box-estimate-general}) which is analogous to the one in semilinear case. Additionally, we will use $\delta \p$ as commutators. Let
\begin{equation}\label{bar-Z-pb}
\bar Z=\{\Gamma\}\cup\{ \delta\p\} \quad \text{and} \, \quad \bar\p = \{L, \nablaslash\}
\end{equation}
be the good vector fields. We recall the definition of the indices $d_n= - \frac{n-1}{2}, \, n=2,3$.
\begin{proposition}\label{pro4.4}
Suppose we have bounded $E_i(\delta, \ub)$ and $\Eb_i(\delta, \ub)$ for $i \leq N, \, N\geq 6$. Then, for all $k \geq 1$ and $l \geq 0$, there is
\begin{equation}\label{decay-C}
 |\p^k\Gamma^l\phi|_{L^\infty (C_{\delta})} \lesssim \delta^{\frac{3}{4}}|\ub|^{-\frac{n-1}{2}}, \quad k+l \leq N-3.
\end{equation}
\end{proposition}

\begin{proof}
We shall prove this proposition by induction. When $k+l=1, \, k \geq 1$, i.e. $k=1, \, l=0$, the smallness for the good derivatives follows from the \emph{a priori} estimates in Section \ref{Section-large-data},
$$|L\phi|\lesssim \delta|\ub|^{d_n-\frac{1}{2}}, \quad |\nablaslash\phi|\lesssim \delta^{\frac{3}{4}}|\ub|^{d_n-\frac{1}{2}}.$$
Therefore, it suffices to show the smallness of $\Lb\phi$ on $C_{\delta}$ in this case. We take $\varphi = \phi$ in the ODE form \eqref{box-estimate-general}, then
\begin{equation}\label{M-ode-0-order-null-eta}
\begin{split}
0 & =  \p_{\ub} \p_u \phi  + \frac{n-1}{2r} \p_{u}\phi - \frac{n-1}{2r} \p_{\ub}\phi - \nablaslash^\omega \nablaslash_\omega \phi \\
& \quad \pm |\ub|^{2d_n} \p^2_{\ub} \phi \pm \delta^2 |\ub|^{2d_n-1} \p_u^2 \phi  \pm \delta^{\frac{7}{4}} |\ub|^{2d_n - 1}  \nablaslash  \p_u \phi  \\
& \quad \pm \delta^{\frac{3}{4}} |\ub|^{2d_n - \frac{1}{2}}  \nablaslash \p_{\ub} \phi  \pm \delta^{\frac{3}{2}} |\ub|^{2d_n -1} \nablaslash^2 \phi  \\
& \quad  \pm |\ub|^{2d_n - 1} \p_{\ub}  \phi \pm \delta  |\ub|^{2d_n - \frac{3}{2}} \p_u \phi \pm \delta^{\frac{3}{4}} |\ub|^{2d_n - \frac{3}{2}}  \nablaslash \phi.
\end{split}
\end{equation}
Substituted the estimates proved in Section \ref{Section-large-data}, \eqref{M-ode-0-order-null-eta} turns into
\begin{equation}\label{ODE-estimate-0}
 \p_{\ub} \p_u \phi + \frac{n-1}{2r} \p_{u}\phi
 \lesssim  \delta^{\frac{3}{4}} |\ub|^{d_n-\frac{3}{2}}.
\end{equation}
Multiplying $2 \underline{u}^{n-1} \p_u \phi$ on both sides of \eqref{ODE-estimate-0}, we obtain
\begin{equation}\label{M-ineq-0-order-null-eta}
\begin{split}
&\quad L\left(\ub^{n-1} (\Lb\phi)^2 \right) =  (n-1) \ub^{n-2} (\Lb\phi)^2 +  2 \ub^{n-1} \Lb \phi L \Lb\phi  \\
&  \lesssim 2  \Big| - \frac{n-1}{2r} + \frac{n-1}{2\ub}\Big| \ub^{n-1} (\Lb\phi)^2 + \ub^{n-1} |\Lb \phi| \cdot \delta^{\frac{3}{4}}   |\ub|^{d_n-\frac{3}{2}}.
\end{split}
\end{equation}
Noticing that, the coefficient has smallness, $$\Big|- \frac{1}{r} + \frac{1}{\ub} \Big| \lesssim \delta|\ub|^{-2}.$$
We then integrate along the curves of $L$ on $C_{\delta}$, and use the fact that $\Lb\phi$ vanishes on $S_{\delta, 1-\delta}$. An application of the Gr\"{o}nwall inequality yields that
\begin{equation}\label{Lb-phi-1-c-delta}
 |\Lb\phi|_{C_{\delta}} \lesssim \delta^{\frac{3}{4}}|\ub|^{-\frac{n-1}{2}}.
\end{equation}

Next, we proceed to the case of $k+l=2, \, k \geq 1$. In view of the definition \eqref{bar-Z-pb}, the smallness of $\bar\p\Gamma\phi$ follows from the energy estimates in Section \ref{Section-large-data}. Hence, we are left with: $\Lb \Gamma \phi$, $\Lb \bar \p\phi$ and $\Lb^2\phi$.

We will first prove the smallness for $\Lb \bar Z\phi$, which covers $\Lb \Gamma \phi$. After commuting $\bar Z$ with the RME (see Lemma \ref{lemma-high-order-eq}), we have, for $\varphi=\bar Z\phi$
\begin{equation}\label{wave-from-membrane-general}
\Box_{g(\p\phi)}\varphi=F(\p\varphi, \p\phi, \p^2 \phi),
\end{equation}
where $F(\p\varphi, \p\phi, \p^2 \phi)$ denotes the lower order derivative terms satisfying the double null condition.
A calculation shows that $F(\partial\varphi,\partial\phi, \p^2 \phi)$ would be of lower order in terms of the power of $\delta$ and the decay of $\ub$:
\begin{equation}\label{source}
|F(\partial\varphi,\partial\phi, \p^2 \phi)|\lesssim \delta|\underline{u}|^{2d_{n}-\frac{3}{2}}|\underline{L}\varphi|+|\underline{u}|^{2d_{n}-1}|L\varphi| + \delta^{\frac{3}{4}}|\underline{u}|^{2d_{n}-\frac{3}{2}} |\nablaslash\varphi| + \delta^{\frac{3}{2}}|\underline{u}|^{3d_{n}-1}.
\end{equation}
Applying \eqref{box-estimate-general} to $\varphi=\bar{Z}\phi$,
and combining the $L^\infty$ estimates of $\p^i \bar Z \phi, \, 1 \leq i \leq 2$ with \eqref{source}, we also derive the inequality \eqref{ODE-estimate-0} with $\phi$ therein being replaced by $\varphi = \bar Z \phi$.
 A similar argument as before leads to
\begin{equation*}
|\Lb\bar Z\phi|_{C_\delta} \lesssim \delta^{\frac{3}{4}} |\ub|^{-\frac{n-1}{2}}.
\end{equation*}
Generally, we can commute $\bar Z$ repeatedly with the membrane equation to conclude
\begin{equation}\label{Lb-estimate-last-slice-n-derivative}
|\Lb\bar Z^k \phi|_{C_\delta} \lesssim \delta^{\frac{3}{4}} |\ub|^{-\frac{n-1}{2}}, \,\,\, k\in \mathbb{N}.
\end{equation}

For $\Lb \bar \p\phi$, noting that $|\Lb \bar \p\phi | \lesssim |\ub|^{-1} | \Lb \Gamma \phi |$, hence the smallness of $\Lb \bar \p\phi$ on  $C_{\delta}$ follows.

For the last case $\Lb^2\phi$, or equivalently $\Lb\p\phi$, we commute $\p$ with the membrane equation, and take $\varphi = \p\phi$. We use the previous results \eqref{Lb-estimate-last-slice-n-derivative}, then the estimate \eqref{ODE-estimate-0} are also satisfied with $\phi$ being replaced by $\varphi = \p\phi$.  Thus, we can repeat the argument before to achieve
\begin{equation}\label{Lb2-estimate-last-slice-2nd-derivative}
 |\Lb\p\phi|_{C_\delta} \lesssim \delta^{\frac{3}{4}} |\ub|^{-\frac{n-1}{2}}.
\end{equation}

Now, we have completed the proof for the case of $k+l=2, \, k \geq 1$. Without loss of generality, the smallness for  higher order derivatives on the last null cone $C_{\delta}$ follows by induction.

\end{proof}

\section{Globally smooth solution in the small data region I}\label{Section-small-data}
In this section, we mainly investigate the solution to \eqref{1.1} in the small data region I, which has boundaries $C_{\delta}$ and $\Sigma_{1}$.

Denote $\bar{D}$ the small data region {\rm I}:
$$
\bar{D}\doteq \Big\{(t,x)\big|  u \geq\delta,\, t\geq1\Big\}.
$$
We still use $\Sigma_{\tau}$ to denote the constant time slice in $\bar{D}$, i.e.,
$\Sigma_{\tau}=\{(t,x)| t=\tau\} \cap \bar{D}.$
Given a point $(t,x)\in \Sigma_{t}$, we use $(t, B(t,x))$ to denote the points in $C_{\delta}$. We now present the Klainerman-Sobolev inequality without proof (refer to \cite{Christodoulou}, see also \cite{Miao-Yu} and \cite{Wang-Yu}).
\begin{lemma}[\emph{Klainerman-Sobolev inequality}]For any $\varphi \in C^{\infty}(\mathbb R^{1+n})$ and $t>1$ with $(t,x)\in\bar{D}$, we have
\begin{equation*}
|\varphi (t,x)|\lesssim \frac{1}{(1+|u|)^{\frac{1}{2}}}|\varphi (t,B(t,x))|+\frac{1}{(1+|\underline{u}|)^{\frac{n-1}{2}}(1+|u|)^{\frac{1}{2}}}\sum_{k\leq 3}\|Z^{k} \varphi \|_{L^{2}(\Sigma_{t})}.
\end{equation*}
\end{lemma}
Before stating our main result, we define the energy on the constant $t$ slice $\Sigma_t$,
\begin{equation*}
\bar{E}_{k}(t)=\left(\int_{\Sigma_{t}}(\partial_{t}Z^{k}\phi)^2+|\nabla Z^{k}\phi|^2d^n x\right)^{\frac{1}{2}}, \quad \bar{E}_{\leq k}(t)=\sum_{0 \leq j \leq k} \bar{E}_{j}(t).
\end{equation*}
The main goal of this section is to prove the following proposition, which can show the global existence of small solution to RME in Region {\rm I}.
\begin{proposition}
Assume the solution to the RME satisfies \eqref{decay-C} on $C_{\delta}$ (see Proposition \ref{pro4.4}), then there exists a uniquely global solution to the RME \eqref{1.1} in the small data region. Moreover, the solution enjoys the following energy estimate
\begin{equation}\label{eq-energy-inequ-small-region-III}
\bar{E}_{\leq N}(t) \leq \delta^{\frac{3}{4}}C(I_{N}(\psi_0, \psi_1)),
\end{equation}
for $t\geq 1$, $12 \leq N$ and $C(I_{N}(\psi_0, \psi_1))$ depends on the data up to $N+1$ order of derivatives.
 \end{proposition}

 \begin{proof}
 The proof is based on a bootstrap argument. We assume that the solution exists up to time $t$ and for all $1\leq t^{\prime}\leq t$, there is a large constant $M$ to be determined such that
 \begin{equation*}
 \bar{E}_{\leq N}(t)\leq M\delta^{\frac{3}{4}}.
 \end{equation*}
 It suffices to prove that at the end, we could choose $M$ such that it only depends on the initial data.
 By the Klainerman-Sobolev inequality, we see that for $p\leq N-3$,
 \begin{equation}\label{L-infty-small-data}
 \begin{split}
 |\partial Z^{p}\phi|&\lesssim \frac{1}{(1+|u|)^{\frac{1}{2}}}|\partial Z^{p}\phi(t,B(t,x))|+\frac{1}{(1+|\underline{u}|)^{\frac{n-1}{2}}(1+|u|)^{\frac{1}{2}}}
 \sum_{j\leq 3}\|Z^j\partial Z^{p}\phi\|_{L^{2}(\Sigma_{t})}\\
 &\lesssim \frac{C(I_{N}(\psi_0, \psi_1))}{(1+|\underline{u}|)^{\frac{n-1}{2}}(1+|u|)^{\frac{1}{2}}}\delta^{\frac{3}{4}}
 +\frac{M}{(1+|\underline{u}|)^{\frac{n-1}{2}}(1+|u|)^{\frac{1}{2}}}\delta^{\frac{3}{4}}.
 \end{split}
 \end{equation}
 This shows that $\partial\phi$ is small with proper decay estimate.

 In the following, we perform the energy argument and take $\xi=2\partial_{t}$ as the multiplier (noting that $g (\p_t, \p_t) = -1 + (\p_t \phi)^2 \leq -1 + \delta^{\frac{3}{2}} M^2<0$, hence $\p_t$ is timelike), then
\begin{align}
- P^{t}[\phi_{k},2\partial_{t}] 
&=(\partial_{t}\phi_{k})^2+|\nabla\phi_{k}|^2 - \frac{Q(\phi_k, \phi)}{g} \left(2\p_t \phi \p_t \phi_{k} + Q(\phi_k, \phi) \right), \label{small-region-energy-t} \\
-P^{u}[\phi_{k},2\partial_{t}] 
&=\frac{1}{2}(L\phi_{k})^2 + \frac{1}{2}|\nablaslash\phi_{k}|^2 - \frac{Q(\phi_k, \phi)}{2g} \left(2L\phi \p_t \phi_{k} + Q(\phi_k, \phi) \right). \label{small-region-energy-u}
\end{align}
Recalling that $\partial\phi$ is small enough, it is easy to check that for \eqref{small-region-energy-t}, \eqref{small-region-energy-u}
\begin{align*}
-P^{t}[\phi_{k},2\partial_{t}] & \sim (\partial_{t}\phi_{k})^2+|\nabla\phi_{k}|^2, \\
-P^{u}[\phi_{k},2\partial_{t}] & \sim \frac{1}{2}(L\phi_{k})^2+\frac{1}{2}|\nablaslash\phi_{k}|^2 + \frac{1}{8} |L\phi|^2 (\Lb\phi_{k})^2.
\end{align*}
 Then by the divergence theorem, we have
 \begin{equation}\label{eq-diver-small-region}
 \begin{split}
 \bar{E}^2_{k}(t)&\lesssim C(I_{k}(\psi_0, \psi_1)) \delta^{\frac{3}{2}} +\iint_{\bar{D}}|\partial_{\alpha}(\sqrt{g}P^{\alpha} [\phi_k, 2 \p_t])|d^{n+1}x\\
 &\lesssim C(I_{k}(\psi_0, \psi_1)) \delta^{\frac{3}{2}}\\
& +\iint_{\bar{D}}\Big|2\sqrt{g}\Box_{g(\p\phi)}\phi_{k}\partial_{t}\phi_{k}-\frac{\partial_{t}Q}{g}g^{\gamma\rho}\partial_{\gamma}\phi_{k}\partial_{\rho}\phi_{k}
 + \frac{2\partial_{t}(\partial^{\gamma}\phi\partial^{\rho}\phi)}{\sqrt{g}}\partial_{\gamma}\phi_{k}\partial_{\rho}\phi_{k}\Big|d^{n+1}x.
 \end{split}
 \end{equation}

 In what follows, the energy estimates will be done separately in two subregions of $\bar{D}$, which are divided by $r=\frac{t}{2}$.
 \subsection{Interior region}
 We consider the equation in the domain $\left\{(t,x)| r\leq \frac{t}{2}\right\}$ at first. By the Klainerman-Sobolev inequality \eqref{L-infty-small-data}, each derivative of the solution admits good decay rate, for we have $|u|\geq \frac{t}{2}$ in this subregion. Namely, there is
 \begin{equation*}
 |\partial Z^{p}\phi|
 \lesssim \frac{1}{t^{n/2}}(I_{N}(\psi_{0},\psi_{1})+M)\delta^{\frac{3}{4}}, \quad n \geq 2.
 \end{equation*}
 The nonlinear terms in the double integrals of \eqref{eq-diver-small-region} are all at least quartic. They can be bounded by (we do not need the double structures here),
 \begin{equation}\label{integral}
 \int_{1}^{t}\frac{\delta^3}{\tau^{n}}M^2 \left( I_{N}(\psi_{0},\psi_{1})+M \right)^2d\tau, \quad n \geq 2.
 \end{equation}

 \subsection{Exterior region}
 When confining ourselves to the domain $\left\{(t,x)|\frac{t}{2}\leq r\leq t-2\delta\right\}$, the double null structure of the non-linearities in \eqref{eq-diver-small-region} comes into play. Again, they can be classified as the $J_{11}$ \eqref{J-11} and $J_{12}$ \eqref{J-12} types (see Section \ref{subsec-E-N} and Remark \ref{rk-J-11-J-12}). The $J_{11}$ type \eqref{J-11} is cubic, which suggests more decay rates on $t$. In applications, we use $L^2$ norm to control the higher order derivative and use $L^{\infty}$ norm to bound the terms with lower order derivatives. Here we only take the worst term such as $\underline{L}^2 \phi_{i_{1}} L\phi_{i_{2}}L \phi_{i_{3}}\partial\phi_{k},\, (i_{1} \leq i_{2}\leq i_{3}\leq k, \, i_1 + i_2 +i_3 \leq k)$ of the $J_{11}$ type for example,
 \begin{align*}
 &\quad \iint_{\bar{D}\cap\{r\geq\frac{t}{2}\}}|\underline{L}^2 \phi_{i_{1}} L \phi_{i_{2}}L\phi_{i_{3}}\partial\phi_{k}|d^nxdt\\
 &\lesssim \int_{1}^{t}\|\underline{L}^2 \phi_{i_{1}}\|_{L^{\infty}(\Sigma_{\tau})}\|L \phi_{i_{2}}\|_{L^{\infty}(\Sigma_{\tau})}
 \|L\phi_{i_{3}}\|_{L^{2}(\Sigma_{\tau})}\|\partial\phi_{k}\|_{L^{2}(\Sigma_{\tau})}d\tau\\
 &\lesssim \int_{1}^{t}\frac{\delta^{3}}{\tau^{n-\frac{1}{2}}}M^{4}\leq \delta^{3}M^{4}, \quad n \geq 2.
 \end{align*}
The other terms of $J_{11}$ type can be estimated in the same way.
 The $J_{12}$ type \eqref{J-12} has an extra factor $r^{-1}$ and possess two angular derivative terms. Generally, for $i_1 + i_2 +i_3 \leq k$, $i_1 \leq i_2 \leq i_3$,
 \begin{align*}
 &\quad \iint_{\bar{D}\cap \{r\geq\frac{t}{2}\}}|r^{-1}\partial_{r}\phi_{i_{1}}\nablaslash\phi_{i_{2}}\nablaslash\phi_{i_{3}}\partial\phi_{k}|d^nxdt\\
 &\lesssim \int_{1}^{t}\tau^{-1}\|\partial_{r}\phi_{i_{1}}\|_{L^{\infty}(\Sigma_{\tau})}\|\nablaslash\phi_{i_{2}}\|_{L^{\infty}(\Sigma_{\tau})}
 \|\nablaslash\phi_{i_{3}}\|_{L^{2}(\Sigma_{\tau})}\|\partial\phi_{k}\|_{L^{2}(\Sigma_{\tau})}d\tau\\
 &\lesssim\int_{1}^{t}\frac{\delta^{\frac{3}{2}}}{\tau^{1+\frac{n}{2}+\frac{n-1}{2}}}M^2\delta^{\frac{3}{2}}M^2d\tau \lesssim \int_{1}^{t}\frac{\delta^{3}M^{4}}{\tau^{n+\frac{1}{2}}}\lesssim \delta^{3}M^{4}, \quad n \geq 2.
 \end{align*}
 Combining the estimates for the $J_{11}$ and $J_{12}$ types and \eqref{integral}, we deduce that for $n \geq 2$,
 \begin{equation*}
 \bar{E}^{2}_{\leq N}(t)\lesssim C(I_{N}(\psi_0, \psi_1))\delta^{\frac{3}{2}}+\int_{1}^{t}\frac{\delta^{3}}{\tau^{n - \frac{1}{2}}}M^{4}d\tau
 \lesssim C(I_{N})\delta^{\frac{3}{2}}+\delta^{3}M^{4}.
 \end{equation*}
 Thus, if $\delta$ is small enough, the \emph{a priori} estimate \eqref{eq-energy-inequ-small-region-III} holds easily.
 \end{proof}
 \begin{remark}
 By the above energy estimates, we have shown that the membrane equation with $n\geq2$ admits a uniform energy bound due to the fact that the nonlinear terms satisfy the double null condition. This is also an improvement of Lindblad's result \cite{Lindblad}, where he proved that the higher order energy is slowly growing with time.
 \end{remark}
%

\section{Nonlinear effect of membrane}\label{sec-blow up-scri}
In this section, we concern the asymptotic behavior near the null infinity rather than the large data problem. Hence, we will focus on the small data problem and take $n=3$, i.e., consider the Cauchy problem of the $1+3$-dim membrane equation with the initial data,
\begin{equation*}
\phi |_{t=0} = \varepsilon \phi_0, \quad \p_t \phi |_{t=0} = \varepsilon \phi_1,
\end{equation*}
where  the size of data $\varepsilon$ is small enough.

As we know, for any function $f$, $ |\p_{\ub} f| +|\nablaslash f|  \lesssim  \frac{1}{|\ub|+1} |\Gamma f|$, then the decay rates of $\phi$ in sections \ref{Section-large-data} and \ref{Section-small-data} can be improved as (in the small data setting), for all $j \in \mathbb{N}$,
\begin{equation}\label{decay-d -phi}
|\p_u \Gamma^j \phi | \lesssim \frac{\varepsilon}{(|u|+1)^{\frac{1}{2}} (|\ub|+1)}, \quad |\p_{\ub} \Gamma^j \phi | +|\nablaslash \Gamma^j \phi |  \lesssim \frac{\varepsilon}{(|\ub|+1)^{\frac{3}{2}}}.
\end{equation}
In the following quantitative estimates, we will need  the decay results \eqref{decay-d -phi}, which can be written more explicitly as
\begin{equation}\label{decay-mem}
\begin{split}
|r^{1+i+j} \nablaslash^j \p_{\ub}^i \p_u^{k+1} \phi|  &\lesssim \varepsilon (|u|+1)^{-\frac{1}{2}-k},\\
|r^{\frac{3}{2} +i+j} \nablaslash^j \p_{\ub}^{i+1} \p_u^{k} \phi|  & \lesssim  \varepsilon (|u|+1)^{-k},\\
|r^{\frac{3}{2} +i+j} \nablaslash^{j+1} \p_{\ub}^{i} \p_u^{k} \phi|  & \lesssim  \varepsilon (|u|+1)^{-k}.
\end{split}
\end{equation}
In fact, only $i+j+k \leq 1$ is needed for our purpose.

\subsection{Geometry of membrane}

\subsubsection{Null frame adapted to the sphere}\label{sec-null-frame}
In this section, we will define for the membrane (with metric $g_{\mu \nu}$) a null frame which is adapted to $S_{u, \ub}$ and quantitatively show how it is close to the standard null frame $\{L, \Lb\}$ in Minkowski spacetime. Recall that $S_{u, \ub}$ is the $2$-sphere with constant $u$ and $\ub$.

Let $\{e_A, \, A=1,2 \}$ be an orthonormal (local) frame (with respect to $g_{\mu \nu}$) on  $S_{u, \ub}$. We shall use the Capital Latin $A,B, \cdots$ to denote the indices on the $2$-sphere. Then $e_A = e_A^\omega \p_\omega$ and
\begin{equation}\label{g-grad-u-ub-e}
g(D u, e_A)=e_A(u) =0, \quad
 g(D \ub, e_A)=e_A(\ub) =0, \quad A=1, \, 2,
\end{equation}
where $D u, \, D \ub$ are the gradients defined in \eqref{grad-u-ub-expansion-1}.
As a result, at each point $p \in S_{u, \ub}$, the linear space expanded by $Du$ and $D\ub$ composes of the orthogonal complement $T_p S_{u, \ub}^{\perp}$ relative to $T_p\mathcal{M}$. Setting
\begin{equation}\label{def-e3-e4}
\tilde e_3 = -2 Du + \beta D\ub, \quad
\tilde e_4 =  -2 D\ub + \alpha Du,
\end{equation}
we choose the parameters $\alpha, \, \beta$,
\begin{equation*}
\begin{split}
\alpha &=\frac{1+ \frac{2}{g} \eta^{u \mu} \p_\mu \phi \eta^{\ub \nu} \p_{\nu} \phi - \sqrt{1+ \frac{ 4}{g}  \eta^{u \mu} \p_\mu \phi \eta^{\ub \nu} \p_{\nu} \phi} }{g^{-1}  (\eta^{u \nu}\p_{\nu} \phi)^2 }, \\
\beta &=\frac{1+ \frac{2}{g}  \eta^{u \mu} \p_\mu \phi \eta^{\ub \nu} \p_{\nu} \phi - \sqrt{1+ \frac{ 4}{g} \eta^{u \mu} \p_\mu \phi \eta^{\ub \nu} \p_{\nu} \phi} }{g^{-1}  (\eta^{\ub \mu}\p_{\mu} \phi)^2 },
\end{split}
\end{equation*}
such that
\begin{equation*}
g(\tilde e_3, \tilde e_3) = g(\tilde e_4, \tilde e_4) = 0.
\end{equation*}
Alternative expressions for $\alpha$ and $\beta$ take the forms of
\begin{equation*}
\begin{split}
\alpha &=\frac{4(\eta^{\ub \nu} \p_{\nu} \phi)^2}{g+ 2  \eta^{u \mu} \p_\mu \phi \eta^{\ub \nu} \p_{\nu} \phi + \sqrt{g^2+ 4g \eta^{u \mu} \p_\mu \phi \eta^{\ub \nu} \p_{\nu} \phi}}, \\
\beta &=\frac{4(\eta^{u \mu} \p_{\mu} \phi)^2}{g + 2 \eta^{u \mu} \p_\mu \phi \eta^{\ub \nu} \p_{\nu} \phi+ \sqrt{g^2+ 4g \eta^{u \mu} \p_\mu \phi \eta^{\ub \nu} \p_{\nu} \phi}},
\end{split}
\end{equation*}
which are approximate to, after using Taylor expansion (for $\frac{2}{g} \eta^{u \mu} \p_\mu \phi \eta^{\ub \nu} \p_{\nu} \phi = O(r^{-2-\frac{1}{2}})$ is small when r tends to infinity),
\begin{equation}\label{approximate-alpha-beta}
\begin{split}
\alpha &\sim \frac{2 (\eta^{\ub \nu} \p_{\nu} \phi)^2}{g} \left( 1-  \frac{2}{g} \eta^{u \mu} \p_\mu \phi \eta^{\ub \nu} \p_{\nu} \phi \right) = \frac{ (\p_{u} \phi)^2}{2 g} \left( 1 - \frac{\p_u \phi \p_{\ub} \phi }{2g} \right), \\
\beta& \sim \frac{2 (\eta^{u \mu} \p_{\mu} \phi)^2}{g} \left(1 -  \frac{2}{g} \eta^{u \mu} \p_\mu \phi \eta^{\ub \nu} \p_{\nu} \phi \right) = \frac{(\p_{\ub} \phi)^2}{2 g} \left( 1 - \frac{\p_u \phi \p_{\ub} \phi }{2g} \right).
\end{split}
\end{equation}
We normalize these null vector fields by
\begin{equation*}
e_4 = \tilde e_4, \quad e_3 = -\frac{2 \tilde e_3}{g(\tilde e_3 , \tilde e_4)},
\end{equation*}
where
\begin{align*}
 \frac{-2}{g(\tilde e_3 , \tilde e_4)} &\sim  \frac{-2}{-2- g^{-1} \p_u \phi \p_{\ub} \phi + O(r^{-5})} \\
 & =  \frac{1}{1 + \frac{1}{2g} \p_u \phi \p_{\ub} \phi + O(r^{-5})} \sim 1- \frac{ \p_u \phi \p_{\ub} \phi}{2g},
 \end{align*}
 (noting that $\frac{1}{2g} \p_u \phi \p_{\ub} \phi = O(r^{-2-\frac{1}{2}})$  is small when r tends to infinity),
so that the null frame $\{ e_3, \, e_4\}$ is normalized as
\begin{equation*}
g(e_3, e_3) = g(e_4, e_4) = 0, \quad g(e_3, e_4) = -2.
\end{equation*}
In view of \eqref{g-grad-u-ub-e}, $\{e_3 ,e_4, e_A,  \, A =1,2 \}$ is a null frame on $\mathcal{M}$ with respect to $g_{\mu \nu}$.

By \eqref{approximate-alpha-beta}, $e_4$ can be expanded as follows
\begin{align*}
e_4 &= \p_{u} + \frac{\p_{\ub} \phi \p_u \phi}{2g} \p_u + \frac{(\p_u \phi)^2}{2g} \p_{\ub}  - \frac{\p_u \phi \p^\omega \phi}{g}  \p_\omega  \\
& \quad - \frac{(\p_u \phi)^2}{4g} \p_{\ub} - \frac{ \p_u \phi \p_{\ub} \phi (\p_u \phi)^2 }{8g^2} \p_{\ub} - \frac{(\p_{\ub} \phi \p_u \phi )^2 }{8g^2} \p_u\\
& \quad +  \frac{ (\p_{u} \phi)^2 \p_{\ub} \phi \p^\omega \phi}{4g^2} \p_\omega + \text{l.o.t.},
\end{align*}
which further simplifies as
\begin{align*}
e_4 &=  \p_{u} + \frac{\p_{\ub} \phi \p_u \phi}{2g} \left( 1- \frac{ \p_u \phi \p_{\ub} \phi }{4} \right)  \p_u + \frac{(\p_u \phi)^2}{4g} \left( 1- \frac{ \p_u \phi \p_{\ub} \phi }{2} \right) \p_{\ub}  \\
&\quad - \frac{\p_u \phi \p^\omega\phi}{g} \left( 1 - \frac{\p_u \phi \p_{\ub} \phi }{4} \right) \p_\omega + \text{l.o.t.},
\end{align*}
and
\begin{align*}
e_3 &= \left( 1- \frac{ \p_u \phi \p_{\ub} \phi }{2g} \right) \left( \p_{\ub} + \frac{\p_{\ub }\phi \p_u \phi}{2g} \p_{\ub} + \frac{(\p_{\ub}\phi)^2}{2g} \p_u - \frac{\p_{\ub} \phi \p^\omega \phi}{g} \p_\omega \right) \\
&\quad - \frac{(\p_{\ub}\phi)^2}{4g} \p_u -  \frac{ \p_u \phi  (\p_{\ub}\phi)^3}{8g^2} \p_u - \frac{(\p_u\phi \p_{\ub} \phi)^2}{8g^2} \p_{\ub} +  \frac{ (\p_{\ub} \phi)^2 \p_u \phi \p^\omega \phi}{4g^2}  \p_\omega + \text{l.o.t.} \\
&= \p_{\ub} - \frac{3 (\p_u \phi \p_{\ub} \phi)^2  }{8g}  \p_{\ub} + \frac{(\p_{\ub}\phi)^2}{4g}  \left( 1- \frac{3\p_u \phi \p_{\ub} \phi }{2} \right) \p_u \\
& \quad - \frac{\p_{\ub} \phi \p^\omega \phi}{g} \left( 1- \frac{ 3\p_u \phi \p_{\ub} \phi }{4} \right)  \p_\omega+ \text{l.o.t.}.
\end{align*}

\subsubsection{The connection coefficient}\label{sec-connection}
We present some quantitative estimates for the connection in the null frame $\{e_3 ,e_4, e_A,  \, A =1,2 \}$.
\begin{corollary}\label{lem-D-4-3-omega-null-infty}
\begin{subequations}
\begin{align}
D_{4} e_3 &= O(r^{-3}) e_3  + O(r^{-3 - \frac{1}{2}}) e_B,  \label{D-e4-e3}\\
D_{3} e_4 &= O(r^{-3 - \frac{1}{2}}) e_4  + O(r^{-3 - \frac{1}{2}}) e_B, \label{D-e3-e4}\\
D_{A} e_3 &=   \frac{1}{r} e_A + O(r^{-4}) e_B  + O(r^{-3 - \frac{1}{2}}) e_3,  \label{D-A-3} \\
D_{A} e_4 &=  \left( - \frac{1}{r} + \frac{ (\p_u \phi)^2 }{4r} \right) e_A + O(r^{-3 - \frac{1}{2}}) e_B  + O(r^{-3 - \frac{1}{2}}) e_4, \label{D-A-4}  \\
D_{4} e_4 &= O(r^{-3}) e_4  + O(r^{-3}) e_A. \label{D-4-4-null-frame}
\end{align}
\end{subequations}
Here $e_B \in TS_{u,\ub}$, the tangent bundle of $S_{u, \ub}$.
\end{corollary}

\begin{remark}
In particular, \eqref{D-4-4-null-frame} suggests that the null vector field $e_4$ is close to the affine (incoming) null vector field.
As shown in the proof followed, our choice of $e_4$ generates a key cancellation \eqref{cancellation-D4-4-eA}, which provides fast decay rate for $D_{e_4} e_4$.
Otherwise, for generic choice of $e_4$, the cancellation \eqref{cancellation-D4-4-eA} does not necessarily hold true and thus we would only expect
\begin{equation}\label{eq-D-e4-e4-relaxed}
D_{4} e_4 = O(r^{-3}) e_4  + O(r^{-2-\frac{1}{2}}) e_A.
\end{equation}
which does not decay fast enough for our analysis later.
\end{remark}

\begin{proof}
We here only compute $g(D_A e_3, e_B)$, $g(D_A e_4, e_B)$ and $D_4 e_4$. The rest of proof is collected in Appendix \ref{sec-calculations}.

First of all,  there are $g(D_A e_3, e_B) = -g(D_A e_B, e_3)$ and $g(D_A e_4, e_B)= -g(D_A e_B, e_4)$.
Notice that $D_{A} e_B = e_A^\varphi \p_\varphi e_B^{\omega} \p_\omega + e_A^\varphi  e_B^{ \omega} \Gamma_{\varphi \omega}^\sigma \p_\sigma$. In view of \eqref{christoffel-g-eta} and  \eqref{christoffel-eta}, we have
\begin{align*}
e_A^{\varphi} e_B^{ \omega} \Gamma_{\varphi \omega}^\sigma \p_\sigma & = e_A^\varphi e_B^{\omega} \Gamma_{\varphi \omega}^\sigma (\eta) \p_\sigma  - g^{-1} e_A^\varphi e_B^{\omega} \Gamma_{\varphi \omega}^\sigma (\eta) \p_\sigma \phi  \p^\alpha \phi  \p_\alpha \\
&\quad + e_A^\varphi e_B^\omega  \p_\varphi \p_\omega \phi g^{\alpha \beta}\p_\beta \phi \p_\alpha,
\end{align*}
and
 \begin{equation}\label{eq-gamma-A-B-eta-p-phi}
 \begin{split}
e_A^{\varphi} e_B^{ \omega}  \Gamma_{\varphi \omega}^\sigma (\eta) \p_\sigma \phi &=e_A^{\varphi} e_B^{ \omega} \left(  \Gamma_{\varphi \omega}^\theta (\eta) \p_\theta \phi  +   \Gamma_{\varphi \omega}^u (\eta) \p_u \phi +   \Gamma_{\varphi \omega}^{\ub} (\eta) \p_{\ub} \phi  \right) \\
&=  e_A^{\varphi} e_B^{ \omega} \left( \Gamma_{\varphi \omega}^\theta (\eta) \p_\theta \phi +  \frac{\eta_{\varphi \omega}}{2r}  \p_u \phi - \frac{\eta_{\varphi \omega}}{2r}  \p_{\ub} \phi \right) = O(r^{-2}).
 \end{split}
 \end{equation}
 We remark that $g(\p_\sigma, e_\nu) = \eta(\p_\sigma, e_\nu) + \p_\sigma \phi e_\nu \phi$ and by the formulae of $e_4, \, e_3$,
 \begin{equation}\label{eta-coordinate-filed-null-frame}
 \begin{split}
 \eta (\p_{\ub}, e_4) &= -2 - \p_u \phi \p_{\ub} \phi + O(r^{-5}) = -2+ O(r^{-2-\frac{1}{2}}), \\
 \eta (\p_u, e_4) &=- \frac{1}{2} (\p_u \phi)^2 + O(r^{-4-\frac{1}{2}}), \quad  \eta (e_A, e_4) = - \p_u \phi e_A \phi + O(r^{-5}),\\
 \eta (\p_{\ub}, e_3) &= - \frac{1}{2} (\p_{\ub} \phi)^2 + O(r^{-5-\frac{1}{2}}) = O(r^{-3}),\\
\eta (\p_u, e_3) &=- 2 + O(r^{-5}), \quad   \eta (e_A, e_3) =-\p_{\ub} \phi e_A \phi + O(r^{-5 - \frac{1}{2} }).
 \end{split}
 \end{equation}
As a result, we achieve
\begin{align*}
e_A^{\varphi} e_B^{ \omega} \Gamma_{\varphi \omega}^\sigma g(\p_\sigma, e_4) &= e_A^\varphi e_B^{\omega} \Gamma_{\varphi \omega}^\theta (\eta) (-\p_u \phi \p_\theta \phi) + O(r^{-6})    \\
& \quad + e_A^\varphi e_B^{\omega} \Gamma_{\varphi \omega}^u (\eta) (-\frac{1}{2}(\p_u \phi)^2 + O(r^{-4-\frac{1}{2}}))    \\
& \quad   +  e_A^\varphi e_B^\omega \Gamma_{\varphi \omega}^{\ub} (\eta) (-2 - \p_u \phi \p_{\ub} \phi + O(r^{-5})) \\
& \quad + e_A^\varphi e_B^{\omega} \Gamma_{\varphi \omega}^\sigma (\eta) \p_\sigma \phi e_4 \phi  - g^{-1} e_A^\varphi e_B^{\omega} \Gamma_{\varphi \omega}^\sigma (\eta) \p_\sigma \phi e_4 \phi \\
&\quad - g^{-1} e_A^\varphi e_B^{\omega} \Gamma_{\varphi \omega}^\sigma (\eta) \p_\sigma \phi \p^{\alpha} \phi \p_{\alpha} \phi e_4 \phi \\
&\quad + e_A^\varphi e_B^\omega  \p_\varphi \p_\omega \phi e_4 \phi,
\end{align*}
where the 4th and 5th lines cancel:
\begin{align*}
&e_A^\varphi e_B^{\omega} \Gamma_{\varphi \omega}^\sigma (\eta) \p_\sigma \phi e_4 \phi  - g^{-1} e_A^\varphi e_B^{\omega} \Gamma_{\varphi \omega}^\sigma (\eta) \p_\sigma \phi e_4 \phi \\
&\quad - g^{-1} e_A^\varphi e_B^{\omega} \Gamma_{\varphi \omega}^\sigma (\eta) \p_\sigma \phi \p^{\alpha} \phi \p_{\alpha} \phi e_4 \phi =0.
\end{align*}
Therefore,
\begin{align*}
e_A^{\varphi} e_B^{ \omega} \Gamma_{\varphi \omega}^\sigma g(\p_\sigma, e_4) &= \frac{1}{r} e_A^\varphi e_B^\omega \eta_{\varphi \omega}  -  \frac{ (\p_u \phi)^2}{4r}  e_A^\varphi e_B^{\omega} \eta_{\varphi \omega} + \frac{ \p_u \phi \p_{\ub} \phi }{2r}  e_A^\varphi e_B^\omega \eta_{\varphi \omega}  \\
& \quad - e_A^\varphi e_B^{\omega} \Gamma_{\varphi \omega}^\theta (\eta) \p_\theta \phi  \p_u \phi + e_A^\varphi e_B^\omega  \p_\varphi \p_\omega \phi e_4 \phi + O(r^{-5-\frac{1}{2}}).
\end{align*}
And in the same way, there is
\begin{align*}
e_A^{\varphi} e_B^{ \omega} \Gamma_{\varphi \omega}^\sigma g(\p_\sigma, e_3) &= e_A^\varphi e_B^{\omega} \Gamma_{\varphi \omega}^\theta (\eta) (-\p_{\ub} \phi \p_\theta \phi) +  O(r^{-6- \frac{1}{2} })    \\
& \quad + e_A^\varphi e_B^{\omega} \Gamma_{\varphi \omega}^u (\eta) (-2 + O(r^{-5}))    \\
& \quad   +  e_A^\varphi e_B^\omega \Gamma_{\varphi \omega}^{\ub} (\eta) (- \frac{1}{2} (\p_{\ub} \phi)^2 + O(r^{-5-\frac{1}{2}})) \\
& \quad + e_A^\varphi e_B^{\omega} \Gamma_{\varphi \omega}^\sigma (\eta) \p_\sigma \phi e_3 \phi  - g^{-1} e_A^\varphi e_B^{\omega} \Gamma_{\varphi \omega}^\sigma (\eta) \p_\sigma \phi e_3 \phi  \\
&\quad  - g^{-1} e_A^\varphi e_B^{\omega} \Gamma_{\varphi \omega}^\sigma (\eta) \p_\sigma \phi \p^{\alpha} \phi \p_{\alpha} \phi e_3 \phi \\
&\quad + e_A^\varphi e_B^\omega  \p_\varphi \p_\omega \phi e_3 \phi,
\end{align*}
where the 4th and 5th lines cancel:
\begin{align*}
& e_A^\varphi e_B^{\omega} \Gamma_{\varphi \omega}^\sigma (\eta) \p_\sigma \phi e_3 \phi  - g^{-1} e_A^\varphi e_B^{\omega} \Gamma_{\varphi \omega}^\sigma (\eta) \p_\sigma \phi  e_3 \phi \\
&\quad  - g^{-1} e_A^\varphi e_B^{\omega} \Gamma_{\varphi \omega}^\sigma (\eta) \p_\sigma \phi \p^{\alpha} \phi \p_{\alpha} \phi e_3 \phi =0.
\end{align*}
Consequently, we obtain
\begin{align*}
e_A^{\varphi} e_B^{ \omega} \Gamma_{\varphi \omega}^\sigma g(\p_\sigma, e_3) &=-\frac{1}{r} e_A^\varphi e_B^\omega \eta_{\varphi \omega} + \frac{1}{4r}  (\p_{\ub} \phi)^2 e_A^\varphi e_B^{\omega} \eta_{\varphi \omega} \\
& \quad - e_A^\varphi e_B^{\omega} \Gamma_{\varphi \omega}^\theta (\eta) \p_\theta \phi \p_{\ub} \phi  + e_A^\varphi e_B^\omega  \p_\varphi \p_\omega \phi e_3 \phi +O(r^{-6}).
\end{align*}
As a summary, we arrive at
\begin{subequations}
\begin{align}
e_A^{\varphi} e_B^{ \omega} \Gamma_{\varphi \omega}^\sigma g(\p_\sigma, e_4) &= \left( \frac{1}{r}  -  \frac{ (\p_u \phi)^2}{4r}  \right) g_{AB} + \frac{ \p_u \phi \p_{\ub} \phi }{2r}  g_{AB} - \nablaslash_A \nablaslash_B \phi e_4 \phi + O(r^{-4}), \label{gamma-A-B-e4} \\
e_A^{\varphi} e_B^{ \omega} \Gamma_{\varphi \omega}^\sigma g(\p_\sigma, e_3)
 &= - \frac{1}{r}g_{AB} + \frac{ (\p_{\ub} \phi)^2}{4r} g_{AB}  + \frac{e_A \phi e_B \phi }{r} - \nablaslash_A \nablaslash_B \phi e_3 \phi + O(r^{-6}).\label{gamma-A-B-e3}
\end{align}
\end{subequations}

We now calculate $g(D_A e_B, e_3)$ and $g(D_A e_B, e_4)$.
By virtue of \eqref{gamma-A-B-e4},
\begin{align*}
g(D_{A} e_B, e_4) &= e_A^\varphi \p_\varphi e_B^{\omega} g(\p_{\omega}, e_4) + e_A^\varphi  e_B^{ \omega} \Gamma_{\varphi \omega}^\sigma  g(\p_\sigma, e_4) \\
&=  \left( \frac{1}{r}  -  \frac{ (\p_u \phi)^2}{4r}  \right) g_{AB}  +  O(r^{-3- \frac{1}{2}}),
\end{align*}
where we notice that $ e_A^\varphi \p_\varphi e_B^{\omega} \p_\omega \phi e_4 \phi = O(r^{-3 - \frac{1}{2}})$ and $ \frac{ \p_u \phi \p_{\ub} \phi }{2r} g_{AB} = O(r^{-3- \frac{1}{2}})$.
That is,
\begin{equation}\label{eq-DA-4-B}
g(D_{A} e_4, e_B) = - g(D_{A} e_B, e_4) =  \left( - \frac{1}{r}  +  \frac{ (\p_u \phi)^2}{4r}  \right) g_{AB}  +  O(r^{-3- \frac{1}{2}}).
\end{equation}
Using \eqref{gamma-A-B-e3}, we have
\begin{align*}
g(D_{A} e_B, e_3) &= e_A^\varphi \p_\varphi e_B^{\omega} g(\p_{\omega}, e_3) + e_A^\varphi  e_B^{ \omega} \Gamma_{\varphi \omega}^\sigma  g(\p_\sigma, e_3) \\
&= - \frac{1}{r} g_{AB}  +  O(r^{-4}),
\end{align*}
where we have used $e_A^\varphi \p_\varphi e_B^{\omega} \p_\omega \phi e_3 \phi =O(r^{-4})$ and $ \frac{ (\p_{\ub} \phi)^2}{4r} g_{AB}  + \frac{e_A \phi e_B \phi }{r} - \nablaslash_A \nablaslash_B \phi e_3 \phi = O(r^{-4})$. And hence,
\begin{equation}\label{eq-DA-3-B}
g(D_{A} e_3, e_B) = - g(D_{A} e_B, e_3) = \frac{1}{r} g_{AB}  +  O(r^{-4}).
\end{equation}

With the help of \eqref{eq-DA-4-B}, \eqref{eq-DA-4-3} and \eqref{eq-DA-3-B}, \eqref{eq-DA-3-4},  \eqref{D-A-4} and \eqref{D-A-3} follow.

We next turn to $D_{e_4} e_4 = e_4^\mu \p_\mu e_4^{\nu} \p_\nu + e_4^\mu  e_4^{ \nu} \Gamma_{\mu \nu}^\sigma \p_\sigma$. From \eqref{christoffel-g-eta} and  \eqref{christoffel-eta}, we see that
\begin{align*}
e_4^\mu  e_4^{ \nu} \Gamma_{\mu \nu}^\sigma \p_\sigma & = e_4^{\mu} e_4^{ \nu} \Gamma_{\mu \nu}^\sigma (\eta) \p_\sigma  - g^{-1} e_4^{\mu} e_4^{ \nu} \Gamma_{\mu \nu}^\sigma (\eta) \p_\sigma \phi  \p^\alpha \phi  \p_\alpha \\
&\quad +e_4^{\mu} e_4^{ \nu} \p_\mu \p_\nu \phi g^{\alpha \beta}\p_\beta \phi \p_\alpha,
\end{align*}
and
\begin{align*}
e_4^\mu  e_4^{ \nu} \Gamma_{\mu \nu}^\sigma (\eta) \p_\sigma \phi & = \left( 2 e_4^\omega e_4^{\ub} \Gamma_{\omega \ub}^\theta (\eta) + 2 e_4^\omega e_4^{u} \Gamma_{\omega u}^\theta (\eta) + e_4^\omega e_4^\varphi \Gamma_{\omega \varphi}^\theta (\eta) \right) \p_\theta \phi \\
&\quad +  e_4^\omega e_4^{\theta}\Gamma_{\omega \theta}^u (\eta) \p_u \phi +  e_4^\omega e_4^{\theta}\Gamma_{\omega \theta}^{\ub} (\eta) \p_{\ub} \phi  \\
&= \frac{(\p_u \phi)^2}{r g^2}  |\nablaslash \phi|^2 \p_u \phi - \frac{(\p_u \phi)^2 }{r g^2} |\nablaslash \phi|^2 \p_{\ub} \phi + \text{l.o.t.}\\
&\quad + \left( 2 \frac{ \p_u \phi \p^\theta \phi}{r g}  - \frac{ (\p_{u} \phi)^3 \p^\theta \phi}{2rg} + (\p_u \phi)^2 \p^\omega \phi \p^\varphi \phi \Gamma_{\omega \varphi}^\theta (\eta) \right)  \p_\theta \phi\\
& = O(r^{-5}).
\end{align*}
It follows that
\begin{equation}\label{gamma-4-4-eA}
\begin{split}
e_4^\mu  e_4^{ \nu} \Gamma_{\mu \nu}^\sigma g(\p_\sigma, e_A) &=e_4^\mu  e_4^{ \nu} \Gamma_{\mu \nu}^{\theta} (\eta) \eta_{\theta A} + e_4^{\mu} e_4^{ \nu}  \p_\mu \p_\nu \phi e_A \phi +  O(r^{-6-\frac{1}{2}}),\\
&= \frac{2}{r} \p_u \phi e_A \phi + e_4^{\mu} e_4^{ \nu}  \p_\mu \p_\nu \phi e_A \phi +  O(r^{-5-\frac{1}{2}}).
\end{split}
\end{equation}

Substituting the formula of $e_4$ and making use of \eqref{gamma-4-4-eA}, we derive
\begin{align*}
g(D_{4} e_4,e_B) &= e_4^\mu \p_\mu e_4^\nu g(\p_\nu, e_B) +  e_4^\mu e_4^\nu \Gamma_{\mu \nu}^\sigma g(\p_\sigma, e_B) \\
&= e_4  \frac{(\p_u \phi)^2}{4g} g_{\ub B} +  e_4 \left( \frac{\p_{\ub} \phi \p_u \phi}{2g} \right) g_{u B} -  e_4 \frac{\p_u \phi \p^\omega\phi}{g} \left( 1 - \frac{\p_u \phi \p_{\ub} \phi }{4} \right) g_{\omega B} \\
& \quad +e_4^{\mu} e_4^{ \nu}  \p_\mu \p_\nu \phi e_B \phi +2 r^{-1} \p_u \phi e_B \phi  + o(r^{-5}).
\end{align*}
As a consequence,
\begin{equation}\label{cancellation-D4-4-eA}
\begin{split}
g(D_4  e_4, e_B) &  = -  e_4 \frac{\p_u \phi \p^\omega\phi}{g} g_{\omega B}  + e_4^{\mu} e_4^{ \nu}  \p_\mu \p_\nu \phi e_B \phi + 2 r^{-1} \p_u \phi e_B \phi + O(r^{-5}) \\
&= -\p_u \phi  e_4 \p^\omega\phi g_{\omega B} +2 r^{-1} \p_u \phi e_B \phi  + O(r^{-5}),
\end{split}
\end{equation}
which implies
\begin{equation}\label{eq-D4-4-B}
g(D_4  e_4, e_B) =-e_4^\mu e_B^\omega \p_\mu \p_\omega\phi  \p_u \phi   + O(r^{-5})  = O(r^{-3}).
\end{equation}
We should remark that,  the cancellation $-  e_4 \frac{\p_u \phi \p^\omega\phi}{g} g_{\omega B}  + e_4^{\mu} e_4^{ \nu}  \p_\mu \p_\nu \phi e_B \phi = -\p_u \phi  e_4 \p^\omega\phi g_{\omega B} + O(r^{-5})$ in \eqref{cancellation-D4-4-eA} is crucial. If this is ignored, we would obtain $g(D_4  e_4, e_B) =O(r^{-2-\frac{1}{2}})$,
which is not good enough.

On the other hand, based on \eqref{eq-D4-3-4},
\begin{equation}\label{eq-D4-4-3}
 g(D_4  e_4, e_3) = -g(D_4  e_3, e_4) = -e_4^{\mu} e_3^{ \nu}  \p_\nu \p_\mu \phi e_4 \phi+O(r^{-5}) = O(r^{-3}).
\end{equation}
Combining \eqref{eq-D4-4-B} with \eqref{eq-D4-4-3}, we prove \eqref{D-4-4-null-frame}.

\end{proof}

The second fundamental forms of $S_{u, \ub}$ along $e_3, \, e_4$ are defined as
\begin{equation*}
\chi_{AB} = g(e_A, D_{e_B} e_3), \quad \underline{\chi}_{AB} = g(e_A, D_{e_B} e_4).
\end{equation*}
Since $\{e_A, e_B\} \subset TS_{u,\ub}$, where $ TS_{u,\ub}$ is the tangent bundle of $S_{u, \ub}$, we have, by the Frobenius theorem, $\chi_{AB}$ and $\underline{\chi}_{AB}$ are both symmetric tensors on $S_{u, \ub}$. Let $\gslash_{AB}$ be the metric on $S_{u, \ub}$ induced from $g_{\mu \nu}$ and define the trace and traceless parts:
\begin{align*}
&\text{tr} \chib \doteq \gslash^{AB} \underline{\chi}_{AB}, &\hat \chib_{AB} \doteq \chib_{AB} - \frac{1}{2}\text{tr} \underline{\chi} \gslash_{AB}, \\
& \text{tr} \chi \doteq\gslash^{AB} \chi_{AB}, &\hat \chi_{AB} \doteq \chi_{AB} - \frac{1}{2}\text{tr} \chi \gslash_{AB},
\end{align*}
where $\gslash^{AB} \doteq \gslash^{-1}_{AB}$ is the inverse of $\gslash_{AB}$. We remark that $g^{\omega \theta} \neq  \gslash^{\omega \theta}$. In fact,
\begin{equation*}
\gslash^{\omega \theta}=\frac{1}{( \det (\gslash_{\omega \theta}) )^2 } \left(
\begin{array}{cc}
g_{\theta \theta} & - g_{\theta \omega}\\
- g_{\omega \theta} & g_{\omega \omega}
\end{array}\right).
\end{equation*}

Recall that $\gamma = \di \theta^2 + \sin^2 \theta \di \varphi^2$ is the standard metric on $S^2$.  We can calculate $\det ( \gslash_{\omega \theta}) = \det ( \eta_{\omega \theta} + \p_\omega \phi \p_{\theta} \phi ) =  r^{4} \det( \gamma_{\omega \theta}) \left( 1+ \eta^{\omega \theta} \p_\omega \phi \p_{\theta} \phi \right)$.
Then the area of $S_{u,\ub}$ is
\begin{equation*}
4 \pi \bar{r}^2 \doteq  \int_{S_{u,\ub}} \di \mu_{\gslash} =\int_{S_{u,\ub}}  r^{2}  \sqrt{ 1+ \eta^{\omega \theta} \p_\omega \phi \p_{\theta} \phi } \di \mu_{\gamma},
\end{equation*}
which entails
\begin{equation}\label{compare-bar-r-r}
4 \pi \bar{ r}^2 = 4 \pi r^2 + O(r^{-1}). 
\end{equation}

As a result of \eqref{D-A-3} and \eqref{D-A-4}, we have the following lemma.
\begin{lemma}\label{lem-2nd-form-1}
The second fundamental forms along $e_3$ and $e_4$  obey
\begin{subequations}
\begin{align}
\text{tr} \underline{\chi} & =  -\frac{2}{r}  +O(r^{-3}),  &\hat{\chib}_{AB}  =  O(r^{-3}),  \label{lem-chi-b-AB-trace-app} \\
\text{tr} \chi & = \frac{2}{r} + O(r^{-4}),   &\hat{\chi}_{AB}  =  O(r^{-4}).   \label{chi-AB-trace-app}
\end{align}
\end{subequations}
More explicitly, there is
\begin{equation}\label{eq-tr-chib}
\text{tr} \underline{\chi}  =  -\frac{2}{r}  + \frac{(\p_u\phi)^2}{2r} +O(r^{-3-\frac{1}{2}}).
\end{equation}
\end{lemma}

In the null frame $\{e_3, e_4, e_A, A=1,2\}$, the membrane equation $\Box_g \phi = -D_3 D_4 \phi + \gslash^{A B} D_A D_{B} \phi =0$ can be split into
\begin{equation}\label{eq-membrane-null-frame}
-D_3 D_4 \phi  - \frac{1}{2}  \tr\chib e_3 \phi - \frac{1}{2}  \tr \chi e_4\phi + \laplacianslash_{\gslash}\phi = 0,
\end{equation}
where $\laplacianslash_{\gslash}$ is the Laplacian with respect to $\gslash$. As a remark, taking advantage of the membrane equation \eqref{eq-membrane-null-frame}, we have
 \begin{equation}\label{eq-D3D4-DADB-phi-estimate}
D_3 D_4 \phi = \gslash^{A B} D_A D_{B} \phi = - \frac{1}{r} e_4\phi + O(r^{-2-\frac{1}{2}}).
\end{equation}

\subsubsection{Null infinity of the membrane}
Let $s$ be the parameter of $e_3$  such that $e_3(s) = 1$.
And let $\sigma_q(s)$ be the integral curve of $e_3$ starting from $q$. Namely, $\frac{\di \sigma_q(s)}{\di s} =e_3,$ with the initial condition $\sigma_q(0) = q$. We will drop the subscription $q$ to denote the curve by $\sigma(s)$ for convenience.

Let $\Psi_s: \mathcal{M} \longrightarrow \mathcal{M}$ be the 1-parameter group generated by $e_3$ as follows. For any $q \in \mathcal{M}$, we define $\Psi_s (q) = \sigma_q(s)$. $\Psi_0$ is the identity map. It is easy to check that $\Psi_s \circ \Psi_{\tau} = \Psi_{s+\tau}$.

The estimates below are done along each curve $\sigma(s)$.
\begin{corollary}\label{coro-along-e3}
Along $\sigma(s)$, there is an $s_0$ large enough, such that for all $s > s_0$, $\ub(s) \sim s$.
\end{corollary}
\begin{proof}
Along the curve $\sigma(s)$, there is
\begin{equation*}
\frac{\di \ub(s)}{\di s} = e_3\left( \ub \right)= 1+ \varepsilon^2 O\left( \ub^{-5} \right).
\end{equation*}
Then we find that $\ub(s) \rightarrow +\infty,$ as $s \rightarrow +\infty$. Additionally,
\begin{align*}
\lim_{s \rightarrow +\infty} \frac{\ub(s)}{s} = \lim_{s \rightarrow +\infty} \frac{\di \ub(s)}{\di s} =1.
\end{align*}
Then there exists an $s_0$ large enough, such that for all $s > s_0$, $\frac{1}{2} < \frac{\ub(s)}{s} <\frac{3}{2}$, namely, $\ub(s)$ and $s$ are equivalent.
\end{proof}

We notice that along  $\sigma(s)$, there is $\lim\limits_{s \rightarrow +\infty} e_4 = \lim\limits_{\ub \rightarrow +\infty} e_4 = \p_u$, hence we can also parametrize the null infinity of the membrane geometry $\mathcal{I}^+$ by $(u, \omega) \in I \times S^2$  and identify $\mathcal{I}^+$ as $ r \rightarrow + \infty$.

For any $f \in C^\infty(\mathcal{M})$, if $ \lim\limits_{s \rightarrow \infty} f (\Psi_s(p))$ exists, we can define $f^\ast \in C^\infty (\mathcal{I}^+)$ by: $\forall q \in \mathcal{I}^+$, there exists a point $p \in \mathcal{M}$ such that $q = \lim\limits_{s \rightarrow \infty} \Psi_s (p)$, then $f^\ast (q) =  \lim\limits_{s \rightarrow \infty} f (\Psi_s(p))$. Alternatively, it reads in coordinates as $f^\ast (u, \omega) =  \lim\limits_{\ub \rightarrow \infty} f (\ub, u, \omega) =  \lim\limits_{r \rightarrow \infty} f (r, u, \omega)$.
And we have $(e_4 f)^\ast  =\p_u  f^\ast$, i.e., $ \lim\limits_{r \rightarrow \infty} e_4 f =\p_u \left( \lim\limits_{r \rightarrow \infty}  f \right)$.
Due to the decay result \eqref{decay-mem}, we have
\begin{corollary}\label{coro-radiation}
The following limitations exist,
\begin{equation}\label{def-Xi-Sigma-1-Psi}
\Xi_1\doteq \lim\limits_{r \rightarrow +\infty} \left( r  e_4  \phi \right), \quad \Xi_2 \doteq \lim\limits_{r \rightarrow +\infty} \left( r e^2_4  \phi \right) = \p_u \Xi_1,
\end{equation}
where $\Xi_1, \, \Xi_2 \in C^\infty(\mathcal{I}^+)$.
\end{corollary}

\subsection{Expanding effect}\label{sec-expanding}
\subsubsection{Expansion (along $e_4$) projected onto $S_{u,\ub}$}\label{sec-expanding-S}
Lemma \ref{lem-2nd-form-1} (specifically \eqref{eq-tr-chib}) tells an expanding effect along the bundle of flow lines of $e_4$ near null infinity (see Remark \ref{rk-expanding}).
\begin{theorem}\label{lem-2nd-form}
Defining $\hb \doteq r \text{tr}  \chib + 2$, we have
\begin{equation}\label{asym-hb}
 \hb=\frac{\Hb}{r^2} + O(r^{-2-\frac{1}{2}}) \quad \text{with} \quad \Hb = \frac{1}{2} (\Xi_1)^2.
\end{equation}
Namely, $\lim\limits_{r \rightarrow \infty} r^2 \hb = \Hb >0$.
\end{theorem}
\begin{remark}\label{rk-expanding}
 As we know, $\hb=0$ in Minkowski spacetime.
Therefore, \eqref{asym-hb} with $\Hb > 0$ is related to the nonlinear effect of $g_{\mu \nu}(\p\phi)$. Meanwhile, $\chib_{AB}$ is related to shape of the cross-sectional area enclosing a fixed bundle of flow lines of $e_4$ projected onto $S_{u, \ub}$. In particular, $\text{tr} \chib$ is nothing but the rate of change of the volume of the projected (onto $S_{u, \ub}$) cross-sectional area of the bundle of flow lines of $e_4$.
 Thus, here $\Hb > 0$ will cause distortions and expansions for the bundles of flow lines along $e_4$ near null infinity of the membrane, see Figure \ref{fig:expanding}.
\end{remark}

In what follows, we will investigate the Raychaudhuri equation (see Lemma \ref{lem-Raychaudhuri} and refer to \cite{Hawking-Ellis}) to justify that the expanding effect shown in Theorem \ref{lem-2nd-form} is caused by the curvature of the membrane. We shall first introduce an important lemma.

\begin{lemma}\label{coro-main-Ray-eq}
We have
\begin{equation}\label{Ray-5}
 D_{4} \tr \chib = - \frac{1}{2}\tr^{2} \chib - R_{4 A 4 B} \gslash^{AB} +O(r^{-4}).
\end{equation}
\end{lemma}
\begin{remark}\label{rk-key-cancellation}
If we chose generic $e_4$ such that only the weak decay estimate \eqref{eq-D-e4-e4-relaxed} holds for $D_4 e_4$, then instead of \eqref{Ray-5}, we would have here
 \begin{equation*}
D_{4} \tr \chib = - \frac{1}{2}\tr^{2} \chib - R_{4 A 4 B} \gslash^{AB} +O(r^{-3-\frac{1}{2}}),
\end{equation*}
which will not be good enough for our purpose, i.e., to infer Theorem \ref{thm-memory}.
\end{remark}

\begin{proof}
Based on the Raychaudhuri equation \eqref{Ray-4} along $e_4$, we need to show that the error terms $ - \gslash^{AB} \gslash^{CD} \hat \chib_{C B} \hat \chib_{D A} +\gslash^{AB}  \left( D_A D_{4} e_{4B} + \frac{1}{2 } \left( 2g \left( D_A e_4, e_3 \right) + D_{3} e_{4A} + D_{4} e_{3A}  \right) D_{4}  e_{4B} \right)$ are $O(r^{-4})$.

At first, by \eqref{eq-D4-4-B}, \eqref{D-A-4} and \eqref{D-4-4-null-frame}
\begin{align*}
\gslash^{AB} D_A D_{4} e_{4B} & = \gslash^{AB} e_A g\left( D_4 e_4, e_B \right) -  \gslash^{AB} g \left(  D_{4} e_4, D_A e_B \right) \\
& = - \gslash^{AB} e_A (e_4^\mu e_B^\omega \p_\mu \p_\omega \phi \p_u \phi) -  \gslash^{AB} \cdot g(D_4 e_4, \Dslash_A e_B) \\
& \quad +  r^{-1} \cdot g(D_4 e_4, e_3) +O (r^{-5}) = O(r^{-4}).
\end{align*}
Here, we use $\Dslash$ to denote the connection associated to $\gslash$, then $\Dslash_A e_B \in TS_{u, \ub}$ and $g(\Dslash_A e_B, \Dslash_A e_B) \sim \frac{1}{r^2}$.

Next, by virtue of Corollary \ref{lem-D-4-3-omega-null-infty}, the remaining quadratic terms share the estimates
 $$\frac{1}{2 } \gslash^{AB} \left( 2g \left( D_A e_4, e_3 \right) + D_{3} e_{4A} + D_{4} e_{3A}  \right) D_{4}  e_{4B} = O(r^{-6-\frac{1}{2}})$$ and $ \gslash^{AB} \gslash^{CD} \hat \chib_{C B} \hat \chib_{D A} = O(r^{-6})$. Then, these estimates together with \eqref{Ray-4} yield \eqref{Ray-5}.
\end{proof}

Thanks to Lemma \ref{coro-main-Ray-eq}, we are now in a position to prove the following theorem, which indicates that nonlinear effect: $\Hb = \frac{1}{2} (\Xi_1)^2 >0$ is due to the curved geometry of membrane.
\begin{theorem}\label{thm-memory}
Let $\Hb$ be defined as that in Theorem \ref{lem-2nd-form}, then
\begin{equation}\label{Ray-weigh-lim-0-1}
\frac{\p}{\partial u} \Hb = -  \lim\limits_{r \rightarrow \infty}  r^{3} R_{4  4} = \frac{1}{2} \partial_u (\Xi_1)^2.
\end{equation}
\end{theorem}
\begin{proof}
We consider
\begin{equation}\label{eq-D4-r2-hb}
\begin{split}
 &\quad D_{4}  \left(  r^{2} \hb \right)  = D_{4}  \left(  r^{3} \tr \chib + 2  r^{2} \right) \\
 &=3  r^{2} D_4 r \cdot \tr \chib +   r^{3} D_4 \tr \chib  +  4   r D_4 r.
\end{split}
\end{equation}
Meanwhile, in view of Lemma \ref{lem-2nd-form-1} and the formula for $e_4$, we have $D_4 r = \frac{ r }{2} \tr \chib+O(r^{-2}).$  What is more, substituting \eqref{Ray-5} into \eqref{eq-D4-r2-hb},
and recalling that $\hb =O(r^{-2})$ and $\tr \underline\chi \sim -\frac{2}{r}$, we achieve
\begin{equation}\label{Ray-weigh-2}
\begin{split}
D_4  \left(   r^{2} \hb \right) &=  r^2 \tr \underline\chi \hb  -   r^{3}  R_{4 A 4 B} \gslash^{A B}   +  O(r^{-1}) \\
 &= -   r^{3} R_{4 A 4 B} \gslash^{A B}  + O(r^{-1}).
\end{split}
\end{equation}
Since $ \lim\limits_{r \rightarrow \infty} D_4 \left(  r^{2} \hb \right) = \frac{\p}{\partial u} \Hb$, \eqref{Ray-weigh-2} leads to
\begin{equation}\label{Ray-weigh-lim}
 \frac{\p}{\partial u} \Hb = -  \lim\limits_{r \rightarrow \infty}   r^{3}  R_{4A 4B} \gslash^{AB}.
\end{equation}
For the curvature term above, we make use of the Gauss equation \eqref{eq-Gauss} to derive
\begin{align*}
 R_{4 A 4 B} \gslash^{A B} & = k_{44} k_{A B}  \gslash^{A B} - k_{4A} k_{4 B}  \gslash^{A B}\\
 &= g^2 \left( D_4 D_4 \phi \cdot D_A D_{B} \phi - D_{A} D_4 \phi \cdot D_{B} D_4 \phi \right)  \gslash^{A B}.
\end{align*}
The decay rates for $e_A e_4 \phi, \, e_A \phi$ and $D_Ae_4$ (see Corollary \ref{lem-D-4-3-omega-null-infty}) suggest that $D_{A} D_4 \phi \cdot D_{B} D_4 \phi  \gslash^{A B} = O(r^{-4})$, which further implies
\begin{equation*}
 -  \lim\limits_{r \rightarrow \infty}  r^{3} R_{4 A 4 B} \gslash^{A B} =  -  \lim\limits_{r \rightarrow \infty} r^{3} D_4 D_4 \phi \cdot D_A D_{B} \phi  \gslash^{A B}.
\end{equation*}
We remind ourselves \eqref{eq-D3D4-DADB-phi-estimate}, namely $\gslash^{A B} D_A D_{B} \phi = - \frac{1}{r} e_4\phi +O(r^{-2-\frac{1}{2}})$. Hence
\begin{equation*}
 -  \lim\limits_{r \rightarrow \infty} r^{3} R_{4A 4 B} \gslash^{A B} =  \lim\limits_{r \rightarrow \infty} r^{2} D_4 D_4 \phi \cdot D_4\phi = \frac{1}{2} \partial_u (\Xi_1)^2.
\end{equation*}
Obviously, there is  $R_{4 A 4 B} \gslash^{A B} = R_{4 4}$, and thus we achieve \eqref{Ray-weigh-lim-0-1}.

\end{proof}

We continue with another interpretation for the non-negativity of $\Hb$ (Theorem \ref{lem-2nd-form}).
Mimicking the definition of Hawking mass, we also define an analogous mass $m$ for each sphere $S_{u, \ub}$:
\begin{equation}\label{def-Hawking-mass}
m \doteq \frac{r}{2} \left(1+\frac{1}{16\pi} \int_{S_{u,\ub}} \text{tr} \chi \text{tr} \chib \di \mu_{\gslash} \right).
\end{equation}
From Lemma \ref{lem-2nd-form-1}, it is easy to see that $\lim\limits_{r \rightarrow \infty} m =0$, i.e. the ``Bondi mass'' of each outgoing null hypersurface (close to $C_u$) vanishes.
Now look at $ r m$,
\begin{align*}
r m 
&= \frac{1}{2} \left(r^2 +\frac{1}{16 \pi} \int_{S_{u,\ub}} (h+2) (\hb -2) \di \mu_{\gslash} \right)  \\
&= \frac{1}{2} \left(r^2-\bar{r}^2 \right) + \frac{1}{16\pi} \int_{S_{u,\ub}} \left( \hb - h + \frac{h \hb}{2} \right)  \di \mu_{\gslash}\\
&= O(r^{-1}) + \frac{1}{16\pi} \int_{S_{u,\ub}} \hb \cdot r^2 \di \mu_{\gamma},
\end{align*}
where \eqref{compare-bar-r-r} and Lemma \ref{lem-2nd-form-1} are used. And here $h = r \tr \chi -2$.
Defining $M  \doteq \lim\limits_{r \rightarrow \infty} r m $, we have
\begin{equation*}
M = \frac{1}{16 \pi} \int_{S^2} \Hb \di \mu_{\gamma} = \frac{1}{32 \pi} \int_{S^2} (\Xi_1)^2 \di \mu_{\gamma} >0.
\end{equation*}
Equivalently, the ``Hawking mass'' of $S_{u, \ub}$ enjoys the following asymptotic behavior $$m= \frac{M}{r} + O(r^{-2}) \quad \text{with} \quad M>0.$$
In addition,
\begin{equation*}
\frac{\p M}{\p u} = - \frac{1}{16 \pi} \int_{S^2} \left( \lim\limits_{r \rightarrow \infty}  r^{3} R_{4  4} \right) \di \mu_{\gamma} = \frac{1}{32 \pi} \int_{S^2} \p_u (\Xi_1)^2 \di \mu_{\gamma}.
\end{equation*}

\subsubsection{An asymptotic form of $D_4e_3$}\label{sec-expanding-e3}
In this section, we will show how $D_4 e_3$ is related to $\Hb >0$.

Note that $$D_4 e_3 = -\frac{1}{2} g(D_4 e_3, e_4) e_3 + g(D_4 e_3, e_A) g^{AB} e_B.$$ As we have shown in Corollary \ref{lem-D-4-3-omega-null-infty}, $g(D_4 e_3, e_A) = O(r^{-3 - \frac{1}{2}})$.
We are now concerning $-\frac{1}{2} g(D_4 e_3, e_4)$.
Recall that in the proof leading to Corollary \ref{lem-D-4-3-omega-null-infty}, \eqref{eq-D4-3-4} is derived, namely $g(D_4  e_3, e_4)= e_4^{\mu} e_3^{ \nu}  \p_\nu \p_\mu \phi e_4 \phi+O(r^{-5})$. It follows that
\begin{align*}
g(D_4  e_3, e_4) 
&=D_3 D_4 \phi e_4 \phi  +  e_4^{\mu} e_3^{ \nu} \Gamma_{\mu \nu}^\sigma \p_\sigma \phi e_4 \phi +O(r^{-5})\\
&=D_3 D_4 \phi e_4 \phi + O(r^{-4 - \frac{1}{2}}),
\end{align*}
where \eqref{Gamma-3-4-p-phi} is used.
Furthermore, by the membrane equation, we have obtained $D_3D_4 \phi = -\frac{1}{r} e_4 \phi + O(r^{-2 -\frac{1}{2}})$ \eqref{eq-D3D4-DADB-phi-estimate}, which yields
\begin{equation}\label{density-positive}
-\frac{1}{2} g(D_4  e_3, e_4) =\frac{1}{2r} (e_4 \phi )^2 + O(r^{-3- \frac{1}{2}}).
\end{equation}
An alternative expression of \eqref{density-positive} is
\begin{equation}\label{density-negative-asym}
\Lambda \doteq \lim\limits_{r \rightarrow \infty} -\frac{r^3 }{2} g(D_4  e_3, e_4) = \frac{1}{2} (\Xi_1)^2 = \Hb.
\end{equation}
As a consequence,
\begin{equation}\label{asy-D4-3}
D_4  e_3 = \frac{\Lambda}{r^3} e_3 + O( r^{-3-\frac{1}{2}}) e_3 + O( r^{-3-\frac{1}{2}}) e_B  \quad \text{with} \quad \Lambda = \Hb.
\end{equation}
At the same time, as presented in Theorem \ref{thm-memory}, we conclude
\begin{equation}\label{eq-pu-density-Lambda}
\frac{\p}{\partial u} \Lambda = \frac{1}{2} \partial_u (\Xi_1)^2 = - \lim\limits_{r \rightarrow \infty} r^{3} R_{4  4}.
\end{equation}
And $ \lim\limits_{r \rightarrow \infty} -\frac{r^3}{2} g(D_4D_4 e_3, e_4) = \frac{\p  \Lambda}{\partial u}$.

\appendix

\section{Some calculations}\label{sec-calculations}
\subsection{Connection coefficients}\label{sec-connection-II}
We give some quantitative estimates for the Christoffel symbol $\Gamma_{\mu \nu}^\alpha$ of the metric $g_{\mu \nu}$.
\begin{proposition}\label{lem-connection}
In the short pulse region II,
\begin{align*}
\Gamma_{\ub u}^{u} & 
\sim \delta |\ub|^{2 d_n - \frac{3}{2}}, &
\Gamma_{\ub u}^{\ub} 
\sim |\ub|^{2d_n-1}, \\
 \Gamma_{u u}^{\ub} & 
 \sim \delta^{-1} |\ub|^{2d_n},
& \Gamma_{u u}^u 
\sim |\ub|^{2d_n -\frac{1}{2}},\\
\Gamma_{\ub \ub}^{u} &\sim  \delta^{2} |\ub|^{2d_n -2}, & \Gamma_{\ub \ub}^{\ub}  \sim \delta |\ub|^{2d_n-  \frac{3}{2}}, \\
e_A^\omega \Gamma_{u \omega}^{u} & \sim  \delta |\ub|^{2d_n-\frac{3}{2}},  & e_A^\omega \Gamma_{u \omega}^{\ub} \sim |\ub|^{2d_n-1}, \\
e_A^\omega \Gamma_{\ub \omega}^{u} & \sim  \delta^{\frac{7}{4}} |\ub|^{2d_n-2}, & e_A^\omega \Gamma_{\ub \omega}^{\ub}  \sim  \delta^{\frac{3}{4}} |\ub|^{2d_n-\frac{3}{2}}, \\
\Gamma_{\ub u}^\theta \p_\theta & \sim  \delta^{\frac{3}{4}} |\ub|^{2d_n-\frac{3}{2}} \nablaslash, & \Gamma_{\ub \ub}^\theta \p_\theta \sim \delta^{\frac{7}{4}} |\ub|^{2d_n-2} \nablaslash, \\
 \Gamma_{u u}^\theta \p_\theta & 
 \sim \delta^{-\frac{1}{4}} |\ub|^{2d_n-  \frac{1}{2}} \nablaslash, &
\end{align*}
and
\begin{align*}
e_A^\omega \Gamma_{\omega u}^\theta \p_\theta & \sim -\frac{1}{r}  e_A + \delta^{\frac{3}{4}} |\ub|^{2d_n-\frac{3}{2}} \nablaslash, \\
e_A^\omega \Gamma_{\omega \ub}^\theta \p_\theta & \sim  \frac{1}{r}  e_A + \delta^{\frac{3}{2}} |\ub|^{2d_n-2} \nablaslash, \\
 e_A^\omega e_B^\theta \Gamma_{\omega \theta}^u & \sim  \frac{1}{2r}  \eta_{AB}  + \delta |\ub|^{2d_n-\frac{3}{2}}, \\
 e_A^\omega e_B^\theta \Gamma_{\omega \theta}^{\ub} & \sim - \frac{1}{2r}  \eta_{AB} + |\ub|^{2d_n-1},\\
  e_A^\omega e_B^\theta \Gamma_{\omega \theta}^{\theta} \p_{\theta}  & \sim  \Gamma_{AB}^{\theta} (\eta) \p_{\theta}  + \delta^{\frac{3}{4}} |\ub|^{2d_n-\frac{3}{2}} \nablaslash.
\end{align*}
\end{proposition}
\begin{remark}\label{rk-connection}
Proposition \ref{lem-connection} also holds true (in the whole spacetime region) with $|\ub|$ being replaced by $r$, and the $\delta$-weight being absent, if we are only concerning with the decay rates in $r$.
\end{remark}
\begin{proof}
By \eqref{2.9}, there is in any coordinates,
\begin{equation}\label{christoffel-g-eta}
\begin{split}
\Gamma_{\mu \nu}^\alpha & = g^{\alpha \beta} \eta_{\beta \gamma} \Gamma_{\mu \nu}^\gamma (\eta) + g^{\alpha \beta} \p_\beta \phi   \p_\mu \p_\nu \phi \\
& = \Gamma_{\mu \nu}^\alpha (\eta) -\frac{1}{g} \p^\alpha \phi \p_\gamma \phi \Gamma_{\mu \nu}^\gamma (\eta) + g^{\alpha \beta}  \p_\beta \phi \p_\mu \p_\nu \phi,
\end{split}
\end{equation}
where $\Gamma_{\mu \nu}^\alpha (\eta)$ denotes the Christoffel symbol of the Minkowski metric $\eta_{\mu \nu}$.
Note that,
\begin{equation}\label{christoffel-eta}
\begin{split}
\Gamma_{u \ub}^\gamma (\eta) &=\Gamma_{u u}^\gamma (\eta) = \Gamma_{\ub \ub}^\gamma (\eta) =0,\\
\Gamma_{\omega u}^{\ub} (\eta) &= \Gamma_{\omega u}^u (\eta) = \Gamma_{\omega \ub}^{\ub} (\eta)= \Gamma_{\omega \ub}^u (\eta) =0,  \\
\Gamma_{\omega u}^\theta (\eta) &= -\frac{1}{r} \eta^\theta_\omega, \quad \Gamma_{\omega \ub}^\theta (\eta) = \frac{1}{r} \eta^\theta_\omega,\\
\Gamma_{\omega \theta}^u (\eta) &= \frac{1}{2r} \eta_{\omega \theta}, \quad \Gamma_{\omega \theta}^{\ub} (\eta) =- \frac{1}{2r} \eta_{\omega \theta}.
\end{split}
\end{equation}
We have $\Gamma_{\ub u}^\alpha = g^{\alpha \beta} \p_u \p_{\ub} \phi \p_\beta \phi$, for $\Gamma_{u \ub}^\gamma (\eta)=0$, then
\begin{align*}
\Gamma_{\ub u}^{u} & = g^{u \beta} \p_u \p_{\ub} \phi \p_\beta \phi = - \frac{1}{2} \p_u \p_{\ub} \phi \p_{\ub} \phi 
 \sim \delta |\ub|^{2 d_n - \frac{3}{2}}, \\
\Gamma_{\ub u}^{\ub} &  = g^{\ub \beta} \p_u \p_{\ub} \phi \p_\beta \phi = - \frac{1}{2} \p_{u} \p_{\ub} \phi \p_{u} \phi 
 \sim |\ub|^{2d_n-1},  \\
\Gamma_{\ub u}^\theta \p_\theta &= g^{\theta \beta} \p_u \p_{\ub} \phi \p_\beta \phi \p_\theta \sim \p_u \p_{\ub} \phi \p^\theta \phi \p_\theta  \sim  \delta^{\frac{3}{4}} |\ub|^{2d_n-\frac{3}{2}} \nablaslash.
\end{align*}
We have $\Gamma_{uu}^\alpha = g^{\alpha \beta} \p_u^2 \phi \p_\beta \phi$, for $\Gamma_{u u}^\gamma (\eta)=0$, and
\begin{align*}
 \Gamma_{u u}^{\ub} &= g^{\ub \beta} \p_u^2 \phi \p_\beta \phi = - \frac{1}{2} \p^2_u \phi  \p_u \phi 
 \sim \delta^{-1} |\ub|^{2d_n}, \\
\Gamma_{u u}^u &= g^{u \beta} \p_u^2 \phi \p_\beta \phi = - \frac{1}{2} \p_u^2 \phi \p_{\ub} \phi 
\sim |\ub|^{2d_n -\frac{1}{2}}, \\
\Gamma_{u u}^\theta \p_\theta &= g^{\theta \beta} \p_u^2 \phi \p_\beta \phi \p_\theta \sim \p_u^2 \phi \p^\theta \phi \p_\theta \sim \delta^{-\frac{1}{4}} |\ub|^{2d_n-  \frac{1}{2}} \nablaslash.
\end{align*}
We have $\Gamma_{\ub \ub}^\alpha = g^{\alpha \beta} \p_{\ub}^2 \phi \p_\beta \phi$, for $\Gamma_{\ub \ub}^\gamma (\eta)=0$,
\begin{align*}
 \Gamma_{\ub \ub}^{u} &=g^{u \beta} \p_{\ub}^2 \phi \p_\beta \phi \sim  - \frac{1}{2} \p^2_{\ub} \phi \p_{\ub} \phi \sim \delta^{2} |\ub|^{2d_n -2}, \\
\Gamma_{\ub \ub}^{\ub} &= g^{\ub \beta} \p_{\ub}^2 \phi \p_\beta \phi \sim - \frac{1}{2} \p_{\ub}^2 \phi \p_{u} \phi \sim \delta |\ub|^{2d_n-  \frac{3}{2}}, \\
\Gamma_{\ub \ub}^\theta \p_\theta &= g^{\theta \beta} \p_{\ub}^2 \phi \p_\beta \phi \p_\theta \sim \p_{\ub}^2 \phi \p^{\theta} \phi \p_\theta \sim \delta^{\frac{7}{4}} |\ub|^{2d_n-2} \nablaslash.
\end{align*}
And  $\Gamma_{\omega u}^\alpha = - \frac{1}{r} g^{\alpha \theta} \eta_{\theta \omega} + g^{\alpha \beta} \p_\omega  \p_u  \phi \p_\beta \phi$, due to $\Gamma_{\omega u}^{\ub} (\eta) = \Gamma_{\omega u}^u (\eta) =0$, $\Gamma_{\omega u}^\theta (\eta) = -\frac{1}{r} \eta^\theta_\omega$,
\begin{align*}
  e_A^\omega \Gamma_{\omega u}^{u} 
&  \sim   \frac{1}{r g} \p^u \phi e_A \phi  - \frac{1}{2} \p_{\ub} \phi e_A\p_{u} \phi \sim  \delta |\ub|^{2d_n-\frac{3}{2}}, \\
  e_A^\omega \Gamma_{\omega u}^{\ub} 
 & \sim   \frac{1}{r g} \p^{\ub} \phi e_A \phi  - \frac{1}{2} \p_{u} \phi e_A \p_{u} \phi \sim |\ub|^{2d_n-1}, \\
  e_A^\omega \Gamma_{\omega u}^\theta \p_\theta  
 & \sim  -\frac{1}{r}  e_A^\omega \p_\omega + \frac{1}{r g}  e_A\phi \p^{\theta} \phi \p_\theta +  e_A \p_u \phi \p^{\theta} \phi \p_\theta \\
 & \sim  -\frac{1}{r}  e_A + \delta^{\frac{3}{4}} |\ub|^{2d_n-\frac{3}{2}} \nablaslash.
\end{align*}
And $\Gamma_{\omega \ub}^\alpha =  \frac{1}{r} g^{\alpha \theta} \eta_{\theta \omega} + g^{\alpha \beta}  \p_\omega \p_{\ub} \phi \p_\beta \phi$, due to $\Gamma_{\omega \ub}^{\ub} (\eta) = \Gamma_{\omega \ub}^u (\eta) =0$, $\Gamma_{\omega \ub}^\theta (\eta) = \frac{1}{r} \eta^\theta_\omega$,
\begin{align*}
  e_A^\omega \Gamma_{\omega \ub}^{u} 
&  \sim  - \frac{1}{r g} \p^u \phi e_A \phi - \frac{1}{2} \p_{\ub} \phi e_A \p_{\ub} \phi \sim  \delta^{\frac{7}{4}} |\ub|^{2d_n-2}, \\
  e_A^\omega \Gamma_{\omega \ub}^{\ub} 
 & \sim  - \frac{1}{r g} \p^{\ub} \phi e_A \phi - \frac{1}{2} \p_{u} \phi e_A \p_{\ub} \phi  \sim \delta^{\frac{3}{4}}  |\ub|^{2d_n - \frac{3}{2}}, \\
  e_A^\omega \Gamma_{\omega \ub}^\theta \p_\theta 
 & \sim  \frac{1}{r}  e_A^\omega \p_\omega - \frac{1}{rg}  e_A \phi \p^{\theta} \phi \p_\theta +  e_A \p_{\ub}  \phi \p^{\theta} \phi \p_\theta \\
 & \sim \frac{1}{r}  e_A + \delta^{\frac{3}{2}} |\ub|^{2d_n-2} \nablaslash.
\end{align*}
Finally, using \eqref{christoffel-g-eta} and \eqref{christoffel-eta},
\begin{align*}
\Gamma_{\omega \theta}^\alpha 
&= g^{\alpha \beta} \eta_{\beta \gamma} \Gamma_{\omega \theta}^\gamma (\eta) + g^{\alpha \beta} \p_\beta \phi   \p_\omega \p_\theta \phi \\
&=  \frac{1}{r} g^{\alpha u} \eta_{\omega \theta} - \frac{1}{r} g^{\alpha \ub} \eta_{\omega \theta} + g^{\alpha \varphi} \Gamma_{\omega \theta}^\gamma (\eta) \eta_{\gamma \varphi} + g^{\alpha \beta} \p_\theta \p_\omega \phi \p_\beta \phi.
\end{align*}
Therefore,
\begin{align*}
e_A^\omega e_B^\theta \Gamma_{\omega \theta}^u &=  \frac{1}{r} g^{u u} \eta_{AB} - \frac{1}{r} g^{u \ub} \eta_{AB} + g^{u \beta} e_A^\omega e_B^\theta  \p_\theta \p_\omega \phi \p_\beta \phi + g^{u \omega }  \Gamma_{AB}^\gamma (\eta) \eta_{\gamma \omega} \\
&\sim \left( \frac{1}{2r}  + \frac{\p_u \phi \p_{\ub} \phi}{4rg} - \frac{(\p_{\ub} \phi)^2}{4rg} \right) \eta_{AB} -  \frac{1}{2} \nablaslash_A \nablaslash_B \phi  \p_{\ub} \phi \\
 & \sim  \frac{1}{2r}  \eta_{AB}  + \delta |\ub|^{2d_n-\frac{3}{2}},\\
e_A^\omega e_B^\theta \Gamma_{\omega \theta}^{\ub} &=  \frac{1}{r} g^{\ub u} \eta_{AB} - \frac{1}{r} g^{\ub \ub} \eta_{AB} + g^{\ub \beta} e_A^\omega e_B^\theta  \p_\theta \p_\omega \phi \p_\beta \phi + g^{\ub \omega}  \Gamma_{AB}^\gamma (\eta) \eta_{\gamma \omega} \\
&\sim \left(- \frac{1}{2r}  + \frac{(\p_u \phi)^2}{4rg} - \frac{\p_u \phi \p_{\ub} \phi}{4rg} \right) \eta_{AB} - \frac{1}{2} \nablaslash_A \nablaslash_B \phi  \p_{u} \phi \\
& 
\sim  - \frac{1}{2r}  \eta_{AB} + |\ub|^{2d_n-1},\\
e_A^{\varphi} e_B^\theta \Gamma_{\varphi \theta}^{\omega} \p_{\omega} &=  \frac{1}{r} g^{\omega u} \eta_{AB} \p_{\omega} - \frac{1}{r} g^{\omega \ub} \eta_{AB} \p_{\omega} + g^{\omega \beta} e_A^{\theta} e_B^\varphi \p_\varphi \p_{\theta} \phi \p_\beta \phi \p_{\omega} + g^{\omega \theta}  \Gamma_{AB}^\gamma (\eta) \eta_{\gamma \theta}\p_{\omega}  \\
&\sim \Gamma_{AB}^{\omega} (\eta) \p_{\omega} - \left( \frac{\p_u\phi }{2rg} - \frac{\p_{\ub} \phi }{2rg}  \right)  \eta_{AB} \p^\omega \phi \p_\omega + \nablaslash_A \nablaslash_B \phi  \p^\omega \phi \p_\omega \\
 &\sim \Gamma_{AB}^{\omega} (\eta) \p_{\omega}  + \delta^{\frac{3}{4}} |\ub|^{2d_n-\frac{3}{2}} \nablaslash.
\end{align*}
\end{proof}

As an application of Proposition \ref{lem-connection}, there is the following lemma.

\begin{lemma}\label{lemma-decom-curved-wave}
Let $\phi$ be the solution to the RME \eqref{1.1}. In the short pulse region II, we have, for any function $\varphi$,
\begin{equation}\label{box-estimate-general}
\begin{split}
\Box_{g(\p\phi)} \varphi& =  -\p_{\ub} \p_u \varphi +\nablaslash^\omega \nablaslash_\omega \varphi  + \frac{n-1}{2r} \p_{\ub} \varphi - \frac{n-1}{2r} \p_{u} \varphi \\
& \quad \pm |\ub|^{2d_n} \p^2_{\ub} \varphi \pm \delta^2 |\ub|^{2d_n-1} \p_u^2 \varphi  \pm \delta^{\frac{7}{4}} |\ub|^{2d_n - 1}  \nablaslash  \p_u \varphi  \\
& \quad \pm \delta^{\frac{3}{4}} |\ub|^{2d_n - \frac{1}{2}}  \nablaslash \p_{\ub} \varphi  \pm \delta^{\frac{3}{2}} |\ub|^{2d_n -1} \nablaslash^2 \varphi  \\
& \quad  \pm |\ub|^{2d_n - 1} \p_{\ub}  \varphi \pm \delta  |\ub|^{2d_n - \frac{3}{2}} \p_u \varphi \pm \delta^{\frac{3}{4}} |\ub|^{2d_n - \frac{3}{2}}  \nablaslash \varphi.
\end{split}
\end{equation}
\end{lemma}
\begin{proof}
We expand the Christoffel symbol
\begin{align*}
g^{\alpha \beta}D_\alpha D_\beta \varphi  & =  2g^{u\ub} \left(\p_{\ub} \p_u \varphi - \Gamma_{\ub u}^\sigma \p_\sigma \varphi \right) +g^{\ub \ub} \left(\p^2_{\ub} \varphi - \Gamma_{\ub \ub}^\sigma \p_\sigma \varphi \right)\\
& \quad + g^{uu} \left(\p^2_{u} \varphi - \Gamma_{u u}^\sigma \p_\sigma \varphi \right) + 2 g^{\ub \omega} \left(\p_{\ub} \p_\omega \varphi - \Gamma_{\ub \omega}^\sigma \p_\sigma \varphi \right)  \\
& \quad + 2 g^{u \omega} \left(\p_{u} \p_\omega \varphi - \Gamma_{u \omega}^\sigma \p_\sigma \varphi \right)  + g^{\theta \omega} \left(\p_{\theta} \p_\omega \varphi - \Gamma_{\theta \omega}^\sigma \p_\sigma \varphi \right).
\end{align*}
Thus, substituted the result of Proposition \ref{lem-connection}, it becomes
\begin{align*}
\Box_{g(\p \phi)} \varphi  & =  - \p_{\ub}  \p_u\varphi  \pm \delta |\ub|^{2d_n - \frac{1}{2}} \p_u \p_{\ub} \varphi \pm |\ub|^{2d_n} \p^2_{\ub} \varphi \pm \delta^2 |\ub|^{2d_n-1} \p_u^2 \varphi \\
& \quad \pm \delta^{\frac{7}{4}} |\ub|^{2d_n - 1}  \nablaslash  \p_u\varphi \pm \delta^{\frac{3}{4}} |\ub|^{2d_n - \frac{1}{2}}  \nablaslash \p_{\ub} \varphi + g^{\theta \omega} \left(\p_{\theta} \p_\omega \varphi - \Gamma_{\theta \omega}^\sigma \p_\sigma \varphi \right)\\
&\quad \pm |\ub|^{2d_n - 1} \p_{\ub}  \varphi \pm \delta  |\ub|^{2d_n - \frac{3}{2}} \p_u \varphi \pm \delta^{\frac{3}{4}} |\ub|^{2d_n - \frac{3}{2}}  \nablaslash \varphi.
\end{align*}
For the remaining part $ g^{\theta \omega} \left(\p_{\theta} \p_\omega \varphi - \Gamma_{\theta \omega}^\sigma \p_\sigma \varphi \right)$,
\begin{align*}
 g^{\theta \omega} \left(\p_{\theta} \p_\omega \varphi - \Gamma_{\theta \omega}^\sigma \p_\sigma \varphi \right)&= \nablaslash^\omega \nablaslash_\omega \varphi + \frac{n-1}{2r} \p_{\ub}\varphi - \frac{n-1}{2r} \p_{u}\varphi  \pm \delta^{\frac{3}{2}} |\ub|^{2d_n -1} \nablaslash^2 \varphi  \\
&\quad  \pm |\ub|^{2d_n-1} \p_{\ub}\varphi \pm \delta |\ub|^{2d_n-\frac{3}{2}} \p_{u}\varphi \pm \delta^{\frac{3}{4}} |\ub|^{2d_n-\frac{3}{2}} \nablaslash\varphi.
\end{align*}
Combining the above calculations, we prove \eqref{box-estimate-general}.
\end{proof}

In the following, we provide the proof of Corollary \ref{lem-D-4-3-omega-null-infty}.
\begin{proof}[Proof of Corollary \ref{lem-D-4-3-omega-null-infty}]
We begin with $D_{4} e_3 = e_4^\mu \p_\mu e_3^{\nu} \p_\nu + e_4^\mu  e_3^{ \nu} \Gamma_{\mu \nu}^\sigma \p_\sigma$.  Note that,
\begin{align*}
e_4^\mu  e_3^{ \nu} \Gamma_{\mu \nu}^\sigma \p_\sigma & = e_4^{\mu} e_3^{ \nu} \Gamma_{\mu \nu}^\sigma (\eta) \p_\sigma  - g^{-1} e_4^{\mu} e_3^{ \nu} \Gamma_{\mu \nu}^\sigma (\eta) \p_\sigma \phi  \p^\alpha \phi  \p_\alpha \\
&\quad +e_4^{\mu} e_3^{ \nu} \p_\mu \p_\nu \phi g^{\alpha \beta}\p_\beta \phi \p_\alpha,
\end{align*}
and by \eqref{christoffel-eta},
\begin{align*}
e_4^\mu  e_3^{ \nu} \Gamma_{\mu \nu}^\sigma (\eta) \p_\sigma \phi & = \left( e_4^\omega e_3^{\ub} \Gamma_{\omega \ub}^\theta (\eta) + e_3^\omega e_4^{u} \Gamma_{\omega u}^\theta (\eta) + e_4^\omega e_3^{u} \Gamma_{\omega u}^\theta (\eta)+ e_3^\omega e_4^{\ub} \Gamma_{\omega \ub}^\theta (\eta)  \right) \p_\theta \phi \\
&\quad + e_3^\omega e_4^\varphi \Gamma_{\omega \varphi}^\theta (\eta) \p_\theta \phi +  e_4^\omega e_3^{\theta}\Gamma_{\omega \theta}^u (\eta) \p_u \phi +  e_4^\omega e_3^{\theta}\Gamma_{\omega \theta}^{\ub} (\eta) \p_{\ub} \phi  \\
&= \left( \frac{ \p_u \phi \p^\theta \phi}{r g}  - \frac{ \p_{\ub} \phi \p^\theta \phi}{rg} + \p_u \phi \p_{\ub} \phi \p^\omega \phi \p^\varphi \phi \Gamma_{\omega \varphi}^\theta (\eta) \right)  \p_\theta \phi  \\
&\quad + \frac{\p_u \phi \p_{\ub} \phi}{2r g^2}  |\nablaslash \phi|^2 \p_u \phi - \frac{\p_u \phi \p_{\ub} \phi }{2r g^2} |\nablaslash \phi|^2 \p_{\ub} \phi + \text{l.o.t.} \\
& = O(r^{-5}).
\end{align*}
In addition, using the fact that $g(\p_\sigma, e_\nu) = \eta(\p_\sigma, e_\nu) + \p_\sigma \phi e_\nu \phi$ and \eqref{eta-coordinate-filed-null-frame}, we derive
\begin{align*}
 e_4^{\mu} e_3^{ \nu} \Gamma_{\mu \nu}^\sigma \p_\sigma \phi & = e_4^{\mu} e_3^{ \nu}  \p_\mu \p_\nu \phi \eta^{\alpha \beta} \p_\beta \phi \p_\alpha \phi + O(r^{-5}), \\ 
e_4^\mu  e_3^{ \nu} \Gamma_{\mu \nu}^\sigma g(\p_\sigma, e_4) &= -2 e_4^\mu  e_3^{ \nu} \Gamma_{\mu \nu}^{\ub} (\eta) + e_4^{\mu} e_3^{ \nu}  \p_\mu \p_\nu \phi e_4 \phi +O(r^{-6}), \\
e_4^\mu  e_3^{ \nu} \Gamma_{\mu \nu}^\sigma g(\p_\sigma, e_A) &=e_4^\mu  e_3^{ \nu} \Gamma_{\mu \nu}^{\theta} (\eta) \eta_{\theta A} + e_4^{\mu} e_3^{ \nu}  \p_\mu \p_\nu \phi e_A \phi +  O(r^{-6-\frac{1}{2}}), \\
e_4^\mu  e_3^{ \nu} \Gamma_{\mu \nu}^\sigma g(\p_\sigma, e_3) &= -2 e_4^\mu  e_3^{ \nu} \Gamma_{\mu \nu}^{u} (\eta) + e_4^{\mu} e_3^{ \nu}  \p_\mu \p_\nu \phi e_3 \phi +  O(r^{-6-\frac{1}{2}}).
\end{align*}
That is,
\begin{subequations}
\begin{align}
 e_4^{\mu} e_3^{ \nu} \Gamma_{\mu \nu}^\sigma \p_\sigma \phi & = O(r^{-4-\frac{1}{2}}), \label{Gamma-3-4-p-phi} \\
e_4^\mu  e_3^{ \nu} \Gamma_{\mu \nu}^\sigma g(\p_\sigma, e_4) &=e_4^{\mu} e_3^{ \nu}  \p_\mu \p_\nu \phi e_4 \phi + O(r^{-6}) = O(r^{-3}), \label{gamma-3-4-e4} \\
e_4^\mu  e_3^{ \nu} \Gamma_{\mu \nu}^\sigma g(\p_\sigma, e_A) &=e_4^{\mu} e_3^{ \nu}  \p_\mu \p_\nu \phi e_A \phi + \frac{\p_u \phi }{r} e_A\phi  +O(r^{-4}) = O(r^{-3- \frac{1}{2}}), \label{gamma-3-4-eA} \\
e_4^\mu  e_3^{ \nu} \Gamma_{\mu \nu}^\sigma g(\p_\sigma, e_3)
 &=e_4^{\mu} e_3^{ \nu}  \p_\mu \p_\nu \phi e_3 \phi  +O(r^{-6 - \frac{1}{2}}) = O(r^{-3 - \frac{1}{2}}).\label{gamma-3-4-e3}
\end{align}
\end{subequations}

Viewing the definition of $e_3$, and \eqref{gamma-3-4-e4}, we deduce
\begin{align*}
g(D_{4} e_3, e_4) &= e_4^{\mu} \p_{\mu}  e_3^{ \nu} g(\p_{\nu}, e_4) + e_4^\mu  e_3^{ \nu} \Gamma_{\mu \nu}^\sigma g(\p_\sigma, e_4) \\
&= - e_4 \frac{3 (\p_{\ub} \phi \p_u \phi)^2 }{8g} g(\p_{\ub}, e_4) + e_4  \frac{(\p_{\ub} \phi)^2}{4g} g(\p_{u}, e_4) - e_4 \frac{\p_{\ub} \phi \p^\theta \phi}{g} g(\p_\theta, e_4) \\
&\quad + e_4^{\mu} e_3^{ \nu}  \p_\mu \p_\nu \phi e_4 \phi + O(r^{-6})+  \text{l.o.t.},
\end{align*}
and hence
\begin{equation}\label{eq-D4-3-4}
g(D_4  e_3, e_4) = e_4^{\mu} e_3^{ \nu}  \p_\nu \p_\mu \phi e_4 \phi+O(r^{-5}) = O(r^{-3}).
\end{equation}
Using \eqref{gamma-3-4-eA}, we have
\begin{align*}
g(D_{4} e_3, e_A) &= e_4^{\mu} \p_{\mu}  e_3^{ \nu} g(\p_{\nu}, e_A) + e_4^\mu  e_3^{ \nu} \Gamma_{\mu \nu}^\sigma g(\p_\sigma, e_A) \\
&= - e_4 \frac{3 (\p_{\ub} \phi \p_u \phi)^2 }{8g} g_{{\ub}A} + e_4  \frac{(\p_{\ub} \phi)^2}{4g} g_{u A} - e_4 \frac{\p_{\ub} \phi \p^\theta \phi}{g} g_{\theta A} \\
& \quad + \frac{1}{r} \p_u \phi e_A \phi  + e_4^{\mu} e_3^{ \nu}  \p_\mu \p_\nu \phi e_A \phi +  O(r^{-4}) +  \text{l.o.t.}\\
&  = - e_4 \frac{\p_{\ub} \phi \p^\theta \phi}{g} g_{\theta A}+ e_4^{\mu} e_3^{ \nu}  \p_\mu \p_\nu \phi e_A \phi +  \frac{\p_u \phi}{r}  e_A\phi  +O(r^{-4}) \\
&=-\p_{\ub} \phi e_4 \p^\theta \phi g_{\theta A} +  \frac{\p_u \phi}{r}  e_A\phi  +O(r^{-4}).
\end{align*}
Therefore,
\begin{equation}\label{eq-D4-3-A}
g(D_4  e_3, e_A) =- e_A^\theta e_4^\mu \p_\mu \p_\theta \phi \p_{\ub} \phi  +  \frac{\p_u \phi }{r}  e_A\phi  +O(r^{-4})= O(r^{-3 - \frac{1}{2}}).
\end{equation}
Combining \eqref{eq-D4-3-4} with \eqref{eq-D4-3-A}, we prove \eqref {D-e4-e3}.

For $D_{3} e_4 = e_3^\mu \p_\mu e_4^{\nu} \p_\nu + e_4^\mu  e_3^{ \nu} \Gamma_{\mu \nu}^\sigma \p_\sigma$, by the definition of  $e_4$, and \eqref{gamma-3-4-e3}
\begin{align*}
g(D_{3} e_4, e_3) &= e_3^{\mu} \p_{\mu}  e_4^{ \nu} g(\p_{\nu}, e_3) + e_4^\mu  e_3^{ \nu} \Gamma_{\mu \nu}^\sigma g(\p_\sigma, e_3) \\
&= e_3  \frac{(\p_u \phi)^2}{4g} g(\p_{\ub}, e_3) + e_3  \frac{\p_{\ub} \phi \p_u \phi}{2g} g(\p_{u}, e_3) - e_3 \frac{\p_u \phi \p^\theta \phi}{g} g(\p_\theta, e_3) \\
&\quad + e_4^{\mu} e_3^{ \nu}  \p_\mu \p_\nu \phi e_3 \phi + O(r^{-6 - \frac{1}{2}})+  \text{l.o.t.} \\
&  = e_3  \frac{\p_{\ub} \phi \p_u \phi}{2g} g(\p_{u}, e_3)+ e_4^{\mu} e_3^{ \nu}  \p_\mu \p_\nu \phi e_3 \phi + O(r^{-6}),
\end{align*}
then there is
\begin{equation}\label{eq-D3-4-3}
g(D_3  e_4, e_3) =  - e_3 \p_{\ub} \phi \p_u \phi + O(r^{-6}) = O(r^{-3 - \frac{1}{2}}).
\end{equation}
Again, in view of \eqref{gamma-3-4-eA},
\begin{align*}
g(D_{3} e_4, e_A) &= e_3^{\mu} \p_{\mu}  e_4^{ \nu} g(\p_{\nu}, e_A) + e_4^\mu  e_3^{ \nu} \Gamma_{\mu \nu}^\sigma g(\p_\sigma, e_A) \\
&= e_3  \frac{(\p_u \phi)^2}{4g} g(\p_{\ub}, e_A) + e_3  \frac{\p_{\ub} \phi \p_u \phi}{2g} g(\p_{u}, e_A) - e_3 \frac{\p_u \phi \p^\theta \phi}{g} g(\p_\theta, e_A) \\
& \quad + \frac{1}{r} \p_u \phi e_A \phi   + e_4^{\mu} e_3^{ \nu}  \p_\mu \p_\nu \phi e_A \phi +  O(r^{-4}) +  \text{l.o.t.} \\
&  =- e_3 \frac{\p_u \phi \p^\theta \phi}{g} g_{\theta A}+ e_4^{\mu} e_3^{ \nu}  \p_\mu \p_\nu \phi e_A \phi +   \frac{\p_u \phi}{r} e_A\phi +O(r^{-4})\\
&= - \p_u \phi e_3\p^\theta \phi g_{\theta A} +  \frac{\p_u \phi}{r} e_A\phi +O(r^{-4}). 
\end{align*}
Namely, we get
\begin{equation}\label{eq-D3-4-A}
g(D_3  e_4, e_A) = - e_A^\theta e_3^\nu \p_\nu \p_\theta \phi  \p_u \phi +  \frac{3\p_u \phi}{r} e_A\phi +O(r^{-4}) = O(r^{-3 - \frac{1}{2}}).
\end{equation}
As a result, \eqref{D-e3-e4} follows by considering \eqref{eq-D3-4-3} and \eqref{eq-D3-4-A}.

We next turn to $D_{A} e_3 = e_A^\omega \p_\omega e_3^\nu \p_\nu +  e_A^\omega e_3^\nu \Gamma_{\omega \nu}^\sigma \p_\sigma$. Similarly, viewing \eqref{christoffel-g-eta} and \eqref{christoffel-eta}, we have
\begin{align*}
e_3^{\nu} e_A^{ \omega} \Gamma_{\nu \omega}^\sigma \p_\sigma & = e_3^{\nu} e_A^{ \omega} \Gamma_{\nu \omega}^\sigma (\eta) \p_\sigma - g^{-1}  e_3^{\nu} e_A^{ \omega} \Gamma_{\nu \omega}^\sigma (\eta) \p_\sigma \phi \p^\alpha \phi \p_\alpha\\
& \quad + e_3^{\nu} e_A^{ \omega} \p_\nu \p_\omega \phi g^{\alpha \beta} \p_\beta \phi \p_\alpha,
\end{align*}
and
\begin{align*}
e_3^{\nu} e_A^{ \omega} \Gamma_{\nu \omega}^\sigma (\eta) \p_\sigma \phi & = e_A^\omega \left(  e_3^{\ub} \Gamma_{\omega \ub}^\theta (\eta)  + e_3^{u} \Gamma_{\omega u}^\theta (\eta) \right) \p_\theta \phi +  e_A^\omega e_3^{\theta} \left( \Gamma_{\omega \theta}^u (\eta) \p_u \phi +  \Gamma_{\omega \theta}^{\ub} (\eta) \p_{\ub} \phi \right) \\
&= \left( \frac{1}{r} - \frac{(\p_{\ub} \phi)^2}{4rg}  \right)  e_A \phi - \frac{\p_{\ub} \phi e_A \phi }{2rg} \p_u \phi + \frac{ \p_{\ub} \phi e_A \phi}{2r} \p_{\ub} \phi\\
 & 
 =  \frac{ e_A \phi}{r} + O(r^{-5}) =  O(r^{-2-\frac{1}{2}}).
\end{align*}
Besides, by the fact that $g(\p_\sigma, e_\nu) = \eta(\p_\sigma, e_\nu) + \p_\sigma \phi e_\nu \phi$ and using \eqref{eta-coordinate-filed-null-frame}, we obtain
\begin{align*}
e_3^{\nu} e_A^{ \omega} \Gamma_{\nu \omega}^\sigma g(\p_\sigma, e_4) &= - \frac{ e_A \phi \p_u \phi}{r} + \frac{ \p_{\ub} \phi e_A \phi}{2r} (-2+ O(r^{-2-\frac{1}{2}})) +  O(r^{-6}) \\
&\quad + e_3^{\nu} e_A^{ \omega} \Gamma_{\nu \omega}^\sigma (\eta) \p_\sigma \phi e_4 \phi - g^{-1}  e_3^{\nu} e_A^{ \omega} \Gamma_{\nu \omega}^\sigma (\eta) \p_\sigma \phi e_4 \phi \\
& \quad - g^{-1}  e_3^{\nu} e_A^{ \omega} \Gamma_{\nu \omega}^\sigma (\eta) \p_\sigma \phi \p^\alpha \phi \p_\alpha \phi e_4 \phi +e_3^{\nu} e_A^{ \omega}  \p_\nu \p_\omega \phi e_4 \phi, 
\end{align*}
i.e.,
\begin{equation}\label{gamma-A-3-e4}
e_3^{\nu} e_A^{ \omega} \Gamma_{\nu \omega}^\sigma g(\p_\sigma, e_4) = - \frac{ e_A \phi \p_u \phi}{r}+e_3^{\nu} e_A^{ \omega}  \p_\nu \p_\omega \phi e_4 \phi +  O(r^{-4}).
\end{equation}

Substituting the formula of $e_3$ and making use of \eqref{gamma-A-3-e4}, we have
\begin{align*}
g(D_{A} e_3, e_4) &= e_A^\omega \p_\omega e_3^\nu g(\p_\nu, e_4) +  e_A^\omega e_3^\nu \Gamma_{\omega \nu}^\sigma g(\p_\sigma, e_4) \\
&= - e_A \frac{3 (\p_{\ub} \phi \p_u \phi)^2 }{8g} g_{{\ub} 4} + e_A  \frac{(\p_{\ub} \phi)^2}{4g} g_{u 4} - e_A \frac{\p_{\ub} \phi \p^\theta \phi}{g} g_{\theta 4} \\
& \quad - \frac{ e_A \phi \p_u \phi}{r}+e_3^{\nu} e_A^{ \omega}  \p_\nu \p_\omega \phi e_4 \phi +  O(r^{-4}).
\end{align*}
Consequently,
\begin{equation}\label{eq-DA-3-4}
g(D_A  e_3, e_4) = - \frac{ e_A \phi \p_u \phi}{r}+e_3^{\nu} e_A^{ \omega}  \p_\nu \p_\omega \phi e_4 \phi +  O(r^{-4})=O(r^{-3-\frac{1}{2}}),
\end{equation}
which further entails
\begin{equation}\label{eq-DA-4-3}
g(D_A  e_4, e_3) = -g(e_4, D_A e_3) =O(r^{-3-\frac{1}{2}}).
\end{equation}

Putting \eqref{eq-DA-4-B} and \eqref{eq-DA-4-3} together, we complete the proof for \eqref{D-A-4}.
And \eqref{D-A-3} follows by combining \eqref{eq-DA-3-B} with \eqref{eq-DA-3-4}.

Finally, the proof for \eqref{D-4-4-null-frame} is presented in the one following Corollary \ref{lem-D-4-3-omega-null-infty}.
\end{proof}

\subsection{Raychaudhuri equation}\label{sec-Raychaudhuri equation}
We shall define the projection onto the sphere $S_{u, \ub}$:
\begin{equation}\label{def-Pi-proje}
\Pi^{\alpha \beta} = g^{\alpha \beta} + \frac{1}{2}e_3^\alpha e_4^\beta + \frac{1}{2}e_4^\alpha e_3^\beta,
\end{equation}
and
\begin{equation*}
\Pi^{\alpha}_{\beta} = \Pi^{\alpha \mu} g_{\mu \beta} , \quad \Pi_{\alpha \beta} = \Pi^{\mu \nu} g_{\alpha \mu} g_{\beta \nu}.
\end{equation*}
We know that $\Pi_{\alpha \beta} $ lives on $S_{u, \ub}$, namely,  $\Pi_{\alpha \beta} e_3^\alpha =0$, and  $\Pi_{\alpha \beta} e_4^\alpha =0$.
And restricted on $S_{u, \ub}$, $\Pi_{AB} = \gslash_{AB}$ is the induced metric on $S_{u, \ub}$, and $\Pi^{AB} = \gslash^{AB}$. Hence, the $\tr \chi$ and $\tr \underline{\chi}$ can be equivalently rewritten as
\begin{align*}
\text{tr} \underline{\chi} = \Pi^{\alpha \beta} D_{\alpha} e_{4\beta}, \quad \text{tr} \chi = \Pi^{\alpha \beta} D_{\alpha} e_{3\beta}.
\end{align*}
\begin{lemma}\label{lem-Raychaudhuri}
We have the Raychaudhuri equation \cite{Hawking-Ellis} with non-affine parametrisations (along $e_4$),
\begin{equation}\label{Ray-4}
\begin{split}
 D_{4} \tr \chib &= - \frac{1}{2}\tr^{2} \chib - \gslash^{AB} \gslash^{CD} \hat \chib_{C B} \hat \chib_{D A}   - R_{4A 4 B} \gslash^{AB} \\
& \quad +\gslash^{AB}  \left( D_A D_{4} e_{4B} +  g \left( D_A e_4, e_3 \right) \cdot D_{4} e_{4B}  \right) \\
&\quad  + \frac{1}{2 } \gslash^{AB}  \left(  D_{3} e_{4A} + D_{4} e_{3A}  \right) D_{4}  e_{4B}.
\end{split}
\end{equation}
\end{lemma}

\begin{proof}
By definition of the Ricci tensor,
\begin{align*}
e_4^\mu D_\mu D_\alpha e_{4 \beta} &= e_4^\mu D_\alpha D_\mu e_{4 \beta} +  e_4^\mu  R_{\mu \alpha \beta}{}^{\! \nu} e_{4\nu}  \\
&=D_\alpha \left( e_4^\mu D_{\mu} e_{4\beta} \right) - D_\alpha e_4^\mu \cdot D_\mu e_{4\beta} + R_{4 \alpha \beta 4}.
\end{align*}
Then multiplying $\Pi^{\alpha \beta}$, we have
\begin{equation}\label{Ray-2}
\begin{split}
 &\quad D_{4} \left( \Pi^{\alpha \beta}  D_\alpha e_{4 \beta} \right) -  D_{4} \Pi^{\alpha \beta} \cdot D_\alpha e_{4 \beta} \\
&= \Pi^{\alpha \beta}  \left( D_\alpha \left(  D_{4} e_{4}\right)_{\beta}  - g \left( D_\alpha e_4, e_3 \right) \cdot D_{4} e_{4\beta}  \cdot g^{34}  \right)   \\
&\quad - \Pi^{\alpha \beta}  \Pi^{\mu \nu}  D_\alpha e_{4\nu}  D_\mu e_{4\beta} - R_{4 \alpha 4 \beta} \Pi^{\alpha \beta}.
\end{split}
\end{equation}
We note that, projected onto $S_{u, \ub}$, $\Pi^{\alpha}_{A}  \Pi_{B}^{\nu}  D_\alpha e_{4\nu} =  \chib_{AB} + \sigma_{AB} = \hat\chib_{AB} +  \frac{1}{2} \tr \chib \gslash_{AB} + \sigma_{AB}$, where $\chib_{AB}$ is the symmetric part, and the anti-symmetric part $\sigma_{AB}$ is the rotation given by $\sigma_{AB} =\frac{1}{2} \left( D_{A} e_{4B} - D_{B} e_{4A}\right)$. Since $\{e_A, \, A=1,2 \}$ is a basis of $TS_{u,\ub}$,  there is  $\sigma_{AB} = 0$, by the Frobenius theorem.
Moreover, by virtue of the definition \eqref{def-Pi-proje}, \eqref{Ray-2} implies that
\begin{equation}\label{Ray-3}
\begin{split}
 D_{4} \tr \chib &= \Pi^{\alpha}_{\beta}  D_\alpha \left(  D_{4} e_{4}  \right)^{\beta}+ \frac{1}{2} g \left( D_\alpha e_4, e_3 \right) \cdot D_{4} e_{4\beta} \Pi^{\alpha \beta}  \\
&\quad + \frac{1}{2 } \left(  D_{3}e_{4\mu} + D_{4}e_{3\mu}  \right) D_{4} e_4^\mu +  \frac{1}{2 } g\left(  D_{\mu}e_{4}, e_3  \right) D_{4} e_4^\mu \\
&\quad - \frac{1}{2} \tr^2 \chib -  \gslash^{AB} \gslash^{CD} \hat \chib_{C B} \hat \chib_{D A} - R_{4 \alpha 4 \beta} \Pi^{\alpha \beta}.
\end{split}
\end{equation}
The second and the fourth terms on the right hand side of the equality \eqref{Ray-3} turn into
\begin{align*}
 & \quad \frac{1}{2} g \left( D_\alpha e_4, e_3 \right) \cdot D_{4} e_{4\beta} \Pi^{\alpha \beta} +  \frac{1}{2 } g\left(  D_{\mu}e_{4}, e_3  \right) D_{4} e_4^\mu \\
 & = \frac{1}{2 } g^{\nu \alpha}g \left(  D_{\alpha}e_{4}, e_3  \right) D_{4} e_4^\mu \left( g_{\mu \nu + \Pi_{\mu \nu} } \right) \\
 & = D_4e_4^\mu \cdot g \left(  D_{\mu}e_{4}, e_3  \right) + \frac{1}{4}  \left( g \left(  D_{4}e_{4}, e_3  \right)\right)^2.
\end{align*}
Now combine these with the third term on the right hand side of \eqref{Ray-3},
\begin{align*}
 &\quad D_4e_4^\mu \cdot g \left(  D_{\mu}e_{4}, e_3  \right) + \frac{1}{4}  \left( g \left(  D_{4}e_{4}, e_3  \right) \right)^2 + \frac{1}{2 } \left(  D_{3}e_{4\mu} + D_{4}e_{3\mu}  \right) D_{4} e_4^\mu \\
 & =\Pi^{\alpha \beta} D_{4} e_{4 \beta} \cdot g \left( D_\alpha e_4, e_3 \right) + \frac{1}{2 } \Pi^{\alpha \beta}  \left(  D_{3}e_{4 \alpha} + D_{4}e_{3 \alpha}  \right) D_{4}  e_{4 \beta},
\end{align*}
where we have used the fact that the $e_3, e_4$ components vanish, i.e.
\begin{align*}
 &\quad g^{34} \left( g \left(  D_{4}e_{4}, e_3  \right) \right)^2 + \frac{1}{4}  \left( g \left(  D_{4}e_{4}, e_3  \right) \right)^2 + \frac{1}{2 }g^{34} g \left( D_{e_4}e_{3}, e_4  \right) g \left(  D_{4}e_{4}, e_3  \right)  \\
 & =- \frac{1}{4}  \left( g \left(  D_{4}e_{4}, e_3  \right) \right)^2 - \frac{1}{2 }g^{34} \left( g \left(  D_{4}e_{4}, e_3  \right) \right)^2 =0.
\end{align*}
Hence, The 2nd-4th terms on the right hand side of \eqref{Ray-3} are reduced to
\begin{align*}
&\quad \Pi^{\alpha \beta} D_{4} e_{4 \beta} \cdot g \left( D_\alpha e_4, e_3 \right) + \frac{1}{2 } \Pi^{\alpha \beta}  \left(  D_{3}e_{4 \alpha} + D_{4} e_{3 \alpha}  \right) D_{4}  e_{4 \beta}. 
\end{align*}
Putting all these together, we complete the proof for \eqref{Ray-4}.
\end{proof}

 \section{Existence of initial data}\label{existence of data}
 In the section, we will prove the existence of data satisfying \eqref{1.2}-\eqref{1.3-associated}. Due to the quasilinear property of the RME, the strategy in Section 2.5 of Miao-Pei-Yu \cite{Miao-Yu} does not apply to our case. In our case, a natural way to construct data on $\{t=1\} \cap ( B_{1}-B_{1-2\delta} ) $ supported in $r \in (1-2\delta, 1)$ and satisfying \eqref{1.2}-\eqref{1.3-associated} is to find a solution to the RME, which exists in a neighbourhood of $\{t=1\} \cap ( B_{1}-B_{1-2\delta} )$, takes the form of $\phi(\frac{u}{\delta},\underline{u},\theta)$ and vanishes on $\{r= 1-2\delta \} \cup \{ r=1\}$. Then we can restrict this solution to the $\{t=1\}$ hypersurface, so that it qualifies for the data satisfying \eqref{1.2}-\eqref{1.3-associated}. This, as shown below, indeed reduces to a small data, finite time existence of smooth solution to the RME, which is of course well-known.

Let us introduce the new coordinates first,
 \begin{equation}\label{n-coordinate}
 \rho^{\prime\mu^{\prime}}=(u^{\prime},\underline{u}^{\prime},\theta^{\prime}) \doteq \left(\frac{u}{\delta},\underline{u},\theta\right)\quad \text{and}\quad
 \rho^{\mu}=(u,\underline{u},\theta) = (\delta u^{\prime},\underline{u}^{\prime},\theta^{\prime}).
 \end{equation}
 The new $t^\prime,\,\, r^\prime$ are defined by
  \begin{equation}\label{t-prime}
t^\prime \doteq \underline{u}^{\prime} + u^{\prime}, \quad r^\prime \doteq \underline{u}^{\prime} - u^{\prime}.
 \end{equation}
 We denote $C^\prime_{u^\prime}$ the hypersurface with constant $u^\prime$, $\underline{C}^\prime_{\ub^\prime}$ the hypersurface with constant $\ub^\prime$, and $S^\prime_{u^\prime, \ub^\prime}$ the intersection of $C^\prime_{u^\prime}$ and $\underline{C}^\prime_{\ub^\prime}$. We remark that, $\underline{C}^\prime_{1-\delta}=\underline{C}_{1-\delta}$ and $C^\prime_0=C_0$. In this section, we will always use the Greek letter with prime $\alpha^\prime, \, \beta^\prime, \, \mu^\prime, \, \nu^\prime, \cdots$ to denote the  new variables $(u^{\prime},\underline{u}^{\prime},\theta^{\prime})$ or $(t^{\prime}, r^{\prime},\theta^{\prime})$.
 Then under the above coordinate transformation \eqref{n-coordinate},
\begin{equation*}
g^{ \alpha^{\prime}\beta^{\prime}}=g^{\alpha\beta}\frac{\partial\rho^{\prime\alpha^{\prime}}}{\partial \rho^{\alpha}}\frac{\partial\rho^{\prime\beta^{\prime}}}{\partial \rho^{\beta}}
\quad \text{and} \quad g_{\alpha^{\prime}\beta^{\prime}}=g_{\alpha\beta}\frac{\partial\rho^{\alpha}}{\partial\rho^{\prime\alpha^{\prime}}}
\frac{\partial\rho^{\beta}}{\partial\rho^{\prime\beta^{\prime}}}.
\end{equation*}

\begin{figure}
\centering
\subfigure[Region in ($u, \ub$) coordinate]{
 \label{fig:subfig:a}
 \includegraphics[width=2.9in]{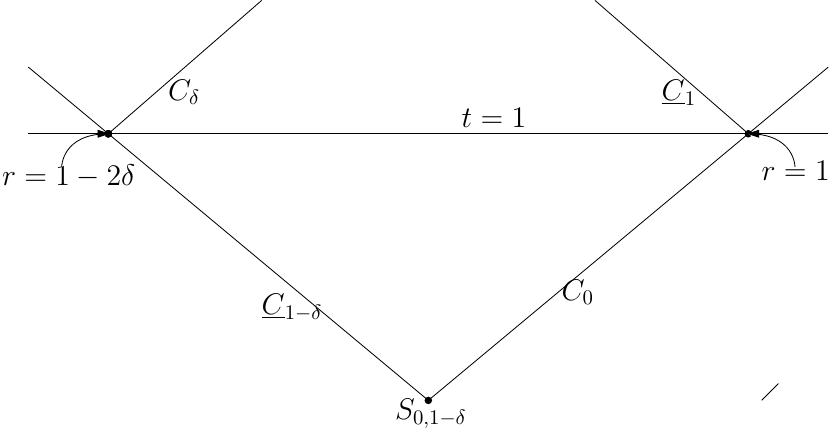}
 }
 \subfigure[Region in $(u^\prime, \ub^\prime)$ coordinate]{
 \label{fig:subfig:b}
  \includegraphics[width=2.8in]{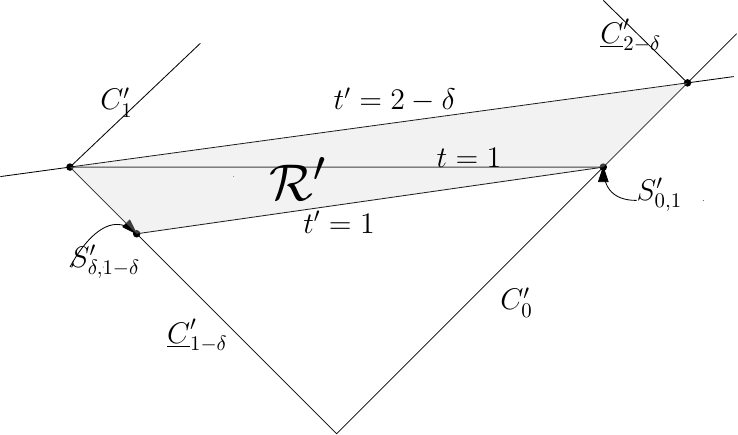}
  }
  \caption{Transformation between $(u, \ub)$ and $(u^\prime, \ub^\prime)$} \label{fig:subfig}
\end{figure}
We expect that the solution to the RME with constraint \eqref{1.2}-\eqref{1.3-associated} on $\{t=1\}$ slice
exists in a neighborhood of $\{t=1\}$ (see Figure \ref{fig:subfig:a}),
which means that we need to prove the local existence for solutions to the RME with large data. After changing to the $(u^\prime, \ub^\prime, \theta^\prime)$ coordinate system, it would be transformed into a small data existence (for solutions to the RME) in the following finite region (see Figure \ref{fig:subfig:b})
\begin{equation}\label{data-region-R-prime}
\mathcal{R}^\prime = \{(u^\prime, \ub^\prime, \theta^\prime) | 1 \leq t^\prime \leq 2-\delta,\,\, 0\leq u^\prime \leq 1,\,\, 1-\delta \leq \ub^\prime \leq 2-\delta\}.
\end{equation}

Actually, we can show that $\{t=1\} \subset \mathcal{R}^\prime$. Note that the bottom of $\Rp$ is $\{t^\prime =1\}$, i.e.
\begin{equation}\label{data-restirct-t-prime}
t^\prime=\ub^\prime +u^\prime =\delta^{-1} u+\underline{u}
= \frac{1+\delta^{-1}}{2} t+ \frac{1-\delta^{-1}}{2} r =1.
\end{equation}
On the other hand, $\ub^\prime \leq 1$ on $\{t^\prime=1\} \cap \Rp$, hence
\begin{equation}\label{data-t=1--r}
r\leq 2-t, \quad \text{on}\quad \{t^\prime=1\} \cap \Rp.
\end{equation}
\eqref{data-restirct-t-prime} and \eqref{data-t=1--r} lead to $$t\leq 1, \quad \text{on} \quad \{t^\prime =1\}  \cap \Rp.$$ In the same way, we have $u^\prime \leq 1$ on $\{t^\prime=2-\delta\} \cap \Rp$, and hence $$t \geq 1, \quad \text{on} \quad \{t^\prime =2-\delta\}  \cap \Rp.$$

We prescribe the following Cauchy data on $\{t^\prime=1\} \cap \Rp$ to the RME
\begin{equation}\label{initial}
\phi(u^{\prime},\underline{u}^{\prime},\theta^{\prime})|_{t^\prime =1}=\delta  f_{0}(r^\prime,\theta^{\prime}),\quad
\p_{t^\prime}\phi(u^{\prime},\underline{u}^{\prime},\theta^{\prime})|_{t^\prime =1}=\delta  f_{1}(r^\prime,\theta^{\prime}),
\end{equation}
where $f_{0}(r^\prime,\theta^{\prime}), \, f_{1}(r^\prime,\theta^{\prime})$ are smooth functions supported in $r^\prime \in (1-2\delta, 1)$ and $\{t^\prime =1\}$ is a spacelike hypersurface with respect to $g_{\alpha^\prime \beta^\prime}$, since
 \begin{equation*}
g_{t^\prime t^\prime}=-\delta + (\p_{t^\prime} \phi)^2, \quad g^{t^\prime t^\prime}=-\delta^{-1} - g^{-1} ( \delta^{-1} \p_{t^\prime} \phi)^2.
\end{equation*}
Moreover, the Cauchy data \eqref{initial} ensure that
\begin{equation}\label{data-initial-higher}
|\p_{u^\prime}^{i_0}\p^{i_1}_{\ub^{\prime}} \p^{i_2}_{\theta^{\prime}} \phi|_{t^\prime =1} \lesssim \delta, \quad  \text{for all} \,\, i_0,\, i_1,\, i_2 \in \mathbb{N}.
\end{equation}

In addition, the boundaries $\Cb^\prime_{1-\delta}$$\cap \{u^\prime \geq\delta\}$,
$C^\prime_0$$\cap \{\ub^\prime \geq 1\}$
are both (truncated) null cones with respect to $g_{\alpha^\prime \beta^\prime}$, and $\phi$ vanishes on these two null cones up to any finite order of derivatives. This is illustrated as follows.
We have chosen $f_{0}, \, f_{1}$ to be smooth functions supported in $r^\prime \in (1-2\delta, 1)$, so that
\begin{equation*}
\partial^{\prime k}\phi(u^{\prime},\underline{u}^{\prime},\theta^{\prime})|_{S^{\prime}_{0,1}\cup S^{\prime}_{\delta,1-\delta}}=0,\quad\text{for}\;\text{any}\;k\in\mathbb N,
\end{equation*}
where $\partial^{\prime} \in \{\partial_{t^{\prime}},\partial_{r^{\prime}}, \partial_{\theta^{\prime}}\}$.
This further entails that
 \begin{equation}\label{data-gradu-grad-ub-pu-pub}
 -2D \underline{u}^{\prime}=\partial_{u^{\prime}},\quad -2D u^{\prime}=\partial_{\underline{u}^{\prime}},\quad \text{on}\; S^{\prime}_{0,1}\cup S^{\prime}_{\delta,1-\delta}.
 \end{equation}
 Here, $D u^\prime, \, D \ub^\prime$ mean taking gradient in terms of the metric $g_{\alpha^{\prime}\beta^{\prime}}$.  Then $D u^\prime, \, D \ub^\prime$ are null vector fields on $S^\prime_{0, 1} \cup S^\prime_{\delta, 1-\delta}$, since $g_{u^\prime u^\prime}=g_{\ub^\prime \ub^\prime}=0$ on $S^\prime_{0, 1} \cup S^\prime_{\delta, 1-\delta}$.
Due to the property of finite speed propagation for wave equations, we have
\begin{equation}\label{p-phi-0-initial}
\partial^{\prime k}\phi=0,\quad \text{on}\quad \underline{C}_{1-\delta}^{\prime}\cap\{u^{\prime}\geq \delta\}\;\text{and}\;
C^{\prime}_{0}\cap \{\underline{u}^{\prime}\geq1\}, \,\,\, \forall k\in\mathbb N,
\end{equation}
and hence $D \underline{u}^{\prime}$, $D u^{\prime}$ are null on $\underline{C}_{1-\delta}^{\prime}\cap\{u^{\prime}\geq \delta\}$ and $C_{0}^{\prime}\cap\{\underline{u}^{\prime}\geq 1\}$. That is to say, these two hypersurfaces are null cones.

By the finite time existence of solutions to quasilinear wave equations with small data, we obtain that the solution to the Cauchy problem of RME \eqref{2.8}, \eqref{initial} exists in the region $\mathcal{R}^\prime$. Then the restriction of $\phi(u^{\prime},\underline{u}^{\prime},\theta^{\prime})= \phi(\frac{u}{\delta}, \ub, \theta)$ to $\{t=1\} \cap ( B_{1}-B_{1-2\delta} ) $,
\begin{equation}\label{initial-restiction-t=1}
\phi\left(\frac{u}{\delta}, \ub, \theta\right)\Big|_{\{t=1\} \cap ( B_{1}-B_{1-2\delta} ) }
\end{equation}
will automatically provide the data supported in $ r \in (1-2\delta, 1)$ and satisfying \eqref{1.2}-\eqref{1.3-associated}.

We here sketch the proof for the existence of solution to the RME in $\mathcal{R}^\prime$ with the small data \eqref{initial}. We will take $\{\partial_{u^{\prime}},\,\partial_{\underline{u}^{\prime}},\Omega^\prime=\Omega\}$ as commutators and $\xi=\partial_{t^\prime}=\frac{1}{2}(\partial_{u^{\prime}}+\partial_{\underline{u}^{\prime}})$ as the multiplier.
 Let $\phi$ be the unknown solution to the RME, and define the $k$-th order derivative
\begin{equation}\label{phi-k-data}
\phi_{k} \doteq  \partial_{u^{\prime}}^{k_{1}}\partial_{\underline{u}^{\prime}}^{k_{2}}\Omega^{\prime k_{3}}\phi, \quad k_1+k_2+k_3 =k.
\end{equation}

\begin{lemma}\label{lemma-high-order-RRME}
We have the equation for $\phi_k$
\begin{equation}\label{RRME-H}
\begin{split}
&\quad \Box_{g(\p \phi)} \phi_k
+f(r)\sum_{q<k}\delta^{-1} \partial_{u^{\prime}}\phi_{q} \pm \partial_{\underline{u}^{\prime}}\phi_{q} \pm \laplacianslash\phi_{q}+S_k+N_{k}=0,
\end{split}
\end{equation}
where $N_{k}$ denotes terms (of lower order derivatives) taking the double null forms, and
\begin{align*}
S_k & \sim \sum_{ k_1+k_2+ k_3 \leq k-1} (\p_{\ub^\prime} - \delta^{-1} \p_{u^\prime}) \phi_{k_1} \nablaslash \phi_{k_2} \nablaslash \phi_{k_3} \\
 & \quad +\sum_{ k_1+k_2+ k_3 \leq k-2} \nablaslash \phi_{k_1}  \nablaslash \phi_{k_2} \nablaslash \phi_{k_3},
 \end{align*}
involving at least two angular derivatives  and $f(r)$ is a bounded and smooth function depending on $r$.
\end{lemma}

The energy associated to the multiplier $\p_{t^\prime}$ is given by
\begin{align*}
E^2[\phi_k](t^{\prime}) &\sim  \int_{\Sigma_{t^{\prime}}} \left( \delta^{-1} |\p_{u^\prime} \phi_k|^2 + \delta^{-1} |\p_{\ub^\prime} \phi_k|^2 + | \nablaslash \phi_k|^2 \right)\sqrt{g^{\prime}} dr^{\prime} d\sigma_{S^{n-1}}.
\end{align*}
Here $g^\prime = | \det (g_{\alpha^\prime \beta^\prime}) |$ and we note that $g^\prime \sim \delta^2$.
 Define $E^2[\phi_{\leq k}] \doteq \sum_{0\leq q\leq k}E^2[\phi_{q}]$.
 We start with the bootstrap assumption: fix $N \in \mathbb N$ and $N \geq 4$,
\begin{equation}\label{bootstrap-assumption-data}
E[\phi_{\leq N}](t^{\prime})\leq \delta M, \quad \text{for} \,\, 1 \leq t^{\prime} \leq 2-\delta,
\end{equation}
 where $M$ is a large constant, which may depend on $\phi$, to be determined.
 With \eqref{bootstrap-assumption-data}, we have the $L^\infty$ estimate, by the Sobolev inequalities,
\begin{equation}\label{L-infty-data}
\|\p_{u^\prime}\phi_{q}\|_{L^{\infty}} + \|\p_{\ub^\prime}\phi_{q}\|_{L^{\infty}} + \delta^{\frac{1}{2}} \|\nablaslash \phi_{q}\|_{L^{\infty}} \lesssim \delta M, \quad q\leq N-2.
\end{equation}
The standard energy estimates in the domain $\mathcal{R}^\prime$ yield that
\begin{equation}\label{E-k-induction-data}
E^{2}[\phi_{\leq N}](t^{\prime}) \lesssim \delta^2 I^2_{N}(f_{0},f_{1})+ \int_{1}^{2-\delta} (1+\delta C(M) )  E^2[\phi_{\leq N}](\tau^{\prime})  d\tau^{\prime},
\end{equation}
where $I_{k}(f_{0},f_{1})$ denotes a constant depending only on the initial data up to  $k$ order derivatives and $C(M)$ depends on $M$.
By Gr\"{o}nwall inequality, $E^{2}[\phi_{\leq N}](t^{\prime}) \lesssim \exp(1+ \delta C(M) ) \delta^2 I^2_{N}(f_{0},f_{1})$. If $\delta$ is small enough, we choose $C  I^2_{N}(f_{0},f_{1}) <  M^2$, so that $E^{2}[\phi_{\leq N}](t^{\prime}) \leq C \delta^2 I^2_{N}(f_{0},f_{1}) < \delta^2 M^2$, where $C$ is a universal constant.
Consequently, we close the bootstrap argument and prove the existence of $\phi$ in $\mathcal{R}^\prime$.
\\

\noindent{\Large {\bf Acknowledgements.}} The authors would like to thank Siyuan Ma and Pin Yu for helpful comments and discussions. We also would like to thank the anonymous referees for their pertinent comments and valuable suggestions. J.W. is supported by the NSF of Fujian Province (Grant No. 2018J05010), NSFC (Grant No. 11701482) and the Fundamental Research Funds for the Central Universities (Grant No. 20720170002). C.W. is supported by the NSFC (Grant Nos. 12071435, 11871212) and the Scientific Research Foundation of Zhejiang Sci-Tech University (Grant No. 16062021-Y).



\begin{thebibliography}{amsplain}
\bibitem{Aurilia-christdoulou}  Aurilia A., Christodoulou D.: {\it Theory of strings and membranes in an external field, I: General formulation}. J. Math. Phys. {\bf 20}, 1446-1452 (1979).

\bibitem{Allen}  Allen P., Andersson L., Isenberg J.: {\it Timelike minimal submanifolds of general co-dimension in Minkowski space time}. J. Hyperbolic Differ. Equ. {\bf 3}, 691-700 (2006).

\bibitem{Brendle} Brendle S.: {\it Hypersurfaces in Minkowski space with vanishing mean curvature}. Comm. Pure Appl. Math. {\bf55}, 1249-1279 (2002).

\bibitem{Hoppe3}  Bordemann M., Hoppe J.: {\it The dynamics of relativistic membranes. Reduction to 2-dimensional fluid dynamics}. Phys. Lett. B. {\bf317}, 315-320 (1993).

\bibitem{Hoppe2}  Bordemann M., Hoppe J.: {\it The dynamics of relativistic membranes. II. Nonlinear waves and covariantly reduced membrane equations}. Phys. Lett. B. {\bf 325}, 359-365 (1994).

\bibitem{Christodoulou-null} Christodoulou D.: {\it Global solutions of nonlinear hyperbolic equations for small initial data}. Comm. Pure Appl. Math. {\bf 39}, 267-282 (1986).

\bibitem{Christodoulou-memory} Christodoulou D.: {\it Nonlinear nature of gravitation and gravitational-wave experiments}. Phys. Rev. Lett. {\bf 67}, 1486-1489 (1991).

\bibitem{Christodoulou-2000} Christodoulou D.: {\it The action principle and partial differential equations}. Princeton Univ. Press, Princeton (2000).

\bibitem{Christodoulou1} Christodoulou D.: {\it The formation of shocks in 3-dimensional fluids}. EMS Monographs in Mathematics, Z\"{u}rich (2007).

\bibitem{Christodoulou} Christodoulou D.: {\it The formation of black holes in general relativity}. Monographs in Mathematics, European Mathematical Society (2009).

\bibitem{Christodoulou-K-93} Christodoulou D., Klainerman S.: {\it The global nonlinear stability of the Minkowski space}. Princeton Mathematical Series. {\bf 41}. Princeton University, Princeton, N.J. (1993).

\bibitem{D-K-S-Wong} Donninger R., Krieger J., Szeftel J., Wong W.W.-Y.:
{\it Codimension one stability of the catenoid under the vanishing mean curvature flow in Minkowski space}. Duke Math. J. {\bf 165}, 723--791 (2016).

\bibitem{Hawking-Ellis} Hawking S.W., Ellis G.F.R.: {\it The large scale structure of space-time}. Cambridge monographs on Mathematical Physics. Cambridge University, Cambridge (1973).

\bibitem{Hol-Speck-K-Wong} Holzegel G., Klainerman S., Speck J., Wong W.W.-Y.: {\it Stable shock formation for nearly simple outgoing plane symmetric waves: an overview}. Journal of Hyperbolic Differential Equations. {\bf13}, 1--105 (2016).

    \bibitem{Hoppe5} Hoppe J.: {\it Some classical solutions of relativistic membrane equations in 4 space
time dimensions}.  Phys. Lett. B. {\bf 329}, 10--14 (1994)

\bibitem{Hoppe4} Hoppe J.: {\it Relativistic membranes}. J. Phys. A. {\bf46}, 023001 (2012) .

\bibitem{K-weak-null-18} Keir J.: {\it The weak null condition and global existence using the p-weighted energy method}, 2018, {\href{http://https://arxiv.org/abs/1808.09982}{arXiv.org:1808.09982}}.

\bibitem{Klainerman-null}  Klainerman S.: {\it The null condition and global existence to nonlinear wave equations}.
Lect. Appl. Math. {\bf23}, 293--326 (1986).

\bibitem{Klainerman-Rodnianski} Klainerman S., Rodnianski I.: {\it On the formation of trapped surfaces}. Acta Math. {\bf 208}, 211--333 (2012).

\bibitem{K-L-W} Kong D., Liu K., Wang Y.: {\it Global existence of smooth solutions to two-dimensional compressible isentropic Euler equations for Chaplygin gases}. Sci. China Math. {\bf 53}, 719-738 (2010).

\bibitem{Krieger-Lindblad} Krieger J., Lindblad H.: {\it On stability of the catenoid under vanishing mean curvature flow on Minkowski space}. Dyn. Partial Differ. Equ. {\bf 9}, 89--119 (2012).

\bibitem{Lindblad}
 Lindblad H.: {\it A remark on global existence for small initial data of the minimal surface equation in Minkowskian space time}. Proc. Amer. Math. Soc. {\bf 132}, 1095--1102 (2004).

 \bibitem{Lindblad-Rodnianski1} Lindblad H., Rodnianski I.: {\it Global existence for the Einstein vacuum equations in wave coordinates}. Commun. Math. Phys. {\bf 256}, 43--100 (2005).

\bibitem{Lind-Rod} Lindblad H. Rodnianski I.: {\it The global stability of Minkowski space-time in harmonic gauge}. Ann of Math. (2) {\bf 171}, 1401--1477 (2010).

\bibitem{Lei-Wei} Lei Z., Wei C.: {\it Global radial solutions to 3D relativistic Euler equations for non-isentropic Chaplygin gases}. Math. Ann. {\bf 367}, 1363--1401 (2017).

\bibitem{Majda} Majda, A.: {\it Compressible fluid flow and systems of conservation laws in several space variables}. Springer-Verlag, New York (1984).

\bibitem{Miao-Yu} Miao S., Pei L., Yu P.: {\it On classical goblal solutions of nonlinear wave equations with large data}. To appear in International Mathematics Research Notices.

\bibitem{Miao-Yu1} Miao S., Yu P.: {\it On the formation of shocks for quasilinear wave equations}. Invent. Math. \textbf{207}, 697-831 (2017).

\bibitem{Matt} Visser M., Paris C. M.:
{\it Acoustic geometry for general relativistic barotropic irrotational fluid flow}.
New Journal of Physics. {\bf 12}, 095014 (2010).

\bibitem{Speck-shock} Speck J.: {\it Shock Formation in Small-Data Solutions to 3D Quasilinear Wave Equations}, 2014, {\href{http://https://arxiv.org/abs/:1407.6320v1}{arXiv.org:1407.6320v1}}.

\bibitem{Speck-Hol-Luk-Wong} Speck J., Holzegel G., Luk J., Wong W.W.-Y.: {\it Stable shock formation for nearly simple outgoing plane symmetric waves}. Ann. PDE. {\bf2} (2016). 

\bibitem{Wang-Yu} Wang J., Yu P.: {\it Long time solutions for wave maps with large data}. J. Hyperbolic Differ Equ. {\bf 10}, 371-414 (2013).

\bibitem{Wang-Yu1} Wang J., Yu P.: {\it A large data regime for nonlinear wave equations}. J. Eur. Math. Soc. {\bf18}, 575-622 (2016).


\bibitem{WWong-11} Wong W.W.-Y.:
{\it Regular hyperbolicity, dominant energy condition and causality for Lagrangian theory of maps}. Class. Quantum Grav.  {\bf 28}, 215008 (2011).

\bibitem{Yang-13} Yang S.:
{\it  Global solutions of nonlinear wave equations with large data}.  Selecta Math. (N.S.) {\bf 21}, 1405-1427 (2015).
\end{thebibliography}
\end{document}